\documentclass[12pt,a4paper]{article}
\usepackage{amsmath}
\usepackage{amsfonts}
\usepackage{amssymb}
\usepackage{amsthm}
\usepackage[mathscr]{euscript}
\usepackage{bbm}
\usepackage{bm}
\usepackage{mathrsfs}
\usepackage{graphicx}
\usepackage[left=0.50in,right=0.5in]{geometry}
\usepackage{setspace}
\usepackage{xcolor}
\newcommand{\hl}[1]{\textcolor{blue}{#1}}
\usepackage{authblk}
\usepackage{hyperref}
\usepackage{paralist}

\graphicspath{{./figure/}}

\usepackage[round]{natbib}
\usepackage{multibib}
\newcites{supp}{References}

\newcommand{\ourtitle}{{Model-Agnostic Bounds for Augmented Inverse Probability Weighted Estimators' Wald-Confidence Interval Coverage in Randomized Controlled Trials}}
\title{\ourtitle}

\date{}
\author{Hongxiang Qiu}
\affil{Department of Epidemiology and Biostatistics, College of Human Medicine, Michigan State University}

\allowdisplaybreaks
\newtheorem{theorem}{Theorem}
\newtheorem{lemma}{Lemma}
\newtheorem{corollary}{Corollary}
\newtheorem{proposition}{Proposition}

\theoremstyle{definition}
\newtheorem{condition}{Condition}
\newtheorem{seriescond}{SeriesCond}

\DeclareMathOperator{\expit}{expit}

\newcommand{\real}{{\mathbb{R}}}

\newcommand{\ind}{{\mathbbm{1}}}
\newcommand{\expect}{{E}}
\newcommand{\Var}{{\mathrm{var}}}
\newcommand{\transform}{{\mathcal{T}}}
\newcommand{\funclass}{{\mathcal{F}}}
\newcommand{\gfunclass}{{\mathcal{G}}}

\newcommand{\intd}{{\mathrm{d}}}

\newcommand{\smallo}{{\mathrm{o}}}
\newcommand{\bigO}{{\mathrm{O}}}

\newcommand{\IPW}{{\mathrm{IPW}}}

\newcommand{\const}{{\mathcal{C}}}

\newcommand{\VC}{{\mathrm{VC}}}
\newcommand{\VChull}{{\mathrm{VC-hull}}}
\newcommand{\empro}{{\mathbb{G}}}

\newcommand{\ltwo}{{\ell_2}}

\begin{document}

\maketitle

\begin{abstract}
    Nonparametric estimators, such as the augmented inverse probability weighted (AIPW) estimator, have become increasingly popular in causal inference.
    Numerous nonparametric estimators have been proposed, but they are all asymptotically normal with the same asymptotic variance under similar conditions, leaving little guidance for practitioners to choose an estimator.
    In this paper, I focus on another important perspective of their asymptotic behaviors beyond asymptotic normality, the convergence of the Wald-confidence interval (CI) coverage to the nominal coverage.
    Such results have been established for simpler estimators (e.g., the Berry-Esseen Theorem), but are lacking for nonparametric estimators.
    I consider a simple but practical setting where the AIPW estimator based on a black-box nuisance estimator, with or without cross-fitting, is used to estimate the average treatment effect in randomized controlled trials.
    I derive non-asymptotic Berry-Esseen-type bounds on the difference between Wald-CI coverage and the nominal coverage.
    I also analyze the bias of variance estimators, showing that the cross-fit variance estimator might overestimate while the non-cross-fit variance estimator might underestimate, which might explain why cross-fitting has been empirically observed to improve Wald-CI coverage even if both estimators converge to the same asymptotic normal distribution.
\end{abstract}

\section{Introduction} \label{sec: intro}

Flexible nonparametric estimators constructed based on semiparametric efficiency theory, such as augmented inverse probability weighted (AIPW) estimators \protect\citep{Robins1994,Robins1995}, double/debiased machine learning (DML) estimation \protect\citep{Chernozhukov2017,Chernozhukov2018}, and targeted minimum loss-based estimation (TMLE) \protect\citep{VanderLaan2018}, have gained increasing popularity in causal inference.
Although these nonparametric estimators rely on substantively weaker model assumptions than parametric models and may use flexible nuisance estimators that converge to the truth slowly, they are root-$n$ consistent and asymptotically normal under relatively mild conditions. As such, these nonparametric estimators are similar to but more robust than traditional parametric estimators, particularly in large samples.

Although these estimators have the same asymptotic normal distribution, their performance in moderate samples has been observed to vary in simulations.
For example, cross-fitting, a general technique that is similar to sample splitting and applicable to many nonparametric estimators involving nuisance function estimation, has been empirically found to improve the performance in moderate samples \protect\citep[e.g.,][]{Li2022,Smith2024}.
Currently, the most widely accepted explanation is that cross-fitting removes a Donsker condition \protect\citep{Kosorok2008,VanderVaart2023}, which restricts the flexibility of the nuisance function estimator and is required when not adopting sample-splitting to control an empirical process term.
However, if the Donsker condition is verified, it is theoretically unclear whether cross-fitting improves performance.

Some existing works try to address moderate-sample performance from various other perspectives.
For example, it has been heuristically argued or suggested that plug-in estimators like TMLE are superior to non-plug-in estimators like DML estimators, because TMLE respects known bounds on the estimand \protect\citep[e.g.,][]{Gruber2010,Gruber2014,Levy2018CVTMLE,Tran2019,Rytgaard2021,Kennedy2022,Guo2023}, but \protect\citet{Qiu2024} provided a counterexample.
\protect\citet{Tran2023} and Chapter~28 of \protect\citet{VanderLaan2018} proposed efficient estimation of the asymptotic variance and bootstrap. However, they did not provide theoretical justification for finite-sample improvements, and the improvements in numerical simulations might be debatable.
\protect\citet{Naimi2024,Schader2024} investigated the effect of exogenous randomness in these nonparametric estimators using numerical simulations.
\protect\citet{Robins1997} developed ``curse of dimensionality appropriate (CODA) asymptotic theory'' for semiparametric models, but it is debatable whether their strict exclusion of any smoothness assumption is overly stringent in practice.
\protect\citet{Liu2020,Liu2023} proposed methods to test whether, for the given dataset and nuisance estimators, the nuisance estimators approximately ensure sufficiently small bias for approximate normality of nonparametric estimators and thus approximately valid Wald-confidence interval (CI) coverage.
However, this test focuses on a bias term rather than the CI coverage. 
Moreover, this test is data-dependent and cannot be used to compare estimators under general scenarios.
\protect\citet{Ballinari2024,VanderLaan2024,VanderLaan2024auto} proposed to calibrate nuisance estimators to improve finite-sample performance, and \protect\citet{VanderLaan2024,VanderLaan2024auto} proved weakened sufficient conditions for asymptotic normality due to isotonic calibration.

In this paper, I attempt to compare the moderate-sample performance of AIPW estimators with and without cross-fitting from another perspective on the asymptotic behavior, the convergence of the corresponding Wald-CI coverage to nominal coverage, namely a Berry-Esseen-type bound, using an estimated asymptotic variance.
This may be among the most important theoretical questions regarding these flexible nonparametric procedures following asymptotic normality, because one of the most common use of these nonparametric estimators in practice is the resulting asymptotically valid inference based on the corresponding Wald-CI.
A few existing works provide Berry-Esseen-type bounds for nonparametric estimators.
\protect\citet{Quintas-Martinez2022,Chernozhukov2023} derived Berry-Esseen-type bounds for debiased machine learning. Because the true asymptotic variance was used to standardize the estimator, these results characterize the convergence rate of the estimator's sampling distribution to the normal distribution, but do not directly characterize the CI coverage when the asymptotic variance is estimated.
\protect\citet{Saco2025} characterized the convergence of CI coverage to nominal coverage for debiased machine learning in the special case of partially linear models.
For relatively simpler procedures, Berry-Esseen-type bounds on CI coverage rates (or more precise characterizations of the rates) have been established.
\protect\citet{Hall1992} reviewed these results for the sample mean, some parametric models, and kernel smoothing.
Berry-Eseen-type bounds for U-statistics \protect\citep[e.g.,][]{Callaert2007,Bickel2007}, the sample mean of high-dimensional random vectors \protect\citep[e.g.,][]{Zhilova2020,Zhilova2022}, stationary dependent data \protect\citep[e.g.,][]{Bentkus1997,Jirak2016,Jirak2023}, or design based average treatment effect (ATE) estimation \protect\citep[e.g.,][]{Shi2022}, have also been derived.
However, to the best of my knowledge at the time of writing, no existing literature addresses the convergence rate of CI coverage for a flexible nonparametric ATE estimator with an estimated variance.

I consider a simple but practically relevant example, the estimation of the ATE with an AIPW estimator, with or without cross-fitting, in a randomized controlled trial (RCT).
Generally, the AIPW estimator involves two important nuisance functions, the outcome model and the propensity score.
In contrast to the observational data setting, in the RCT setting, the propensity score is known and need not be estimated or calibrated.
Because of this knowledge, the AIPW estimator behaves fundamentally differently in the RCT setting. It is fully robust against inconsistent estimation of the outcome model and may achieve efficiency gains, making it a potentially desirable estimator in RCTs.
Moreover, there is no issue of the non-existence of root-$n$ consistent estimators when both nuisance functions are unknown in CODA asymptotic theory \protect\citep[Theorem~2 in][]{Robins1997}.
The focus of \protect\citet{Liu2020,Liu2023}, a potentially problematic bias term or its Cauchy-Schwarz upper bound, also equals zero.
Because of the additional substantive complexities in the observational setting, I confine attention to RCTs.

I consider a plug-in estimator of the asymptotic variance when computing a standard error to construct a Wald-CI. This is the most common and widely applicable standard error in the literature.
Efficient asymptotic variance estimators in \protect\citet{Tran2023,VanderLaan2018} only apply to binary outcomes. For general outcomes, these estimators may need additional nuisance estimators, and thus are less computationally convenient and not studied in this paper.
Although this example appears simple and thoroughly studied, I show new and non-trivial findings regarding Wald-CI coverage.

The main contribution of this paper is the following.
\begin{enumerate}
    \item I derive the first non-asymptotic model-agnostic Berry-Esseen-type bounds for AIPW estimators with black-box nuisance estimators in RCTs when the variance is estimated (Theorems~\ref{thm: CV BE bound} and \ref{thm: BE bound} in Section~\ref{sec: additional theory} of the Supplementary Material).
    For better readability, I present the asymptotic versions in Theorems~\ref{thm: asymptotic CV BE bound} and \ref{thm: asymptotic BE bound} in the main text.
    The rate depends on whether cross-fitting is adopted, the convergence behavior of the outcome model estimator, and the complexity of the Donsker class when cross-fitting is not adopted.
    The best-case rate is root-$n$ up to a log factor, which is achieved under sub-Weibull conditions and when the outcome model estimator converges to some function at the root-$n$-rate.
    This rate is comparable to the Berry-Esseen Theorem for the sample mean.
    Although the known propensity score resolves many issues studied in existing literature, my Berry-Esseen-type bounds still suggest potential differences between AIPW estimators with and without cross-fitting in terms of Wald-CI coverage due to different bounds on empirical processes.
    Section~\ref{sec: BE bounds interpretation} contains a discussion of the interpretation of these bounds.
    
    \item I derive non-asymptotic model-agnostic bounds on the difference between (i) the expectation of the plug-in variance estimators, and (ii) an oracle variance of the AIPW estimator based on a finite-sample deterministic approximation to the outcome estimator, with and without cross-fitting (Theorems~\ref{thm: CV expect Var} and \ref{thm: expect Var} in Section~\ref{sec: additional theory} of the Supplementary Material).
    I present asymptotic versions in Theorems~\ref{thm: asymptotic CV expect Var} and \ref{thm: asymptotic expect Var} in the main text.
    Somewhat surprisingly, with a known propensity score and for a reasonably chosen oracle variance, the difference vanishes at a rate that is almost second order, rather than first order, regarding the $L_2$-distance between the outcome model estimator and deterministic approximation.
    That is, the bias of the plug-in variance estimator for the oracle variance vanishes relatively fast in RCTs.
    
    \item I study the bias of the plug-in variance estimators for the oracle variance, with and without cross-fitting, particularly the sign of the bias.
    I show that the plug-in variance estimator with cross-fitting overestimates oracle variance in large samples if the outcome model converges slowly (Part~\ref{prop: asymptotic Var bias 1} of Proposition~\ref{prop: asymptotic Var bias}), contributing to increased CI coverage.
    In Section~\ref{sec: Var bias}, I also provide a heuristic argument showing that plug-in variance estimator without cross-fitting might underestimate the oracle variance, contributing to decreased CI coverage.
    
    \item In Section~\ref{sec: sim}, I conduct a simulation to investigate my theoretical results' implications in moderate samples.
    The simulation results align with some of my theoretical results well, but also highlight some theoretical gaps for future research.
\end{enumerate}

With the above results, I provide new theoretical justification for adopting cross-fitting when using AIPW estimators in RCTs, even if Donsker conditions hold.

\section{Preliminaries} \label{sec: prelim}

Suppose that a prototypical full data point consists of a vector of baseline covariates $X \in \mathcal{X}$, a binary treatment $A \in \{0,1\}$, and real-valued counterfactual outcomes $Y(a)$ corresponding to treatment status $a \in \{0,1\}$ \protect\citep{Neyman1923,Rubin1974}.
A prototypical observed data point consists of $V:=(X,A,Y)$ where $Y=Y(A)$.
Suppose that the data consists of $n$ independent and identically distributed (i.i.d.) copies $(V_i)_{i=1}^n$ of $V$ drawn from a true distribution $P_*$.
I consider an RCT setting where treatment $A$ is randomized according to $\Pr(A=a \mid X=x) = \pi_*(a \mid x)$ for a known propensity score function $\pi_*$.
The most common choice of $\pi_*$ is a constant function.

Let $Q_{*,a}: x \mapsto \expect[Y \mid X=x,A=a]$ be the true outcome model and $\psi_{*,a} := \expect[Y(a)]$ be the mean counterfactual outcome corresponding to treatment $a$, which equals $\expect[Q_{*,a}(X)]$ \protect\citep{Robins1986}.
These estimands are building blocks for common treatment effect estimates, including risk difference $\psi_{*,1}-\psi_{*,0}$, risk ratio $\psi_{*,1}/\psi_{*,0}$ (when $Y \geq 0$), and odds ratio $\{\psi_{*,1}/(1-\psi_{*,1})\}/\{\psi_{*,0}/(1-\psi_{*,0})\}$ (when $Y \in [0,1]$).
In the rest of this paper, I mainly focus on the results for $\psi_{*,a}$ and the additive ATE $\psi_* := \psi_{*,1}-\psi_{*,0}$.

I will adopt empirical process notations \protect\citep{VanderVaart2023}: For any distribution $P$ and $P$-integrable function $f$, let $P f := \int f \intd P$. The empirical distribution is denoted by $P_n$. If the data is split into $K$ (disjoint) folds, I use $I_k$ to denote the index of the data points in fold $k$ so that $\sqcup_{k=1}^K I_k=\{1,\ldots,n\}$, and use $P_{n,k}$ to denote the empirical distribution in fold $k$.
For any probability measure $P$ and $q \geq 1$, let $\| \cdot \|_{q,P}$ denote the $L_q(P)$-norm of a function, that is, $\| f \|_{q,P} := (P |f|^q)^{1/q}$, and let $\| \cdot \|_{q} := \| \cdot \|_{q,P_*}$.

All asymptotics refer to the scenario where the sample size $n \to \infty$ and all other aspects of the data-generating mechanism are fixed. I may use $n \to \infty$ to denote this asymptotic scenario for short.
I adopt big-O and little-o notations in asymptotic scenarios, and use $\Theta$ to denote the same order; that is, for any real sequences $a_n$ and $b_n$, $a_n = \Theta(b_n)$ if $a_n=\bigO(b_n)$ and $b_n=\bigO(a_n)$.
I write $a_n=\Theta_+(b_n)$ if $a_n=\Theta(b_n)$ and $a_n \geq 0$.
If cross-fitting is adopted, the number of folds $K$ may vary with $n$.

An estimator $\hat{\theta}$ of $\theta$ is said to be asymptotically linear with influence function $f$ if $P_* f=0$, $P_* f^2 < \infty$, and $\hat{\theta} = \theta + P_n f + \smallo_p(n^{-1/2})$.
For any $\alpha \in (0,1)$, let $z_\alpha$ denote the $(1-\alpha)$-quantile of standard Gaussian. Let $\phi$ denote the standard Gaussian density.
I use $\const$ to denote an absolute positive constant that may vary line by line.
For a random function $\hat{f}$, I use $\expect_{\hat{f}}$ to denote the expectation over $\hat{f}$ only.
I use $\ind$ to denote the indicator function.

A non-cross-fit AIPW estimator $\tilde{\psi}_a$ \protect\citep{Robins1994,Robins1995} of $\psi_{*,a}$ is constructed as $P_n \transform_a(\tilde{Q}_a)$, where, for any function $Q: \mathcal{X} \to \real$ and $a' \in \{0,1\}$, I define transformation $\transform_{a'}(Q): v=(x,a,y) \mapsto \ind(a=a') \{y-Q(x)\}/\pi_*(a' \mid x) + Q(x)$, and $\tilde{Q}_a$ is a (possibly black-box) flexible estimator of $Q_{*,a}$.
Under somewhat mild conditions below, $n^{1/2} (\tilde{\psi}_a-\psi_{*,a})$ converges to a normal distribution. Usually, a consistent estimator of the asymptotic variance is computed based on the estimated influence function, namely $\tilde{\sigma}_a^2 := P_n \{\transform_a(\tilde{Q}_a)-\tilde{\psi}_a\}^2$, so that $\tilde{\sigma}_a/\sqrt{n}$ is asymptotically valid standard error to construct a Wald-CI. For example, $\tilde{\psi}_a \pm n^{-1/2} z_{\alpha/2} \tilde{\sigma}_a$ is the $(1-\alpha)$-level symmetric two-sided Wald-CI.
Similarly, a non-cross-fit AIPW estimator of the ATE $\psi_*$ is $\tilde{\psi} := \tilde{\psi}_1-\tilde{\psi}_0$, with the correponding asymptotic variance estimator $\tilde{\sigma}^2 := P_n \{\transform_1(\tilde{Q}_1) - \transform_0(\tilde{Q}_0) - \tilde{\psi}\}^2$.

The cross-fit AIPW estimator, also known as the double/debiased machine learning estimator \protect\citep{Chernozhukov2017,Chernozhukov2018}, is constructed as follows. First, split the data into $K$ folds of approximately equal sizes. Then, for each fold $k$ ($k=1,\ldots,K$), estimate $Q_{*,a}$ with a flexible estimator $\hat{Q}_{k,a}$ based on data outside the fold $k$, and compute a sample-split estimator $\hat{\psi}_{k,a} := P_{n,k} \transform_a(\hat{Q}_{k,a})$. Finally, average over all the folds to compute the cross-fit estimator $\hat{\psi}_a := n^{-1} \sum_{k=1}^K |I_k| \hat{\psi}_{k,a}$.
A consistent estimator of the asymptotic variance is also often constructed similarly: $\hat{\sigma}_a^2 := n^{-1} \sum_{k=1}^K |I_k| \hat{\sigma}_{k,a}^2$ where $\hat{\sigma}_{k,a}^2 := P_{n,k} \{\transform_a(\hat{Q}_{k,a})-\hat{\psi}_{k,a}\}^2$.
A Wald-CI can be constructed similarly to the non-cross-fit version.
For the ATE $\psi_*$, a cross-fit estimator and its asymptotic variance estimator can be constructed similarly: $\hat{\psi} := n^{-1} \sum_{k=1}^K |I_k| \hat{\psi}_k$ where $\hat{\psi}_k := \hat{\psi}_{k,1}-\hat{\psi}_{k,0}$, and $\hat{\sigma}^2 := n^{-1} \sum_{k=1}^K |I_k| P_{n,k} \hat{\sigma}_k^2$ where $\hat{\sigma}_k^2 := P_{n,k} \{\transform_1(\hat{Q}_{k,1}) - \transform_0(\hat{Q}_{k,0}) - \hat{\psi}_k\}^2$ .
In most notations throughout this paper, a tilde $\tilde{\phantom{a}}$ corresponds to the non-cross-fit estimator and a hat $\hat{\phantom{a}}$ corresponds to the cross-fit estimator.

It is well known that, when the propensity score $\pi_*$ is known, $\tilde{\psi}_a$ and $\hat{\psi}_a$ have the following more robust asymptotic properties than when $\pi_*$ is unknown \protect\citep[e.g.,][]{Moore2009,Rosenblum2009,VanderLaan2018}. For any function $Q: \mathcal{X} \to \real$, any scalar $\psi$, and any $a' \in \{0,1\}$, let $D_{a'}(Q,\psi): v \mapsto \transform_{a'}(Q)(v) - \psi$. The following asymptotic properties hold.
\begin{itemize}
    \item If $\tilde{Q}_a$ ($\hat{Q}_{k,a}$ resp.) converges in probability to some fixed function $Q_{\infty,a}$ in the $L_2(P_*)$ sense, $P_* D_a(Q_{\infty,a},\psi_{*,a})^2<\infty$, and $D_a(\tilde{Q}_a,\tilde{\psi}_a) - D_a(Q_{\infty,a},\psi_{*,a})$ lies in a fixed $P_*$-Donsker class with probability tending to one for the non-cross-fit estimator $\tilde{\psi}_a$, then $\tilde{\psi}_a$ ($\hat{\psi}_a$ resp.) is a regular and asymptotically linear estimator of $\psi_{*,a}$ with influence function $D_a(Q_{\infty,a},\psi_{*,a})$, and hence is asymptotically normal.
    \item Moreover, if $Q_{\infty,a} = Q_{*,a}$, then $\tilde{\psi}_a$ ($\hat{\psi}_a$ resp.) is asymptotically efficient.
\end{itemize}

Motivated by such robustness against arbitrary bias in $\tilde{Q}_a$ and $\hat{Q}_{k,a}$, for each $a \in \{0,1\}$, I consider a fixed function $Q_{\#,a}: \mathcal{X} \to \real$ in $L_2(P_*)$ that may depend on the sample size $n$ and may differ from $Q_{*,a}$ and $Q_{\infty,a}$.
Since I analyze estimators with and without cross-fitting separately, $Q_{\#,a}$ may differ in the two cases.
The theoretical results below hold for any fixed $Q_{\#,a}$ sufficiently close to $\tilde{Q}_a$ and $\hat{Q}_{k,a}$, but the bounds are the tightest when $Q_{\#,a}$ is a good non-asymptotic approximation to $\tilde{Q}_a$ or $\hat{Q}_{k,a}$.
In particular, $Q_{\#,a}$ may be taken as $x \mapsto \expect[\tilde{Q}_a(x)]$ for the non-cross-fit estimator, and $x \mapsto \expect[\hat{Q}_{k,a}(x)]$ for the cross-fit estimator if $\hat{Q}_{k,a}$ ($k=1,\ldots,K$) has a common mean.
This setup enables a sample-size-dependent deterministic approximation to the outcome model estimator, highlighting the non-asymptotic nature of the bounds below.
As shown in the simulation in Section~\ref{sec: sim}, a sample-size-dependent approximation can be much more accurate than the asymptotic approximation.
I assume that $\expect \| \tilde{Q}_a - Q_{\#,a} \|_{2}^2$ or $\expect \| \hat{Q}_{k,a} - Q_{\#,a} \|_{2}^2$ ($k=1,\ldots,K$), depending on the estimator being considered, is non-zero, namely the outcome model estimator is non-degenerate.
Otherwise, the AIPW estimators reduce to trivial sample means of i.i.d. random variables, and results on Wald-CI coverage have been established with Edgeworth expansion \protect\citep{Hall1992,Bhattacharya2007}.
In the asymptotic scenario $n \to \infty$, I assume that $\expect \| \tilde{Q}_a - Q_{\#,a} \|_{2}^2$ or $\expect \| \hat{Q}_{k,a} - Q_{\#,a} \|_{2}^2$ ($k=1,\ldots,K$) converges to zero.

Let $\sigma_{\#,a}^2 := P_* D_a(Q_{\#,a},\psi_{*,a})^2$ denote the finite-sample oracle approximation to the scaled variance of the AIPW estimator based on $Q_{\#,a}$, and $\sigma_{\dagger,a}^2$ denote the expectation of asymptotic variance estimators, depending on whether cross-fitting is adopted:
$$\sigma_{\dagger,a}^2 := \begin{cases}
    \expect[\tilde{\sigma}_a^2] & \text{for non-cross-fit AIPW estimator} \\
    \expect[\hat{\sigma}_a^2] & \text{for cross-fit AIPW estimator}
\end{cases}$$
The finite-sample oracle approximation $\sigma_\#^2$ and the expectation $\sigma_\dagger^2$ for the asymptotic variance estimators $\tilde{\sigma}^2$ or $\hat{\sigma}^2$ when estimating ATE are defined similarly.
Since each AIPW estimator is studied separately, this definition should lead to no confusion.
Throughout this paper, I assume the following mild moment and boundedness conditions.

\begin{condition}[Nonzero variance] \label{cond: 2nd moment}
    $\sigma_{*,a}^2 := P_* D_a(Q_{*,a},\psi_{*,a})^2 > 0$.
\end{condition}

\begin{condition}[Finite third moment] \label{cond: 3rd moment}
    $\rho_{\#,a} := P_* |D_a(Q_{\#,a},\psi_{*,a})|^3 < \infty$.
\end{condition}

\begin{condition}[Treatment positivity] \label{cond: positivity}
    There exists a constant $\tau_\pi \in (0,0.5]$ such that $\pi_* \in [\tau_\pi,1-\tau_\pi]$.
\end{condition}

\begin{condition}[Bounded outcome model] \label{cond: bounded Q}
    There exists a constant $M \in (0,\infty)$ such that $Q_{*,a}$ and $Q_{\#,a}$ are bounded by $M$, and, with probability one, $\tilde{Q}_a$ and $\hat{Q}_{k,a}$ are bounded by $M$ for all $k \in \{1,\ldots,K\}$ and $a \in \{0,1\}$.
\end{condition}

\begin{condition}[Higher moments] \label{cond: higher moments}
    $\varsigma_{\#,a}^2 :=P_* \{\transform_a(Q_{\#,a})^2 - P_* \transform_a(Q_{\#,a})^2\}^2 > 0$. Moreover, there exist constants $0 < \underline{m} \leq \bar{m} < \infty$ such that, with probability one, $P_* | \transform_a(\hat{Q}_{k,a})^2 - P_* \transform_a(\hat{Q}_{k,a})^2 |^3 \leq \bar{m}$ and $P_* \{\transform_a(\hat{Q}_{k,a})^2 - P_* \transform_a(\hat{Q}_{k,a})^2\}^2 \in [\underline{m}, \bar{m}]$.
\end{condition}

The positivity in Conditions~\ref{cond: 2nd moment} and \ref{cond: higher moments} holds if $\Pr(\Var(Y \mid X,A=a)>0)>0$.
Under Condition~\ref{cond: higher moments}, with probability one, $P_* D_a(\tilde{Q}_a,\psi_{*,a})^2$ and $P_* D_a(\hat{Q}_{k,a},\psi_{*,a})^2$ ($k=1,\ldots,K$) are bounded, say by $\bar{\sigma}_a^2$; $P_* |D_a(\tilde{Q}_a,\psi_{*,a})|^3$ and $P_* |D_a(\hat{Q}_{k,a},\psi_{*,a})|^3$ ($k=1,\ldots,K$) are bounded, say by $\bar{\rho}_a$.
The boundedness in Conditions~\ref{cond: 3rd moment}, \ref{cond: bounded Q} and \ref{cond: higher moments} holds (or is anticipated to hold) if the outcome $Y$ is bounded.
Condition~\ref{cond: positivity} holds in almost all RCTs.

\section{Berry-Esseen-type bounds} \label{sec: BE bounds}

\subsection{Cross-fit procedure} \label{sec: CV BE bounds}

To simplify the presentation of the results in the main text, I consider the following assumption on sample splitting, whose second part holds if all outcome model estimators $\hat{Q}_{k,a}$ ($k=1,\ldots,K$) are based on the same estimation algorithm in addition to the first part.
\begin{condition}[Exchangeable sample splitting] \label{cond: CV exchangeable}
    All folds have equal sizes $|I_k|=n/K$ and all outcome model estimators $\hat{Q}_{k,a}$ ($k=1,\ldots,K$) are identically distributed.
\end{condition}

The next approximate sub-Weibull conditions are unnecessary for obtaining a Berry-Esseen-type bound, but it may be reasonable and leads to tighter bounds.

\begin{condition}[Approximate sub-Weibull tail of $\hat{Q}_{k,a}$] \label{cond: CV tail}
    For all $k \in \{1,\ldots,K\}$ and $a \in \{0,1\}$, there exist constants $\hat{a}_1,\hat{b}_1,\hat{d}_1 \geq 0, \hat{c}_1, \hat{q}_1 > 0$ and a deterministic positive function $\hat{r}_1$ such that, for all $t>0$,
    $$\Pr \left( \frac{\| \hat{Q}_{k,a} - Q_{\#,a} \|_{2}}{\hat{r}_1(n)} > t \right) \leq \hat{a}_1 n^{\hat{d}_1} \exp \left( -\hat{c}_1 t^{\hat{q}_1} \right) + \hat{b}_1 (\log n / n)^{1/2}.$$
\end{condition}

Condition~\ref{cond: CV tail} requires that the $L_2(P_*)$-distance between $\hat{Q}_{k,a}$ and $Q_{\#,a}$, normalized by a rate $\hat{r}_1(n)$, is roughly sub-Weibull, with a possible multiplicative factor that grows at a polynomial rate at most and a possible relaxation term that vanishes faster than $n^{-1/2}$.
The rate $\hat{r}_1(n)$ can be (an upper bound of) the convergence rate of $\| \hat{Q}_{k,a} - Q_{\#,a} \|_2$.
The $(\log n / n)^{1/2}$ term allows for any rare event whose probability decays with $n$ rapidly, for example, the event that all individuals are in the other treatment arm so that there is no data to obtain $\hat{Q}_{k,a}$.
If one wishes to control the tail of $\| \hat{Q}_{k,a} - Q_{\#,a} \|_{2}$ via asymptotic normality, this $(\log n / n)^{1/2}$ term can also handle normal approximation errors to be controlled by a Berry-Esseen-type bound.
I illustrate a concrete example of an outcome model estimator satisfying this sub-Weibull tail bound, series regression, in Section~\ref{sec: light tail example} of the Supplementary Material.
For empirical risk minimizers, such approximate sub-Weibull tail bounds might be obtained via empirical process tail bounds \protect\citep[e.g., Theorems~2.14.36, 2.14.37, 2.15.5, 2.15.11 in][]{VanderVaart2023}.

\begin{condition}[Approximate sub-Weibull tail of $P_* \transform_a(\hat{Q}_{k,a})^2$] \label{cond: more CV tail}
    There exist constants $\hat{a}_2,\hat{b}_2,\hat{d}_2 \geq 0, \hat{c}_2, \hat{q}_2 > 0$ and a deterministic positive function $\hat{r}_2$ such that, for all $k \in \{1,\ldots,K\}$, $a \in \{0,1\}$, and $t>0$,
    $$\Pr \left( \frac{| P_* \transform_a(\hat{Q}_{k,a})^2 - \expect[P_* \transform_a(\hat{Q}_{k,a})^2] |}{\hat{r}_2(n)} > t \right) \leq \hat{a}_2 n^{\hat{d}_2} \exp \left( -\hat{c}_2 t^{\hat{q}_2} \right) + \hat{b}_2 (\log n / n)^{1/2}.$$
\end{condition}

As shown in Lemma~\ref{lemma: transform sqaure} in the Supplementary Material, because of the integration with respect to $P_*$ and the boundedness of $\hat{Q}_{k,a}$ in Condition~\ref{cond: bounded Q}, the random variable in the numerator has a tail similar to the $L_2(P_*)$-distance between $\hat{Q}_{k,a}$ and its mean.
Thus, if $Q_{\#,a}$ is the mean of $\hat{Q}_{k,a}$ and Condition~\ref{cond: CV tail} holds, Condition~\ref{cond: more CV tail} may hold with $\hat{r}_2=\hat{r}_1$.

For readability, I next present an asymptotic Berry-Esseen-type bound for the cross-fit Wald-CI in the scenario $n \to \infty$, with $P_*$ and all constants in Conditions~\ref{cond: 2nd moment}--\ref{cond: higher moments}, \ref{cond: CV tail}, and \ref{cond: more CV tail} fixed.
A more intricate non-asymptotic version is presented in Theorem~\ref{thm: CV BE bound} in Section~\ref{sec: additional theory} of the Supplementary Material.

\begin{theorem}[Asymptotic Berry-Esseen-type bound for cross-fit AIPW Wald-CI] \label{thm: asymptotic CV BE bound}
    For the mean counterfactual outcome $\psi_{*,a}$, under Conditions~\ref{cond: 2nd moment}--\ref{cond: CV exchangeable}, as $n \to \infty$,
    \begin{align}
        &\left| \Pr(n^{1/2} (\hat{\psi}_a - \psi_{*,a}) \leq z_\alpha \hat{\sigma}_a) - (1-\alpha) - \phi(z_\alpha) z_\alpha \frac{\sigma_{\dagger,a}-\sigma_{\#,a}}{\sigma_{\#,a}} \right| \nonumber \\
        &= \bigO \left( K^{2/3} \left\{ \expect \|\hat{Q}_{k,a} - Q_{\#,a} \|_2^2 \right\}^{1/3} + \left\{ \frac{K \log n}{n} \right\}^{1/2} \right). \label{eq: asymptotic CV BE bound}
    \end{align}
    Additionally under approximate sub-Weibull conditions \ref{cond: CV tail} and \ref{cond: more CV tail}, \eqref{eq: asymptotic CV BE bound} holds with its right-hand side replaced by
    \begin{equation}
        \bigO \left( K^{1/2} \hat{r}_1(n) (\log n)^{(\hat{q}_1+2)/(2 \hat{q}_1)} + \hat{r}_2(n) (\log n)^{1/\hat{q}_2} + K \{\log n/n\}^{1/2} \right). \nonumber \label{eq: improved asymptotic CV BE bound}
    \end{equation}
    For the ATE $\psi_*$, the same statements hold with subscript $a$ indicating the treatment arm dropped from the left-hand side of \eqref{eq: asymptotic CV BE bound} and $\{ \expect \|\hat{Q}_{k,a} - Q_{\#,a} \|_2^2 \}^{1/3}$ replaced by $\sum_{a=0}^1 \{ \expect \|\hat{Q}_{k,a} - Q_{\#,a} \|_2^2 \}^{1/3}$.
\end{theorem}

The bound for a $(1-\alpha)$-level two-sided Wald-CI can be derived by noting that $\Pr(-z_{\alpha/2} \hat{\sigma}_a \leq n^{1/2} (\hat{\psi}_a - \psi_{*,a}) \leq z_{\alpha/2} \hat{\sigma}_a) = \Pr(n^{1/2} (\hat{\psi}_a - \psi_{*,a}) \leq z_{\alpha/2} \hat{\sigma}_a) - \Pr(n^{1/2} (\hat{\psi}_a - \psi_{*,a}) < -z_{\alpha/2} \hat{\sigma}_a)$.
The proof of this theorem can be found in Section~\ref{sec: CV proof} of the Supplementary Material.
It is based on a more general version of the delta method for CI coverage \protect\citep[Section~2.7 in][]{Hall1992}, a Taylor's expansion to analyze the effect of the bias $\sigma_{\dagger,a}-\sigma_{\#,a}$ of the variance estimator, non-asymptotic bounds on the deviation of $n^{1/2} (\hat{\psi}_a-\psi_{*,a})$ from its linear approximation $n^{1/2} (P_n - P_*) D_a(Q_{\#,a},\psi_{*,a})$, the Berry-Esseen Theorem to bound the deviation of the linear approximation $n^{1/2} (P_n - P_*) D_a(Q_{\#,a},\psi_{*,a})$ from the asymptotic normal distribution, and a non-asymptotic tail bound of $\hat{\sigma}_a - \sigma_{\dagger,a}$.

Theorem~\ref{thm: asymptotic CV BE bound} is not very useful without knowing the convergence rate of $\sigma_{\dagger,a} - \sigma_{\#,a}$.
It is straightforward to show that the difference is upper-bounded by the convergence rate of $\|\hat{Q}_{k,a} - Q_{\#,a} \|_2$.
However, in the following Theorem~\ref{thm: asymptotic CV expect Var}, I show a somewhat surprising result that, under the following fairly mild condition on the choice of $Q_{\#,a}$, $|\sigma_{\dagger,a} - \sigma_{\#,a}|$ is roughly upper-bounded by the square of this rate.
\begin{condition}[Consistency or common mean for $\hat{Q}_{k,a}$] \label{cond: CV common mean Q or consistent Q}
    One of the following conditions holds:
    \begin{enumerate}
        \item $Q_{\#,a}=Q_{*,a}$; \label{cond: CV common mean Q}
        \item For all $k = 1,\ldots,K$, $Q_{\#,a}(X) = \expect_{\hat{Q}_{k,a}}[\hat{Q}_{k,a}(X)]$ $P_*$-almost surely. \label{cond: CV consistent Q}
    \end{enumerate}
\end{condition}
Part~\ref{cond: CV common mean Q} of Condition~\ref{cond: CV common mean Q or consistent Q} holds if $\hat{Q}_{k,a}$ is consistent for the true outcome model $Q_{*,a}$.
Part~\ref{cond: CV consistent Q} of Condition~\ref{cond: CV common mean Q or consistent Q} holds if Condition~\ref{cond: CV exchangeable} holds and $Q_{\#,a}$ is taken to be the common mean of all outcome model estimators $\hat{Q}_{k,a}$, regardless of whether the outcome model estimator $\hat{Q}_{k,a}$ converges to the truth $Q_{*,a}$.

\begin{theorem}[Convergence rates of $|\sigma_{\dagger,a} - \sigma_{\#,a}|$ and $|\sigma_{\dagger} - \sigma_{\#}|$ with cross-fitting] \label{thm: asymptotic CV expect Var}
    Under Conditions~\ref{cond: 2nd moment} and \ref{cond: positivity},
    as $n \to \infty$, $|\sigma_{\dagger,a} - \sigma_{\#,a}| = \bigO( \{\expect \|\hat{Q}_{k,a} - Q_{\#,a} \|_2^2 \}^{1/2} + n^{-1})$ and $|\sigma_{\dagger} - \sigma_{\#}| = \bigO( \{\expect \|\hat{Q}_{k,1} - Q_{\#,1} \|_2^2 \}^{1/2} + \{\expect \|\hat{Q}_{k,0} - Q_{\#,0} \|_2^2 \}^{1/2} + n^{-1})$.
    Additionally under Condition~\ref{cond: CV common mean Q or consistent Q}, $|\sigma_{\dagger,a} - \sigma_{\#,a}| = \bigO(\expect \|\hat{Q}_{k,a} - Q_{\#,a} \|_2^2 + n^{-1})$ and $|\sigma_{\dagger} - \sigma_{\#}| = \bigO(\expect \|\hat{Q}_{k,1} - Q_{\#,1} \|_2^2 + \expect \|\hat{Q}_{k,0} - Q_{\#,0} \|_2^2 + n^{-1})$.
\end{theorem}

A non-asymptotic version is presented in Theorem~\ref{thm: CV expect Var} in Section~\ref{sec: additional theory} of the Supplementary Material.
The term involving $\expect \|\hat{Q}_{k,a} - Q_{\#,a} \|_2^2$ arises from explicitly calculating or bounding the variance of $\transform_a(\hat{Q}_{k,a})$, which contributes to $\sigma_{\dagger,a}^2-\sigma_{\#,a}^2$.
Under Condition~\ref{cond: CV common mean Q or consistent Q}, the rate is faster because, when the propensity score $\pi_*$ is known, $Q_{\#,a}$ is a projection of $\hat{Q}_{k,a}$ or $Y$, so that the linear contribution of $\hat{Q}_{k,a}-Q_{\#,a}$ to $\sigma_{\dagger,a}^2-\sigma_{\#,a}^2$ vanishes, and thus $\sigma_{\dagger,a}^2-\sigma_{\#,a}^2$ is almost quadratic.
A formal proof can be found in Section~\ref{sec: CV proof} of the Supplementary Material.
This theorem implies that, regardless of whether Condition~\ref{cond: CV common mean Q or consistent Q} holds, the term involving $\sigma_{\dagger,a} - \sigma_{\#,a}$ in \eqref{eq: asymptotic CV BE bound} is not a leading term in the Berry-Esseen-type bound in Theorem~\ref{thm: asymptotic CV BE bound}.
I defer further interpretation of Theorem~\ref{thm: asymptotic CV BE bound} to Section~\ref{sec: BE bounds interpretation}.

\subsection{Non-cross-fit procedure} \label{sec: non CV BE bounds}

A Donsker condition is needed for $\tilde{\psi}_a$ to be asymptotically normal, but slightly stronger conditions appear necessary to obtain more insights like convergence rates into CI coverage than a mere convergence of CI coverage to the nominal coverage.
For concreteness and illustration, I consider the following sufficient condition of a Donsker condition based on covering number and uniform entropy integral of a uniformly bounded function class \protect\citep{VanderVaart2023}.
For any function class $\gfunclass$, let $N(\epsilon,\gfunclass,\| \cdot \|)$ denote the $\epsilon$-covering number of $\gfunclass$ with respect to a pseudo-norm $\| \cdot \|$ (the minimal number of balls with radius $\epsilon$ that altogether cover $\gfunclass$), and
$$J(\delta,\gfunclass,G) := \int_0^\delta \sup_\mu \left\{1+\log N \big(\epsilon \| G \|_{2,\mu} ,\gfunclass,L_2(\mu) \big) \right\}^{1/2} \intd \epsilon$$
denote its uniform entropy integral, where $G$ is an envelope function of $\gfunclass$ (i.e., $|g| \leq G$ pointwise for all $g \in \gfunclass$) and the supremum is taken over all finitely discrete probability measures $\mu$.

\begin{condition}[Sufficient conditions for Donsker condition] \label{cond: donsker}
    The outcome $Y$ is bounded by the constant $M$ introduced in Condition~\ref{cond: bounded Q}.
    There exists a uniformly bounded function class $\funclass$ (i.e., with a constant envelope $M$) that consists of some real-valued functions defined on $\mathcal{X}$ and may depend on the treatment $a \in \{0,1\}$ under consideration, such that $\tilde{Q}_a \in \funclass$ with probability one and $J(1,\funclass,M) < \infty$.
\end{condition}

In the asymptotic scenario, I assume that $\funclass$ is fixed and does not depend on $n$.
I also consider the following approximate sub-Weibull conditions similar to Conditions~\ref{cond: CV tail} and \ref{cond: more CV tail}.

\begin{condition}[Approximate sub-Weibull tail of $\tilde{Q}_a$] \label{cond: tail}
    There exist constants $\tilde{a}_1,\tilde{b}_1,\tilde{d}_1 \geq 0, \tilde{c}_1, \tilde{q}_1 > 0$ and a deterministic positive function $\tilde{r}_1$ such that, for all $t>0$,
    $$\Pr \left( \frac{\| \tilde{Q}_a - Q_{\#,a} \|_{2}}{\tilde{r}_1(n)} > t \right) \leq \tilde{a}_1 n^{\tilde{d}_1} \exp \left( -\tilde{c}_1 t^{\tilde{q}_1} \right) + \tilde{b}_1 n^{-1/2} \log n.$$
\end{condition}

\begin{condition}[Approximate sub-Weibull tail of $P_* \transform_a(\tilde{Q}_a)^2$] \label{cond: more tail}
    There exist constants $\tilde{a}_2,\tilde{b}_2,\tilde{d}_2 \geq 0, \tilde{c}_2, \tilde{q}_2 > 0$ and a deterministic positive function $\tilde{r}_2$ such that, for all $a \in \{0,1\}$ and $t>0$,
    $$\Pr \left( \frac{| P_* \transform_a(\tilde{Q}_a)^2 - \expect[P_* \transform_a(\tilde{Q}_a)^2] |}{\tilde{r}_2(n)} > t \right) \leq \tilde{a}_2 n^{\tilde{d}_2} \exp \left( -\tilde{c}_2 t^{\tilde{q}_2} \right) + \tilde{b}_2 n^{-1/2} \log n.$$
\end{condition}

For readability, similarly to Theorem~\ref{thm: asymptotic CV BE bound}, I next present an asymptotic Berry-Esseen-type bound for non-cross-fit Wald-CI under the same asymptotic scenario.
A non-asymptotic version is presented in Theorem~\ref{thm: BE bound} in Section~\ref{sec: additional theory} of the Supplementary Material.

\begin{theorem}[Asymptotic Berry-Esseen-type bound for non-cross-fit AIPW Wald-CI] \label{thm: asymptotic BE bound}
    For the mean counterfactual outcome $\psi_{*,a}$, under Conditions~\ref{cond: 2nd moment}--\ref{cond: higher moments} and \ref{cond: donsker}, for any $\delta > 0$ that may depend on $n$,
    \begin{align}
        &\left| \Pr(n^{1/2} (\tilde{\psi}_a - \psi_{*,a}) \leq z_\alpha \tilde{\sigma}_a) - (1-\alpha) - \phi(z_\alpha) z_\alpha \frac{\sigma_{\dagger,a}-\sigma_{\#,a}}{\sigma_{\#,a}} \right| \nonumber \\
        &= \bigO \left( \tilde{\Delta}(\delta) + J(2 \delta,\funclass,M) + \frac{J^2(2 \delta,\funclass,M)}{\delta^2 n^{1/2}} + (\delta + n^{-1/2}) \log n \right), \label{eq: asymptotic BE bound}
    \end{align}
    where $\tilde{\Delta}(\delta) := \expect \| \tilde{Q}_a-Q_{\#,a} \|_2^2/\delta^2 + (\expect \| \tilde{Q}_a-Q_{\#,a} \|_2^2)^{1/3}$.
    Additionally under approximate sub-Weibull conditions \ref{cond: tail} and \ref{cond: more tail}, \eqref{eq: asymptotic BE bound} holds with $\tilde{\Delta}(\delta)$ replaced by
    $$\tilde{\Delta}'(\delta) := n^{\tilde{d}_1} \exp \left\{ - \tilde{c}_1 \left( \frac{2 \delta M}{\tilde{r}_1(n)} \right)^{\tilde{q}_1} \right\} + \tilde{r}_2(n) (\log n)^{1/\tilde{q}_2}.$$
    For the ATE $\psi_*$, the same statements hold with subscript $a$ indicating the treatment arm dropped from the left-hand side of \eqref{eq: asymptotic BE bound} and $\tilde{\Delta}(\delta)$ replaced by $\sum_{a=0}^1 \{ \expect \| \tilde{Q}_a-Q_{\#,a} \|_2^2/\delta^2 + (\expect \| \tilde{Q}_a-Q_{\#,a} \|_2^2)^{1/3} \}$.
\end{theorem}

The rates in Theorem~\ref{thm: asymptotic BE bound} might not be intuitive, so I will show the rates for VC-subgraph and VC-hull classes in Corollary~\ref{coro: VC & hull rates} later in this section.
Since $J(2\delta, \funclass, M) \geq 2 \delta$, the rate's upper bound in the right-hand side of \eqref{eq: asymptotic BE bound} is lower bounded by $\bigO( \{\expect \| \tilde{Q}_a-Q_{\#,a} \|_2^2\}^{1/3} + n^{-1/2})$ (up to log factors) without approximate sub-Weibull conditions, and lower bounded by $\bigO(\{\expect \| \tilde{Q}_a - Q_{\#,a} \|_{2}^2\}^{1/2} + n^{-1/2})$ (up to log factors) with approximate sub-Weibull conditions, which are comparable to Theorem~\ref{thm: asymptotic CV BE bound} with fixed $K$.
Theorem~\ref{thm: BE bound} can be shown with an argument similar to Theorem~\ref{thm: CV BE bound}, except that the additional empirical process terms are bounded by Theorems~5.1 and 5.2 in \protect\citet{Chernozhukov2014}.

I next show a similar result on $|\sigma_{\dagger,a}-\sigma_{\#,a}|$ to Theorem~\ref{thm: asymptotic CV expect Var}, with a non-asymptotic version, Theorem~\ref{thm: expect Var} in Section~\ref{sec: additional theory} of the Supplementary Material.
I consider a similar condition to Condition~\ref{cond: CV common mean Q or consistent Q}.

\begin{condition}[Consistency of $\tilde{Q}_a$ or suitable $Q_{\#,a}$] \label{cond: mean Q or consistent Q}
    One of the following conditions holds:
    \begin{enumerate}
        \item $Q_{\#,a}=Q_{*,a}$;
        \item $Q_{\#,a}(X) = \expect_{\tilde{Q}_a}[\tilde{Q}_a(X)]$ $P_*$-almost surely.
    \end{enumerate}
\end{condition}
Part~1 of Condition~\ref{cond: mean Q or consistent Q} holds if $\tilde{Q}_a$ is consistent for $Q_{*,a}$, similarly to Part~1 of Condition~\ref{cond: CV common mean Q or consistent Q}.
Part~2 of Condition~\ref{cond: mean Q or consistent Q} is even weaker than Part~2 of Condition~\ref{cond: CV common mean Q or consistent Q} as it always holds for an appropriately chosen $Q_{\#,a}$.

\begin{theorem}[Convergence rates of $|\sigma_{\dagger,a} - \sigma_{\#,a}|$ and $|\sigma_{\dagger} - \sigma_{\#}|$ without cross-fitting] \label{thm: asymptotic expect Var}
    Under Conditions~\ref{cond: 2nd moment}, \ref{cond: positivity}, \ref{cond: bounded Q} and \ref{cond: donsker}, as $n \to \infty$, $|\sigma_{\dagger,a} - \sigma_{\#,a}| = \bigO(\{\expect \|\tilde{Q}_a - Q_{\#,a} \|_2^2\}^{1/2} + n^{-1/2})$ and $|\sigma_{\dagger} - \sigma_{\#}| = \bigO(\{\expect \|\tilde{Q}_1 - Q_{\#,1} \|_2^2\}^{1/2} + \{\expect \|\tilde{Q}_0 - Q_{\#,0} \|_2^2\}^{1/2} + n^{-1/2})$. Additionally under Condition~\ref{cond: mean Q or consistent Q}, $|\sigma_{\dagger,a} - \sigma_{\#,a}| = \bigO(\expect \|\tilde{Q}_a - Q_{\#,a} \|_2^2 + n^{-1/2})$ and $|\sigma_{\dagger} - \sigma_{\#}| = \bigO(\expect \|\tilde{Q}_1 - Q_{\#,1} \|_2^2 + \expect \|\tilde{Q}_0 - Q_{\#,0} \|_2^2 + n^{-1/2})$.
\end{theorem}

It is possible to show that the $\bigO(n^{-1/2})$ term is $\smallo(n^{-1/2})$, with the exact rate depending on the covering number of $\funclass$.
This result shows that the term involving $\sigma_{\dagger,a} - \sigma_{\#,a}$ in \eqref{eq: asymptotic BE bound} does not dominate in the Berry-Esseen-type bound in Theorem~\ref{thm: asymptotic BE bound}.
If one is willing to ignore log factors when interpreting rates and $\expect \|\tilde{Q}_a - Q_{\#,a} \|_2^2 = \bigO(n^{-1})$ (e.g., $\tilde{Q}_a$ is estimated in a parametric model), then $\sigma_{\dagger,a} - \sigma_{\#,a} = \bigO(n^{-1/2})$ is at most comparable to $(\log n / n)^{1/2}$ in \eqref{eq: asymptotic BE bound}.

In many cases, the bounded uniform entropy integral condition $J(1,\funclass,M) < \infty$ in Condition~\ref{cond: donsker} is verified by establishing that $\funclass$ is subset of a VC or VC-hull class. I adopt the definition of VC index in \protect\citet{VanderVaart2023} and consider the following VC or VC-hull condition.

\begin{condition}[VC or VC-hull] \label{cond: donsker2}
    \leavevmode
    \begin{enumerate}
        \item $\funclass$ is a subset of a VC class with VC index $\nu$. \label{cond: VC}
        \item $\funclass$ is a subset of a VC-hull (the symmetric convex hull of a VC class) whose associated VC class has VC index $\nu$. \label{cond: VC hull}
    \end{enumerate}
\end{condition}
Under Part~\ref{cond: VC} of Condition~\ref{cond: donsker2}, $\sup_\mu \log N(\epsilon M, \funclass, L_2(\mu)) \leq C_1(\nu) + 2 \nu \log (1/\epsilon)$ for all $\epsilon \in (0,1)$, where $C_1(\nu)$ is a constant depending on $\nu$ only \protect\citep[Theorem~2.6.7 in][]{VanderVaart2023}.
Under Part~\ref{cond: VC hull} of Condition~\ref{cond: donsker2}, $\sup_\mu \log N(\epsilon M, \funclass, L_2(\mu)) \leq C_2(\nu) \epsilon^{-2\nu/(2\nu+1)}$, where $C_2(\nu)$ is another constant depending on $\nu$ only \protect\citep[Corollary~2.6.12 in][]{VanderVaart2023}.
These lead to the following result on the rates for \eqref{eq: asymptotic BE bound}, depending on which part of Condition~\ref{cond: donsker2} holds, and whether Conditions~\ref{cond: tail} and \ref{cond: more tail} hold.
I will compare and interpret the rates in Section~\ref{sec: BE bounds interpretation}.

\begin{corollary}[Convergence rates of non-cross-fit AIPW Wald-CI coverage for VC classes and VC-hulls] \label{coro: VC & hull rates}
    Suppose that Conditions~\ref{cond: 2nd moment}--\ref{cond: higher moments} and \ref{cond: donsker} hold. Consider the asymptotic scenario $n \to \infty$.
    \begin{enumerate}
        \item Under Part~\ref{cond: VC} of Condition~\ref{cond: donsker2}, the left-hand side of \eqref{eq: asymptotic BE bound} is
        $$\bigO \left( \Delta_{\VC,1} \{\log (1/\Delta_{\VC,1})\}^{1/2} + n^{-1/2} \log (1/\Delta_{\VC,1}) + \left\{ (\Delta_{\VC,1} + n^{-1/2} \right\} \log n + \Delta_{\VC,2} \right),$$
        where $\Delta_{\VC,1}=(\expect \| \tilde{Q}_a - Q_{\#,a} \|_{2}^2)^{1/3}$ and $\Delta_{\VC,2}=0$. Additionally under Conditions~\ref{cond: tail} and \ref{cond: more tail}, if $\tilde{r}_1(n) = \smallo(\{\log n\}^{-1/\tilde{q}_1})$, the rate holds with $\Delta_{\VC,1}$ and $\Delta_{\VC,2}$ replaced by $\tilde{r}_1(n) (\log n)^{1/\tilde{q}_1}$ and $\tilde{r}_2(n) (\log n)^{1/\tilde{q}_2}$, respectively.
        
        \item Under Part~\ref{cond: VC hull} of Condition~\ref{cond: donsker2}, the left-hand side of \eqref{eq: asymptotic BE bound} is
        $$\bigO \left( \Delta_{\VChull,1} + \Delta_{\VChull,2} + \left\{ \Delta_{\VChull,1}^{(2\nu+1)/(\nu+1)} + n^{-1/2} \right\} \log n \right),$$
        where $\Delta_{\VChull,1}=(\expect \| \tilde{Q}_a - Q_{\#,a} \|_{2}^2)^{(\nu+1)/(5 \nu+3)} + n^{-(\nu+1)/(6 \nu +2)}$ and $\Delta_{\VChull,2} = 0$.
        
        Additionally under Conditions~\ref{cond: tail} and \ref{cond: more tail}, if $\expect \| \tilde{Q}_a - Q_{\#,a} \|_{2}^2 = \smallo(\{\log n\}^{-1/\tilde{q}_1})$, this rate holds with $\Delta_{\VChull,1}$ and $\Delta_{\VChull,2}$ replaced by $\{\tilde{r}_1(n)\}^{(\nu+1)/(2\nu+1)} (\log n)^{(\nu+1)/\{\tilde{q}_1 (2\nu+1)\}}$ and $n^{-1/2} \{\tilde{r}_1(n)\}^{-2\nu/(2\nu+1)} \allowbreak (\log n)^{-2\nu/\{\tilde{q}_1 (2\nu+1)\}} + \tilde{r}_2(n) (\log n)^{1/\tilde{q}_2}$, respectively.
    \end{enumerate}
\end{corollary}

\subsection{Interpretation and comparison of asymptotic rates} \label{sec: BE bounds interpretation}

First of all, it is worth noting that I only obtain upper bounds on the rates in Theorems~\ref{thm: asymptotic CV BE bound}--\ref{thm: asymptotic expect Var}, rather than the rates themselves.
In other words, two terms whose rates have different upper bounds may actually have the same rate.
I disregard the nuance between the rate and its upper bound for the rest of this section.

For simplicity, particularly when comparing cross-fit and non-cross-fit procedures, I treat $K$ as an absolute constant that does not vary with $n$ for now.
Because of the convergence rate of $|\sigma_{\dagger,a}-\sigma_{\#,a}|$ shown in Theorems~\ref{thm: asymptotic CV expect Var} and \ref{thm: asymptotic expect Var}, the forms in \eqref{eq: asymptotic CV BE bound} and \eqref{eq: asymptotic BE bound} are somewhat deceiving in the sense that the seemingly first order term $\phi(z_\alpha) z_\alpha (\sigma_{\dagger,a}-\sigma_{\#,a})/\sigma_{\#,a}$ on the left-hand side has a faster convergence rate than the right-hand side, regardless of whether approximate sub-Weibull conditions hold and whether $Q_{\#,a}$ is chosen to yield a faster rate of $|\sigma_{\dagger,a}-\sigma_{\#,a}|$.
In other words, the rate of the right-hand side of \eqref{eq: asymptotic CV BE bound} and \eqref{eq: asymptotic BE bound} is an upper bound on the convergence rate of the Wald-CI coverage to its nominal coverage $1-\alpha$.

For the cross-fit procedure, the $\bigO(\{\log n/n\}^{1/2})$ term in the rate in Theorem~\ref{thm: asymptotic CV BE bound} is the same as the rate in the Berry-Esseen Theorem except for a log factor.
The contribution from estimating the nuisance outcome model can converge slower and strongly depends on whether approximate sub-Weibull conditions \ref{cond: CV tail} and \ref{cond: more CV tail} hold:
If $\hat{Q}_{k,a}$ is estimated from a (potentially misspecified) parametric model, so that we can take $\hat{r}_1(n)=\hat{r}_2(n)=\bigO(n^{-1/2})$, and approximate sub-Weibull conditions hold, then the contribution from estimating the outcome model is $\bigO(n^{-1/2} \log n)$, same as the Berry-Esseen theorem up to a log factor. However, if $\hat{Q}_{k,a}$ is estimated more flexibly so that $\|\hat{Q}_{k,a} - Q_{\#,a} \|_2$ has a slower rate, or approximate sub-Weibull conditions do not hold, then the contribution from nuisance estimation can have a slower rate.

For the non-cross-fit procedure, the rate additionally strongly depends on the complexity of the function class $\funclass$ containing $\tilde{Q}_a$, as measured by its covering number.
This is because, compared to the cross-fit procedure, Donsker conditions are additionally required to bound empirical process terms, whose rate depends on the function class complexity.
Among the empirical process terms in my derivation, the term that impacts the rate most is $(P_n-P_*) \{ \transform_a(\tilde{Q}_a) - \transform_a(Q_{\#,a}) \}$ concerning the deviation of $\tilde{\psi}_a - \psi_{*,a}$ from the sample mean $(P_n - P_*) D_a(Q_{\#,a},\psi_{*,a})$.

By Theorem~\ref{thm: asymptotic BE bound} and Corollary~\ref{coro: VC & hull rates}, the rates for VC classes (e.g., many parametric models) are similar to the rates for singleton classes, the simplest possible classes, which achieve the lower bound of the right-hand side of \eqref{eq: asymptotic BE bound}.
These rates are also similar to those for the cross-fit procedure, suggesting that cross-fit and non-cross-fit procedures have similar CI coverage if $\funclass$ is not rich.

In contrast, the rates for VC-hull classes, which are much richer than VC classes, can be much slower. If the VC index $\nu$ of the associated VC-class is large, namely $\funclass$ is rich, then the rate would be slow, and the approximate sub-Weibull conditions do not improve the rate as much as in the cross-fit procedure or the non-cross-fit procedure with VC classes.
For a rich VC-hull class, namely large $\nu$, the rate contribution from estimating the outcome model is approximately $\tilde{r}_1(n)^{1/2}$ up to a log factor, which is slower than the convergence rate of $\tilde{Q}_a$.
Theorem~\ref{thm: asymptotic BE bound} and Corollary~\ref{coro: VC & hull rates} are mainly for illustration rather than for proving accurate rates for specific regression algorithms, because tighter bounds might be derived for a specific algorithm to estimate $Q_{*,a}$ with additional knowledge on this algorithm.

\section{Bias of variance estimators} \label{sec: Var bias}

Although the bias of variance estimators does not contribute to the dominating terms in the difference between CI coverage and nominal coverage, because the magnitude of dominating terms might be conservative or the impact of the bias might be strong in moderate samples due to a large constant, I study this bias in a somewhat heuristic manner in this section.

In proving Theorems~\ref{thm: asymptotic CV expect Var} and \ref{thm: asymptotic expect Var}, the following proposition can be derived as the byproducts.
Consider the asymptotic scenario $n \to \infty$ in this section and recall the notation $\Theta_+$ introduced in Section~\ref{sec: prelim}.

\begin{proposition}[Asymptotic order of bias of variance estimators] \label{prop: asymptotic Var bias}
    Assume that Condition~\ref{cond: positivity} holds and let $f$ denote the positive bounded function $x \mapsto \{ \pi_*(0 \mid x)/\pi_*(1 \mid x) \}^{1/2}$.
    \begin{enumerate}
        \item For the cross-fit procedure, under Conditions~\ref{cond: CV exchangeable} and \ref{cond: CV common mean Q or consistent Q}, \label{prop: asymptotic Var bias 1}
        \begin{align}
            &\sigma_{\dagger,a}^2 - \sigma_{\#,a}^2 = \Theta_+ \left( \expect \|\hat{Q}_{k,a} - Q_{\#,a}\|_2^2 \right) - \Theta_+(K/n), \label{eq: asymptotic CV Var bias} \\
            &\sigma_{\dagger}^2 - \sigma_{\#}^2 = \Theta_+ \left( \expect \left\| f (\hat{Q}_{k,1} - Q_{\#,1}) + (\hat{Q}_{k,0} - Q_{\#,0})/f \right\|_2^2 \right) - \Theta_+(K/n). \label{eq: asymptotic ATE CV Var bias}
        \end{align}

        \item For the non-cross-fit procedure, under Conditions~\ref{cond: donsker} and \ref{cond: mean Q or consistent Q}, \label{prop: asymptotic Var bias 2}
        \begin{align}
            &\sigma_{\dagger,a}^2 - \sigma_{\#,a}^2 = \Theta_+ \left( \expect \|\tilde{Q}_a - Q_{\#,a}\|_2^2 \right) + \bigO(n^{-1/2}) - \Theta_+(n^{-1}), \label{eq: asymptotic Var bias} \\
            &\sigma_{\dagger}^2 - \sigma_{\#}^2 = \Theta_+ \left( \expect \left\| f (\tilde{Q}_1 - Q_{\#,1}) + (\tilde{Q}_0 - Q_{\#,0})/f \right\|_2^2 \right) + \bigO(n^{-1/2}) - \Theta_+(n^{-1}). \label{eq: asymptotic ATE Var bias}
        \end{align}
    \end{enumerate}
\end{proposition}

The corresponding non-asymptotic results can be found in Theorems~\ref{thm: CV expect Var} and \ref{thm: expect Var} in Section~\ref{sec: additional theory} of the Supplementary Material.
The term $\bigO(n^{-1/2})$ in \ref{eq: asymptotic Var bias} and \ref{eq: asymptotic ATE Var bias} might be improved to $\smallo(n^{-1/2})$ with its exact order depending on the function class $\funclass$.

For the cross-fit procedure, suppose that $K$ is fixed as $n \to \infty$.
If $\expect \|\hat{Q}_{k,a} - Q_{\#,a}\|_2^2$ has order $n^{-1}$ (e.g., $\hat{Q}_{k,a}$ is a parametric estimator), then the bias $\sigma_{\dagger,a}^2 - \sigma_{\#,a}^2 = \bigO(n^{-1})$.
If $\expect \|\hat{Q}_{k,a} - Q_{\#,a}\|_2^2 = \smallo(n^{-1/2})$, as anticipated when $\|\hat{Q}_{k,a} - Q_{\#,a}\|_2 = \smallo_p(n^{-1/4})$ (a common requirement in the nonparametric inference literature), then the bias $\sigma_{\dagger,a}^2 - \sigma_{\#,a}^2 = \smallo(n^{-1/2})$ still vanishes at a rate faster than the parametric rate $n^{-1/2}$.
If $\expect \|\hat{Q}_{k,a} - Q_{\#,a}\|_2^2$ vanishes slower than $n^{-1}$, then $\sigma_{\dagger,a}^2 > \sigma_{\#,a}^2$ for sufficiently large $n$. That is, for sufficiently large samples, $\hat{\sigma}_a^2$ is anticipated to overestimate $\sigma_{\#,a}^2$, contributing to overcoverage.
Similarly, $\hat{\sigma}^2$ is anticipated to overestimate $\sigma_{\#}^2$ if $\expect \|\hat{Q}_{k,a} - Q_{\#,a}\|_2^2$ vanishes slower than $n^{-1}$ by \eqref{eq: asymptotic ATE CV Var bias}, contributing to overcoverage of the Wald-CI for the ATE $\psi_*$.

For the non-cross-fit procedure, the right-hand side of \eqref{eq: asymptotic Var bias} has two terms similar to \eqref{eq: asymptotic CV Var bias}, while the additional empirical process term of order $\bigO(n^{-1/2})$ needs more investigation.
In particular, the sign of the bias $\sigma_{\dagger,a}^2 - \sigma_{\#,a}^2$ is more challenging to analyze.
The following proposition might shed light on this sign in a special case.

\begin{proposition} \label{prop: sign Var bias}
    Consider the non-crossfit procedure.
    Suppose that Conditions~\ref{cond: positivity}, \ref{cond: donsker} and \ref{cond: mean Q or consistent Q} hold,
    and $\tilde{Q}_a$ and $Q_{\#,a}$ are projections in the sense that
    $$P_n \frac{\ind(A=a)}{\pi_*(a \mid X)} \{Y - \tilde{Q}_a(X)\} \tilde{Q}_a(X) = P_* \frac{\ind(A=a)}{\pi_*(a \mid X)} \{Y-Q_{\#,a}(X)\} Q_{\#,a}(X) = 0.$$
    Then
    \begin{equation}
        \sigma_{\dagger,a}^2 - \sigma_{\#,a}^2 = \expect \left[ P_n \left\{ \frac{\ind(A=a)}{\pi_*(a \mid X)^2} \right\} \left\{ [Y-\tilde{Q}_a(X)]^2 - [Y-Q_{\#,a}(X)]^2 \right\} \right] + \expect[P_n (\tilde{Q}_a^2 - Q_{\#,a}^2)] - \Theta_+(n^{-1}), \label{eq: asymptotic Var bias 2}
    \end{equation}
    where the $\Theta_+(n^{-1})$ term is identical to that in \eqref{eq: asymptotic Var bias}.
    Under the same above conditions, suppose that $\tilde{Q}_a$ and $Q_{\#,a}$ are projections in the stronger sense that
    $$P_n \frac{\ind(A=a)}{\pi_*(a \mid X)} \{Y - \tilde{Q}_a(X)\} \tilde{Q}_{a'}(X) = P_* \frac{\ind(A=a)}{\pi_*(a \mid X)} \{Y-Q_{\#,a}(X)\} Q_{\#,a'}(X) = 0$$
    for any $a,a' \in \{0,1\}$. Then
    \begin{align}
        \sigma_{\dagger}^2 - \sigma_{\#}^2 &= \sum_{a=0}^1 \expect \left[ P_n \left\{ \frac{\ind(A=a)}{\pi_*(a \mid X)^2} \right\} \left\{ [Y-\tilde{Q}_a(X)]^2 - [Y-Q_{\#,a}(X)]^2 \right\} \right] \nonumber\\
        &\quad+ \expect[P_n \{ ( \tilde{Q}_1 - \tilde{Q}_0 )^2 - ( Q_{\#,1} - Q_{\#,0} )^2 \}] - \Theta_+(n^{-1}), \label{eq: asymptotic ATE Var bias 2}
    \end{align}
    where the $\Theta_+(n^{-1})$ term is identical to that in \eqref{eq: asymptotic ATE Var bias}.
\end{proposition}

The first term on the right-hand side of \eqref{eq: asymptotic Var bias 2} is non-positive when $\tilde{Q}_a$ fits the data no worse than $Q_{\#,a}$.
This holds, for example, if $x \mapsto \pi_*(a \mid x)$ is a constant function, $\tilde{Q}_a$ minimizes the empirical mean squared error among observations with $A=a$ in the function class $\funclass$, and $Q_{\#,a} \in \funclass$.
The magnitude of the second term $\expect[P_n (\tilde{Q}_a^2 - Q_{\#,a}^2)]$ on the right-hand side of \eqref{eq: asymptotic Var bias 2} can be bounded with bounds for empirical processes, but such analysis might fail to reveal this term's sign.
When $P_n \tilde{Q}_a^2$ is unrelated to the loss for obtaining $\tilde{Q}_a$, such bounds might also be overly conservative for its magnitude and $\expect[P_n (\tilde{Q}_a^2 - Q_{\#,a}^2)]$ might be arguably anticipated to be close to zero.
In other cases, $P_n \tilde{Q}_a^2$ often relates to a penalization on the roughness of $\tilde{Q}_a$ to avoid overfitting, in which case it might be arguably anticipated that $P_n (\tilde{Q}_a^2 - Q_{\#,a}^2) \leq 0$.
Therefore, it might be arguably anticipated that $\sigma_{\dagger,a}^2 - \sigma_{\#,a}^2 \leq 0$ if $\expect \|\tilde{Q}_a - Q_{\#,a}\|_2^2$ vanishes slowly and $n$ is large, contributing to decreased CI coverage.
Similarly, it might be arguably anticipated that the CI coverage for the ATE $\psi_*$ is decreased if $\expect \|\tilde{Q}_a - Q_{\#,a}\|_2^2$ vanishes slowly based on \eqref{eq: asymptotic ATE Var bias 2}.

The above analyses also suggest a potential trade-off between efficiency and Wald-CI coverage convergence rate, if the bias of the variance estimator actually contributes substantially to CI coverage, although this trade-off does not show up in (potentially conservative) Theorems~\ref{thm: asymptotic CV BE bound} and \ref{thm: asymptotic BE bound}.
I illustrate with the cross-fit procedure.
If we use a highly flexible estimator $\hat{Q}_{k,a}$ of $Q_{*,a}$ in order to achieve more efficiency in the ATE estimator $\hat{\psi}_a$, we might anticipate the followings: (i) $\expect \|\hat{Q}_{k,a} - Q_{\#,a}\|_2^2$ vanishes slowly, so (ii) $\sigma_{\dagger,a}^2 \to \sigma_{\#,a}^2$ slowly, and hence (iii) the Wald-CI coverage $\Pr(n^{1/2} (\hat{\psi}_a - \psi_{*,a}) \leq z_\alpha \hat{\sigma}_a)$ converges to the nominal coverage $1-\alpha$ slowly.

\section{Simulation} \label{sec: sim}

\subsection{Setup} \label{sec: sim setup}

I conduct a simulation to investigate the performance of various estimators of the ATE $\psi_{*,1}-\psi_{*,0}$ in an RCT.
The data is generated as follows.
The covariate $X$ is $p$ i.i.d.\  random variables generated from $\mathrm{Unif}(-1,1)$, with $p \in \{3,5,7\}$.
The treatment $A \mid X \sim \mathrm{Binom}(1, 0.5)$, so $\pi_*(1 \mid X) = 0.5$.
Conditional on $(X,A)$, the outcome $Y$ is generated from a normal distribution with mean
\begin{align*}
    Q_{*,a}(x) &:= 8 a + 7 |x_1| + 10 \sin(20 x_2) - 10 \{ \ind(x_3>0.3) x_3^2 + \ind(x_3 \leq 0.3)\} + 10 x_1 |x_2| x_3 + 14 \sqrt{|x_3|} \\
    &\quad+ 4 a/\{1+\expit(-8 x_1 x_2)\} + 9 \ind(x_1^2+x_2^2 > 0.25) - 13 \min\{x_1,x_2,x_3\} \\
    &\quad+ \max\{ \cos(7 x_1 + 13 x_3), x_2 \} (-17 + 5 a)
\end{align*}
and standard deviation $\max\{ 1, 0.1 |Q_{*,a}(x)|, 2 |\sin(Q_{*,a}(x))| \}$, where $x_j$ denote the $j$th entry of $x$.
The true mean function $Q_{*,a}$ depends on the first three components of $X$ and has nonlinear transformations, discrete jumps and interactions.
The sample size $n \in \{100,200,400\}$ is relatively small to moderate.
Each scenario has 800 simulation runs.
The true ATE is calculated by generating an independent sample of $10^6$ observations and computing the Monte Carlo mean of $Q_{*,1}(X)-Q_{*,0}(X)$. The efficient asymptotic variance (similar to $\sigma_{*,a}^2$ when estimating the mean counterfactual outcome $\psi_{*,a}$) is calculated as the Monte Carlo variance of the efficient influence function, similarly to the computation of the true ATE.

I consider the following estimators:
(i)~\texttt{IPW}: the inverse probability weighted estimator $\hat{\psi}_\IPW := P_n \{ \ind(A=1)/\pi_*(1 \mid X) - \ind(A=0)/\pi_*(0 \mid X) \} Y$, whose standard error is calculated based on the empirical variance of the estimated influence function $v \mapsto \{ \ind(a=1)/\pi_*(1 \mid x) - \ind(a=0)/\pi_*(0 \mid x) \} y - \hat{\psi}_\IPW$.
(ii)~\texttt{ANCOVA}: the analysis of covariance estimator
with heteroskedasticity-consistent standard error \protect\citep{Long2000}.
(iii)~\texttt{misSL}: a non-cross-fit AIPW estimator with $Q_{*,a}$ estimated via Super Learner \protect\citep{VanderLaan2007}, using a library that contains relatively simple estimators (e.g., GLM) and no consistent estimator of $Q_{*,a}$ (see Section~\ref{sec: sim2} of the Supplementary Material for a complete list).
(iv)~\texttt{CVmisSL}: the cross-fit version of \texttt{misSL} with $K=20$ to achieve a relatively large training sample size for estimating $Q_{*,a}$, following the recommendation of \protect\citet{Phillips2023}.
(v)~\texttt{SL}: an estimator similar to \texttt{misSL} except that the library further contains highly adaptive lasso \protect\citep{Benkeser2016,Hejazi2020} with various tuning parameters (see Section~\ref{sec: sim2} of the Supplementary Material for a complete list), which should be flexible enough to be consistent in large samples.
(vi)~\texttt{CVSL}: the cross-fit version of \texttt{SL} with $K=20$.
I construct 95\% Wald-CIs for all estimators.

I also calculate the mean plug-in estimated asymptotic variance and the oracle scaled variance (similar to $\sigma_{\dagger,a}^2$ and $\sigma_{\#,a}^2$ when estimating $\psi_{*,a}$) for $Q_{\#,a}$ being the mean of $\hat{Q}_{k,a}$ or $\tilde{Q}_a$, that is, satisfying Part~2 of Condition~\ref{cond: CV common mean Q or consistent Q} or \ref{cond: mean Q or consistent Q}.
For each AIPW estimator, the mean estimated asymptotic variance is approximated by the Monte Carlo mean over simulation runs.
It is more challenging to calculate $\sigma_{\#,a}^2$. I describe its calculation in Section~\ref{sec: sim2} of the Supplementary Material.

\subsection{Results} \label{sec: sim result}

Figures~\ref{fig: distr} and \ref{fig: var} present the sampling distributions of the ATE estimators and the corresponding estimated asymptotic variance, respectively.
Figure~\ref{fig: qq} presents the QQ-plot of the AIPW estimators.
As expected, all estimators appear approximately unbiased; \texttt{IPW} has the largest variance because predictive covariates are not used; due to less bias in estimating $Q_{*,a}$, \texttt{SL} and \texttt{CVSL} have smaller variance than \texttt{ANCOVA}, \texttt{misSL}, and \texttt{CVmisSL}.

The efficient asymptotic variance (similar to $\sigma_{*,a}^2$, the horizontal gray line in Fig.~\ref{fig: var}) is substantially smaller than the scaled variance of the ATE estimators, even for \texttt{SL} and \texttt{CVSL}.
In contrast, the oracle variance (similar to $\sigma_{\#,a}^2$, red crosses in Fig.~\ref{fig: var}) is much closer to the estimators' scaled variance.
Although the oracle variance is often smaller than the estimators' variance, the QQ-plots in Fig.~\ref{fig: qq} suggest that the oracle variance describes the sampling distribution well, except for a few points in the tails of \texttt{CVSL} and \texttt{SL}.
This phenomenon suggests that a sample size-dependent oracle variance characterizes the variance of the estimator much better than the theoretical large-sample variance in small-to-moderate samples, particularly when the outcome model estimator converges to the truth slowly.

All AIPW estimators' sampling distributions are close to normal, even for the non-cross-fit AIPW estimator with a rich functions class (\texttt{SL}).
Since the empirical processes concern the deviation of $\tilde{\psi}_a$ from its asymptotic linear expansion, the empirical process bounds used to derive Theorem~\ref{thm: asymptotic BE bound} appears to be overly conservative, as whether cross-fitting is adopted appears not to strongly deteriorate the normality of the non-cross-fit AIPW estimator, even if the function class $\funclass$ is rich and the sample size is small to moderate.

For \texttt{misSL} and \texttt{CVmisSL}, the estimators' scaled variance, the oracle variance, and the mean estimated variance are close, particularly when $n \geq 200$, which aligns with the fast convergence rate when the outcome model estimator converges fast and the function class $\funclass$ is not rich.
The mean estimated variance of the cross-fit procedure \texttt{CVSL} is larger than the oracle variance, particularly when $n=400$, which aligns with \eqref{eq: asymptotic CV Var bias} when $\hat{Q}_{k,a}$ converges to $Q_{\#,a}$ slowly.
For \texttt{SL}, the estimated variance is much smaller than the oracle variance and further smaller than the estimator's scaled variance. In fact, the estimator's scaled variance is almost out of the range of estimated variance. This aligns with the heuristic discussion in Section~\ref{sec: Var bias}.

Figure~\ref{fig: coverage} presents the CI coverage.
\texttt{SL} severely undercovers, and a plausible reason is the underestimated variance.
All other CIs appear to have approximately correct coverage, particularly when $n \geq 200$, aligning with the results and discussions in Sections~\ref{sec: BE bounds} and \ref{sec: Var bias}.
Although sometimes \texttt{CVSL} has a coverage that is statistically significantly lower than 95\%, the coverage is arguably practically close to 95\%, particularly when $n \geq 200$. Interestingly, with more covariates, \texttt{CVSL} appears to have higher coverage.
A possible reason is that, with more covariates, $\hat{Q}_{k,a}$ has more variability, so that the estimated variance has a positive bias by \eqref{eq: asymptotic CV Var bias}, which might remedy undercoverage.
The potential trade-off between efficiency and Wald-CI coverage shows up for the non-cross-fit procedure in the simulation: \texttt{misSL} is less efficient than $\texttt{SL}$ but has better Wald-CI coverage.
This trade-off is less clear for the cross-fit procedure in the simulation. A potential reason is the missing empirical process term.

In this simulation, \texttt{CVSL} appears to have the best overall performance: Its variance is among the smallest, and its CI coverage appears practically close to 95\% in all scenarios.
This simulation provides another piece of evidence for the superior performance of cross-fit nonparametric AIPW estimators of ATE in RCTs.

\section{Discussion} \label{sec: discussion}

The theoretical results might have room for improvement.
First, the bounds for empirical processes might be overly conservative. For example, it might be possible to reduce those bounds with some additional high-level conditions.
Second, there remains a theoretical gap regarding how the bias of the estimated variance impacts the CI coverage. Although the simulation and my argument in Section~\ref{sec: Var bias} indicate that cross-fitting may increase CI coverage and non-cross-fitting may decrease CI coverage, this bias is not a first-order term in my Berry-Esseen-type bounds in Theorems~\ref{thm: asymptotic CV BE bound} and \ref{thm: asymptotic BE bound}.
Third, I derived the bounds based on empirical process theory and tail bounds.
For simpler estimators such as the sample mean, methods based on characteristic functions yield sharper characterization (without log factors) under other mild conditions (e.g., the Berry-Essen Theorem).
I did not adopt this approach to accommodate the black-box nuisance function estimators.
With additional knowledge on the nuisance estimators, it is possible to use characteristic functions to achieve a more accurate characterization of the CI coverage convergence rate.
Finally, for the cross-fitting procedure, although I treat the number of folds $K$ as a variable, as I discuss in Section~\ref{sec: K choice} in the Supplementary Material, my bound might have room for improvement in terms of $K$. A theoretically sound recommendation of $K$ is an open question.

Nevertheless, I provide new theoretical justifications for adopting cross-fitting when using an AIPW estimator with a flexible outcome model estimator to estimate the ATE in an RCT.
On the one hand, my upper bound of the convergence rate of CI coverage to nominal coverage for cross-fitting vanishes faster than the non-cross-fit version.
On the other hand, cross-fitting overestimates the variance, providing additional protection against undercoverage.
The non-asymptotic results in Section~\ref{sec: additional theory} of the Supplementary Material may also enable asymptotic analysis in high-dimensional settings.
The bounds under the observational setting are an interesting future direction, but they might require substantively different analyses, although the analytic tools in the paper might help.
The results of this paper lay the foundation for future research on the coverage of CIs based on non- and semi-parametric estimators, and potentially for providing practical guidance on using such procedures.

\bibliographystyle{apalike}
\bibliography{ref}

\clearpage

\begin{figure}[tbh]
    \centering
    \includegraphics[width=0.9\linewidth]{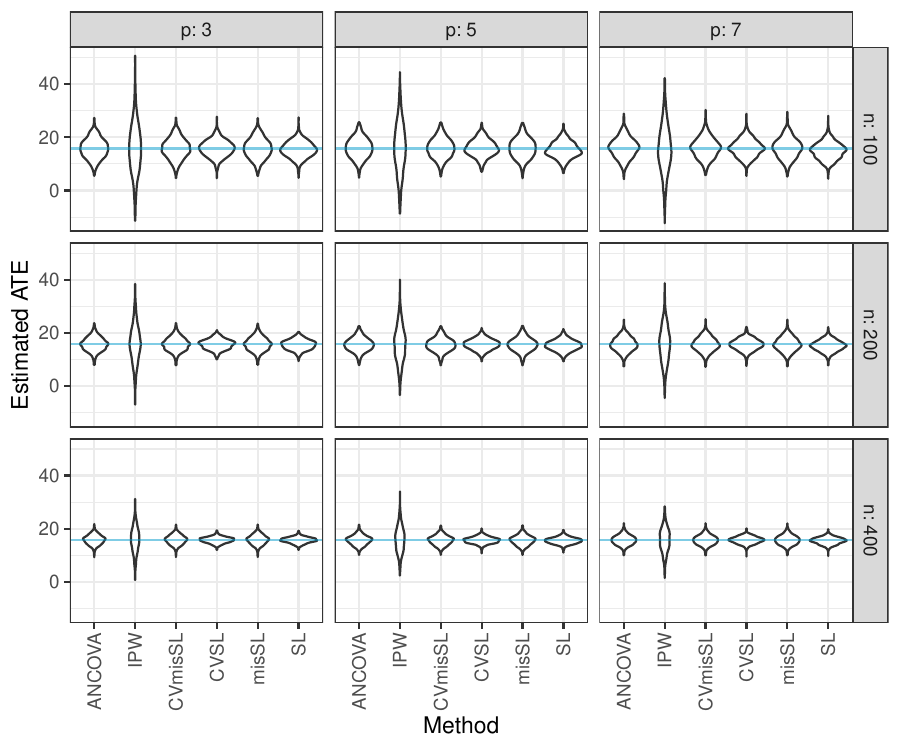}
    \caption{Sampling distribution of ATE estimators. The horizontal blue line is the true ATE.}
    \label{fig: distr}
\end{figure}

\begin{figure}[tbh]
    \centering
    \includegraphics[width=0.9\linewidth]{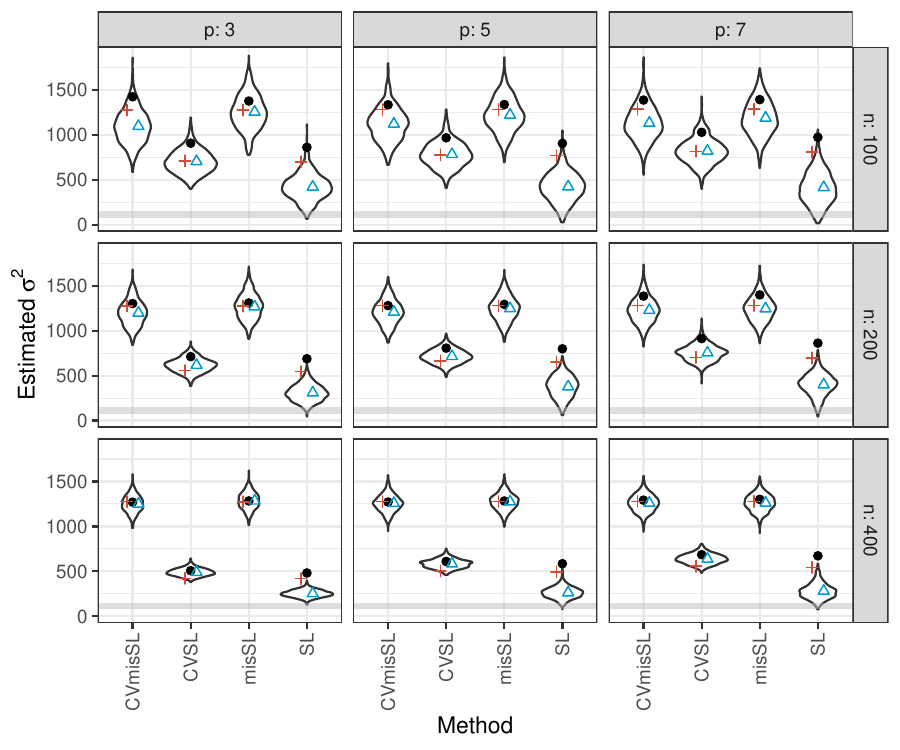}
    \caption{Sampling distribution of estimated influence function-based asymptotic variance for AIPW estimators (similar to $\hat{\sigma}_a^2$ and $\tilde{\sigma}_a^2$).
    The violin-like shapes are vertical densities of the estimated variance.
    The black dots are the Monte Carlo variance of ATE estimators, scaled by $n$ (similar to $n \Var(\hat{\psi}_a)$ and $n \Var(\tilde{\psi}_a)$).
    The blue triangles are the Monte Carlo average of the estimated variance (similar to $\sigma_{\dagger,a}^2$).
    The red crosses are the oracle variance (similar to $\sigma_{\#,a}^2$).
    The horizontal gray line is the efficient asymptotic variance (similar to $\sigma_{*,a}^2$).}
    \label{fig: var}
\end{figure}

\begin{figure}[tbh]
    \centering
    \includegraphics[width=0.9\linewidth]{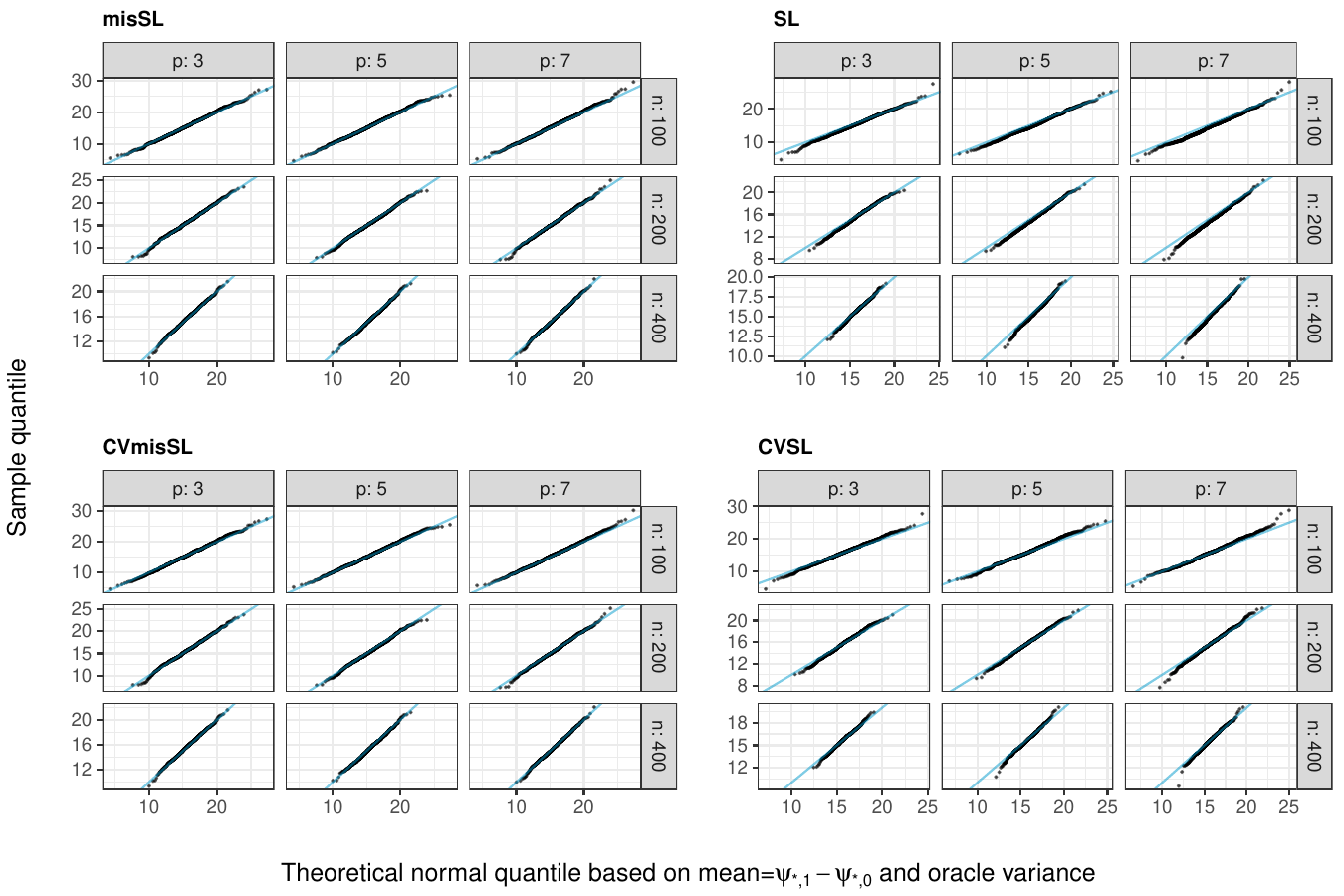}
    \caption{QQ-plot of AIPW estimators. The y-axis is the Monte Carlo sample quantile. The x-axis is the theoretical quantile of a normal distribution with mean being true ATE $\psi_{*,1}-\psi_{*,0}$ and variance being the oracle variance (similar to $\sigma_{\#,a}^2$). The blue line is the diagonal line $y=x$.}
    \label{fig: qq}
\end{figure}

\begin{figure}[tbh]
    \centering
    \includegraphics[width=0.9\linewidth]{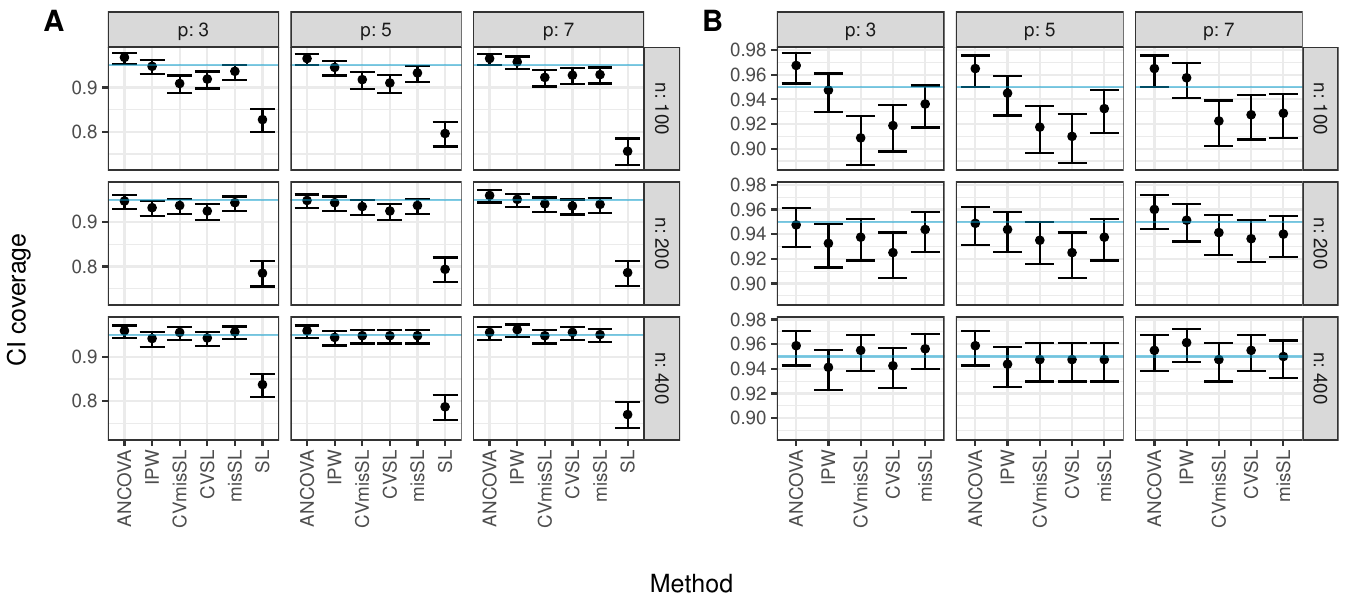}
    \caption{CI coverage with 95\% Wilson confidence intervals of coverage. The horizontal blue line is 95\%, the nominal coverage. Subfigure~(B) is a zoomed version of Subfigure~(A) with \texttt{SL} removed.}
    \label{fig: coverage}
\end{figure}

\clearpage

\setcounter{page}{1}
\setcounter{section}{0}
\renewcommand{\thepage}{S\arabic{page}}%
\renewcommand{\thesection}{S\arabic{section}}%
\renewcommand{\theHsection}{S\thesection}%
\setcounter{table}{0}
\renewcommand{\thetable}{S\arabic{table}}%
\renewcommand{\theHtable}{S\thetable}%
\setcounter{figure}{0}
\renewcommand{\thefigure}{S\arabic{figure}}%
\renewcommand{\theHfigure}{S\thefigure}%
\setcounter{equation}{0}
\renewcommand{\theequation}{S\arabic{equation}}%
\renewcommand{\theHequation}{S\theequation}%
\setcounter{condition}{0}
\renewcommand{\thecondition}{S\arabic{condition}}%
\renewcommand{\theHcondition}{S\thecondition}%
\setcounter{seriescond}{0}
\renewcommand{\theseriescond}{S\arabic{seriescond}}%
\renewcommand{\theHseriescond}{S\theseriescond}%
\setcounter{lemma}{0}
\renewcommand{\thelemma}{S\arabic{lemma}}%
\renewcommand{\theHlemma}{S\thelemma}%
\setcounter{theorem}{0}
\renewcommand{\thetheorem}{S\arabic{theorem}}%
\renewcommand{\theHtheorem}{S\thetheorem}%
\setcounter{corollary}{0}
\renewcommand{\thecorollary}{S\arabic{corollary}}%
\renewcommand{\theHcorollary}{S\thecorollary}%
\setcounter{proposition}{0}
\renewcommand{\theproposition}{S\arabic{proposition}}%
\renewcommand{\theHproposition}{S\theproposition}%

\begin{center}
    \LARGE Supplementary Material to ``\ourtitle''
\end{center}

\section{Additional theoretical results} \label{sec: additional theory}

In the following non-asymptotic results, the terms concerning the rates in the corresponding asymptotic scenario (Theorems~\ref{thm: asymptotic CV BE bound}--\ref{thm: asymptotic expect Var}), i.e., the terms depending on $(n,K,\hat{Q}_{k,a},\delta)$ in the upper bounds, are colored blue.

The next theorem is a non-asymptotic version of Theorem~\ref{thm: asymptotic CV BE bound} for $\psi_{*,a}$. A non-asymptotic version for $\psi_*$ can be derived similarly.
It implies Theorem~\ref{thm: asymptotic CV BE bound} by taking $q \geq 2$ to be an absolute constant and noting (i) that $\sigma_{\dagger,a}$ is bounded away from zero as $n \to \infty$ by Theorem~\ref{thm: CV expect Var} and Condition~\ref{cond: 2nd moment}, and (ii) that $2 \leq K \leq n$.
The case $z_\alpha=0$ needs special consideration because, in this case, the randomness in $\hat{\sigma}_a$ vanishes.
The case $\psi_{*,a}=0$ also needs special consideration because it concerns the behavior of $\hat{\psi}_{k,a}^2$, which is involved when analyzing the tail probability of $\hat{\sigma}_{k,a}$.
Recall the constants $\bar{\sigma}_a^2$ and $\bar{\rho}$ introduced below Condition~\ref{cond: higher moments}.

\begin{theorem}[Non-asymptotic Berry-Esseen-type bound for cross-fit AIPW Wald-CI] \label{thm: CV BE bound}
    Suppose that $\sigma_{\dagger,a}>0$, i.e., $\Pr(\hat{\sigma}_a>0) > 0$.
    Under Conditions~\ref{cond: 2nd moment}--\ref{cond: higher moments}, for any constant $q>0$, if $\psi_{*,a} = 0$ or $z_\alpha = 0$ (i.e., $\alpha = 0.5$) or $|I_k|$ is sufficiently large so that $2 (\sqrt{2}+1)^2 q \bar{\sigma}_a^2 \hl{\log n/|I_k|} \leq \psi_{*,a}^2$ for all $k=1,\ldots,K$, then
    \begin{align}
        &\left| \Pr(\sqrt{n} (\hat{\psi}_a - \psi_{*,a}) \leq z_\alpha \hat{\sigma}_a) - (1-\alpha) - \phi(z_\alpha) z_\alpha \frac{\sigma_{\dagger,a}-\sigma_{\#,a}}{\sigma_{\#,a}} \right| \nonumber \\
        &\leq \const \frac{\rho_{\#,a}}{\sigma_{\#,a}^3 \hl{\sqrt{n}}} + \const \hl{\sum_{k=1}^K} \Bigg[ \hl{\hat{\mathcal{R}}_k} + \frac{\hl{\hat{S}_k^2}}{\sigma_{\#,a}^4} + \left( \frac{\bar{m}}{\underline{m}^{3/2}} + \ind(z_\alpha \neq 0) \frac{\bar{\rho}}{\sigma_{*,a}^{3/2}} \right) \hl{\frac{1}{\sqrt{|I_k|}}} + \frac{|z_\alpha| \sqrt{2 q \hl{|I_k|} \bar{m} \hl{\log n}}}{\hl{n} \sigma_{\#,a} \sigma_{\dagger,a}} + \hl{n^{-q}} \nonumber \\
        &\qquad+ \frac{\hl{|I_k|} |z_\alpha|}{\hl{n} \sigma_{\#,a} \sigma_{\dagger,a}} \left\{ |\psi_{*,a}| \sqrt{q \frac{\bar{\sigma}_a^2}{\hl{|I_k|}} \hl{\log n}} + \ind(\psi_{*,a}=0) q \frac{\bar{\sigma}_a^2}{\hl{|I_k|}} \hl{\log n} + \hl{\frac{1}{|I_k|}} \left\{ \sigma_{\#,a} + \sqrt{\frac{1-\tau_\pi}{\tau_\pi} \hl{\expect \|\hat{Q}_{k,a} - Q_{\#,a}\|_2^2}} \right\}^2 \right\} \Bigg]. \label{eq: CV BE bound}
    \end{align}
    where
    \begin{align*}
        \hat{\mathcal{R}}_k &:= \left[ \left\{ \frac{\hl{|I_k|}}{\hl{n} \sigma_{\#,a}^2} \frac{1-\tau_\pi}{\tau_\pi} \right\}^{1/3} + \left\{ \frac{\hl{|I_k|^2} M^2 z_\alpha^2}{\hl{n^2} \sigma_{\#,a}^2 \sigma_{\dagger,a}^2} \left( \frac{1-\tau_\pi}{\tau_\pi} \right)^2 \right\}^{1/3} \right] \hl{\left\{ \expect \|\hat{Q}_{k,a} - Q_{\#,a} \|_2^2 \right\}^{1/3}}, \\
        \hat{S}_k &:= \hl{\frac{|I_k|}{n}} \Bigg\{ \frac{1-\tau_\pi}{\tau_\pi} \hl{\expect \| \hat{Q}_{k,a} - Q_{\#,a} \|_2^2} + 2 \frac{1-\tau_\pi}{\tau_\pi} \|Q_{\#,a} - Q_{*,a}\|_2 \hl{\expect \|\hat{Q}_{k,a} - Q_{\#,a}\|_2} \\
        &\qquad\qquad+ \hl{\frac{1}{|I_k|}} \left\{ \sigma_{\#,a} + \sqrt{\frac{1-\tau_\pi}{\tau_\pi} \hl{\expect \|\hat{Q}_{k,a} - Q_{\#,a}\|_2^2}} \right\}^2 \Bigg\}.
    \end{align*}
    Moreover, under Condition~\ref{cond: CV common mean Q or consistent Q}, $\hat{S}_k$ can be replaced by
    \begin{align*}
        \hat{S}_k' := \hl{\frac{|I_k|}{n}} \Bigg\{ \frac{1-\tau_\pi}{\tau_\pi} \expect \| \hat{Q}_{k,a} - Q_{\#,a} \|_2^2 + \hl{\frac{1}{|I_k|}} \left\{ \sigma_{\#,a} + \sqrt{\frac{1-\tau_\pi}{\tau_\pi} \hl{\expect \|\hat{Q}_{k,a} - Q_{\#,a}\|_2^2}} \right\}^2 \Bigg\};
    \end{align*}
    under approximate sub-Weibull conditions \ref{cond: CV tail} and \ref{cond: more CV tail}, $\mathcal{R}_k$ can be replaced by
    \begin{align*}
        \hat{\mathcal{R}}_k' &:= \frac{1}{\sigma_{\#,a}} \left\{ 2 [(q+\hat{d}_1) \hl{\log n}]^{(\hat{q}_1+2)/(2 \hat{q}_1)} \hat{c}_1^{-1/\hat{q}_1} \hl{\hat{r}_1(n)} \sqrt{\frac{\tau_\pi}{1-\tau_\pi} \hl{\frac{|I_k|}{n}} } \right\} \vee \left\{ \frac{8}{3} M \frac{1-\tau_\pi}{\tau_\pi} \frac{(q+\hat{d}_1) \hl{\log n}}{\hl{\sqrt{n}}} \right\} \\
        &\qquad+ \left[ \frac{(q+\hat{d}_2) \hl{\log n}}{\hat{c}_2}\right]^{1/\hat{q}_2} \frac{\hl{|I_k|} |z_\alpha| }{\hl{n} \sigma_{\dagger,a}} \hl{\hat{r}_2(n)} + \frac{\{ 2 \hat{a}_1 + \ind(z_\alpha \neq 0) \hat{a}_2 \} \hl{n^{-q}} + \{2 \hat{b}_1 + \ind(z_\alpha \neq 0) \hat{b}_2 \} \hl{\sqrt{\log n/n}}}{\sigma_{\#,a}}.
    \end{align*}
\end{theorem}

The next theorem is a non-asymptotic version of Theorem~\ref{thm: asymptotic CV expect Var} and Part~\ref{prop: asymptotic Var bias 1} of Proposition~\ref{prop: asymptotic Var bias}.
Recall the function $f$ defined in Proposition~\ref{prop: asymptotic Var bias}.
Theorem~\ref{thm: asymptotic CV expect Var} follows by noting that $\expect \|\hat{Q}_{k,a} - Q_{\#,a} \|_2 \leq \sqrt{\expect \|\hat{Q}_{k,a} - Q_{\#,a} \|_2^2}$ and $\Var(\hat{\psi}_{k,a}) = \bigO(n^{-1})$ by Lemma~\ref{lemma: bound CV Var psi} under the setting of Theorem~\ref{thm: asymptotic CV expect Var}.
Part~\ref{prop: asymptotic Var bias 1} of Proposition~\ref{prop: asymptotic Var bias} follows by Lemma~\ref{lemma: bound CV Var psi}.

\begin{theorem}[Bound on $|\sigma_{\dagger,a} - \sigma_{\#,a}|$ with cross-fitting] \label{thm: CV expect Var}
    Consider the cross-fit procedure. Under Conditions~\ref{cond: 2nd moment} and \ref{cond: positivity}, $|\sigma_{\dagger,a} - \sigma_{\#,a}| \leq \hl{\sum_{k=1}^K} \hl{\hat{S}_k}/\sigma_{\#,a}$.
    Additionally under Condition~\ref{cond: CV common mean Q or consistent Q}, 
    \begin{align}
        & \sigma_{\dagger,a}^2 - \sigma_{\#,a}^2 = \hl{\sum_{k=1}^K} \hl{\frac{|I_k|}{n}} \left\{ \expect \left[ P_* \frac{1-\pi_*(a \mid \cdot)}{\pi_*(a \mid \cdot)} \hl{(\hat{Q}_{k,a} - Q_{\#,a})^2} \right] - \hl{\Var(\hat{\psi}_{k,a})} \right\}, \label{eq: CV Var bias} \\
        & \sigma_{\dagger}^2 - \sigma_{\#}^2 = \hl{\sum_{k=1}^K} \hl{\frac{|I_k|}{n}} \left\{ \expect \left[ P_* \left\{ f \hl{(\hat{Q}_{k,1} - Q_{\#,1})} + \frac{1}{f} \hl{(\hat{Q}_{k,0} - Q_{\#,0})} \right\}^2 \right] - \hl{\Var(\hat{\psi}_{k})} \right\}, \label{eq: ATE CV Var bias} \\
    \end{align}
    and $|\sigma_{\dagger,a} - \sigma_{\#,a}| \leq \sum_{k=1}^K \hl{\hat{S}_k'}/\sigma_{\#,a}$.
\end{theorem}

The next theorem is a non-asymptotic version of Theorem~\ref{thm: asymptotic BE bound} for $\psi_{*,a}$.
It also implies Theorem~\ref{thm: asymptotic BE bound} by taking $q = 1$ and $\delta'=1$, and noting that $\sigma_{\dagger,a}$ is bounded away from zero as $n \to \infty$ by Theorem~\ref{thm: expect Var}.
Define constants $F^{H_a} := 2 M (1 - \tau_\pi)/\tau_\pi$, $F^{\transform_a^2} := 8 M^2 (1-\tau_\pi)/\tau_\pi^2$, and $\varrho_{\#,a} := P_* | \transform_a(Q_{\#,a})^2 - P_* \transform_a(Q_{\#,a})^2 |^3$.

\begin{theorem}[Non-asymptotic Berry-Esseen-type bound for non-cross-fit AIPW Wald-CI] \label{thm: BE bound}
    Suppose that $\sigma_{\dagger,a}>0$, i.e., $\Pr(\tilde{\sigma}_a>0) > 0$.
    Under Conditions~\ref{cond: 2nd moment}--\ref{cond: higher moments} and \ref{cond: donsker}, for any constants $\delta, \delta'>0$ and $q \geq 1$, 
    \begin{align}
        &\left| \Pr(\sqrt{n} (\tilde{\psi}_a - \psi_{*,a}) \leq z_\alpha \tilde{\sigma}_a) - (1-\alpha) - \phi(z_\alpha) z_\alpha \frac{\sigma_{\dagger,a}-\sigma_{\#,a}}{\sigma_{\#,a}} \right| \nonumber \\
        &\leq \const \frac{\rho_{\#,a}}{\sigma_{\#,a}^3 \hl{\sqrt{n}}} + \const \ind(z_\alpha \neq 0) \frac{\varrho_{\#,a}}{\varsigma_{\#,a}^3 \hl{\sqrt{n}}} + \const \hl{n^{-q}} + \hl{\tilde{\mathcal{R}}(\delta,\delta')} \nonumber \\
        &\quad+ \frac{\const}{\sigma_{\#,a}^4} \Bigg[ \frac{(F^{\transform_a^2})^2 J^2(2 \delta', \funclass, M)}{\hl{n}} + \frac{(F^{\transform_a^2})^2 J^4(2 \delta',\funclass,M)}{(\delta')^4 \hl{n^2}} + \left( \frac{1-\tau_\pi}{\tau_\pi} \right)^2 \hl{(\expect \|\tilde{Q}_a - Q_{\#,a} \|_2^2)^2} \nonumber \\
        &\qquad\qquad + \left( \frac{1-\tau_\pi}{\tau_\pi} \right)^2 \| Q_{\#,a} - Q_{*,a} \|_2^2 \hl{(\expect \|\tilde{Q}_a - Q_{\#,a} \|_2)^2} + \frac{\sigma_{\#,a}^4 + (F^{H_a})^4 J^4(2,\funclass,M)}{\hl{n^2}} \Bigg] \nonumber \\
        &\quad+ \const \frac{|z_\alpha|}{\sigma_{\#,a} \sigma_{\dagger,a}} \Bigg[ \frac{\sigma_{\#,a}^2 q \hl{\log n}}{\hl{n}} + (|\psi_{*,a}| \sigma_{\#,a} + \varsigma_{\#,a}) \sqrt{\frac{q \hl{\log n}}{\hl{n}}} + \frac{F^{H_a} \sigma_{\#,a} \sqrt{q \hl{\log n}}}{\hl{n}} \nonumber \\
        &\qquad\qquad+ \frac{(F^{H_a})^2 \{ J^2(2,\funclass,M) + q^2 \hl{(\log n)^2} \}}{\hl{n}} + \frac{|\psi_{*,a}| F^{H_a} \{J(2,\funclass,M) + q \hl{\log n} \}}{\hl{\sqrt{n}}} \nonumber \\
        &\qquad\qquad+ \left( F^{H_a} + \frac{F^{\transform_a^2}}{\hl{\sqrt{n}}} \right) \hl{J(2 \delta, \funclass, M)} + \frac{(F^{H_a} + F^{\transform_a^2} \hl{n^{-1/2}}) \hl{J^2(2 \delta,\funclass,M)}}{\hl{\delta^2 \sqrt{n}}} \nonumber \\
        &\qquad\qquad+ q \left( F^{H_a} + \frac{F^{\transform_a^2}}{\hl{\sqrt{n}}} \right) \hl{\left( \delta + \frac{1}{\sqrt{n}} \right) \log n} \Bigg]. \label{eq: BE bound}
    \end{align}
    where
    \begin{align*}
        \tilde{\mathcal{R}}(\delta,\delta') &:= \frac{\hl{\expect \| \tilde{Q}_a-Q_{\#,a} \|_2^2}}{4 \hl{\delta^2} M^2} + \const \left\{ \frac{16 M^2 (1-\tau_\pi)^2 z_\alpha^2}{\tau_\pi^2 \sigma_{\#,a}^2 \sigma_{\dagger,a}^2} \hl{\expect \| \tilde{Q}_a-Q_{\#,a} \|_2^2} \right\}^{1/3} \\
        &\quad+ \const \frac{(F^{\transform_a^2})^2 J^2(2, \funclass, M) \hl{(\expect \|\tilde{Q}_a - Q_{\#,a} \|_2^2)^2}}{\sigma_{\#,a}^4 (\delta')^4 M^4 n}.
    \end{align*}
    Additionally under approximate sub-Weibull conditions \ref{cond: tail} and \ref{cond: more tail}, \eqref{eq: BE bound} holds with $\tilde{\mathcal{R}}(\delta,\delta')$ replaced by
    \begin{align*}
        \tilde{\mathcal{R}}'(\delta,\delta') &:= \tilde{a}_1 n^{\tilde{d}_1} \exp \left\{ - \tilde{c}_1 \left( \frac{2 \hl{\delta} M}{\hl{\tilde{r}_1(n)}} \right)^{\tilde{q}_1} \right\} + \tilde{b}_1 \hl{n^{-1/2}  \log n} + \const \frac{|z_\alpha|}{\sigma_{\#,a} \sigma_{\dagger,a}} \hl{\tilde{r}_2(n)} \left( \frac{q+\tilde{d}_2}{\tilde{c}_2} \hl{\log n} \right)^{1/\tilde{q}_2} \\
        &\qquad+ \const \frac{(F^{\transform_a^2})^2 J^2(2, \funclass, M)}{\sigma_{\#,a}^4 \hl{n}} \left[ \tilde{a}_1^2 \exp \left\{ - 2\tilde{c}_1 \left( \frac{2 \delta' M}{\hl{\tilde{r}_1(n)}} \right)^{\tilde{q}_1} \right\} + \tilde{b}_1^2\hl{n^{-1} (\log n)^2} \right].
    \end{align*}
    Under Condition~\ref{cond: mean Q or consistent Q}, the term $\{ (1-\tau_\pi)/\tau_\pi \}^2 \| Q_{\#,a} - Q_{*,a} \|_2^2 \hl{(\expect \|\tilde{Q}_a - Q_{\#,a} \|_2)^2}$ can de dropped from the right-hand side of \eqref{eq: BE bound}.
\end{theorem}

The next theorem is a non-asymptotic version of Theorem~\ref{thm: asymptotic expect Var} and Part~\ref{prop: asymptotic Var bias 2} of Proposition~\ref{prop: asymptotic Var bias}.
It implies Theorem~\ref{thm: asymptotic expect Var} by taking $\delta=1$ and noting that $\expect \| \tilde{Q}_a - Q_{\#,a} \|_2 \leq \sqrt{\expect \| \tilde{Q}_a - Q_{\#,a} \|_2^2}$ and $\Var(\tilde{\psi}_a) = \bigO(n^{-1})$ by Lemma~\ref{lemma: bound Var tail psi} under the setting of Theorem~\ref{thm: asymptotic expect Var}.
The $\bigO(n^{-1/2})$ term might be able to be improved to $\smallo(n^{-1/2})$ or a more precise rate depending on the covering number of $\funclass$, by taking $\delta \to 0$ at a suitable rate.
Part~\ref{prop: asymptotic Var bias 2} of Proposition~\ref{prop: asymptotic Var bias} follows by noting that $\mathscr{B}$ in Theorem~\ref{thm: expect Var} is $\bigO(n^{-1/2})$ and Lemma~\ref{lemma: bound Var tail psi}.

\begin{theorem}[Bound on $|\sigma_{\dagger,a} - \sigma_{\#,a}|$ without cross-fitting] \label{thm: expect Var}
    Under Conditions~\ref{cond: 2nd moment}, \ref{cond: positivity}, \ref{cond: bounded Q} and \ref{cond: donsker}, for any $\delta > 0$,
    \begin{align}
        &\Bigg| \sigma_{\dagger,a}^2-\sigma_{\#,a}^2 \nonumber \\
        &- \left[ P_* \left\{ \frac{1-\pi_*(a \mid \cdot)}{\pi_*(a \mid \cdot)} \expect_{\tilde{Q}_a} [\hl{(\tilde{Q}_a-Q_{\#,a})^2}] \right\} + 2 P_* \left\{ \frac{1-\pi_*(a \mid \cdot)}{\pi_*(a \mid \cdot)} (Q_{\#,a}-Q_{*,a}) \expect_{\tilde{Q}_a}[\hl{\tilde{Q}_a - Q_{\#,a}}] \right\} - \hl{\Var(\tilde{\psi}_a)} \right] \Bigg| \nonumber \\
        &\leq \mathscr{B}, \label{eq: Var bias}
    \end{align}
    where
    \begin{align*}
        \mathscr{B} := \const \frac{F^{\transform_a^2} \hl{J(2 \delta, \funclass, M)}}{\hl{\sqrt{n}}} + \const \frac{F^{\transform_a^2} \hl{J^2(2 \delta,\funclass,M)}}{\hl{\delta^2 n}} + \const \frac{F^{\transform_a^2} J(2, \funclass, M) \hl{\expect \|\tilde{Q}_a - Q_{\#,a} \|_2^2}}{\sigma_{\#,a} \hl{\delta^2} M^2 \hl{\sqrt{n}}}.
    \end{align*}
    and
    \begin{align}
        |\sigma_{\dagger,a} - \sigma_{\#,a}| &\leq \frac{\mathscr{B}}{\sigma_{\#,a}} + \frac{1-\tau_\pi}{\tau_\pi \sigma_{\#,a}} \hl{\expect \|\tilde{Q}_a - Q_{\#,a} \|_2^2} + 2 \frac{1-\tau_\pi}{\tau_\pi \sigma_{\#,a}} \| Q_{\#,a} - Q_{*,a} \|_2 \hl{\expect \|\tilde{Q}_a - Q_{\#,a} \|_2} \nonumber \\
        &\quad+ \frac{\const}{\sigma_{\#,a} \hl{n}} \left\{ \sigma_{\#,a}^2 + (F^{H_a})^2 J^2(2,\funclass,M) \right\}. \label{eq: sd bias}
    \end{align}
    Under Condition~\ref{cond: mean Q or consistent Q}, the terms $2 P_* \left\{ \frac{1-\pi_*(a \mid \cdot)}{\pi_*(a \mid \cdot)} (Q_{\#,a}-Q_{*,a}) \expect_{\tilde{Q}_a}[\hl{\tilde{Q}_a - Q_{\#,a}}] \right\}$ and $2 \frac{1-\tau_\pi}{\tau_\pi \sigma_{\#,a}} \| Q_{\#,a} - Q_{*,a} \|_2 \allowbreak \hl{\expect \|\tilde{Q}_a -Q_{\#,a} \|_2}$ can be dropped from \eqref{eq: Var bias} and \eqref{eq: sd bias}, respectively.
    Under approximate sub-Weibull condition \ref{cond: tail}, the third term $\const \frac{F^{\transform_a^2} J(2, \funclass, M) \hl{\expect \|\tilde{Q}_a - Q_{\#,a} \|_2^2}}{\sigma_{\#,a} \hl{\delta^2} M^2 \hl{\sqrt{n}}}$ in $\mathscr{B}$ can be replaced by
    $$\frac{\const F^{\transform_a^2} J(2, \funclass, M)}{\sigma_{\#,a} \hl{\sqrt{n}}} \left[ \tilde{a}_1 \hl{n^{\tilde{d}_1}} \exp \left\{ - \tilde{c}_1 \left( \frac{2 \hl{\delta} M}{\hl{\tilde{r}_1(n)}} \right)^{\tilde{q}_1} \right\} + \tilde{b}_1 \hl{n^{-1/2} \log n} \right].$$
    Similarly, with $f$ defined in Proposition~\ref{prop: asymptotic Var bias}, under Conditions~\ref{cond: 2nd moment}, \ref{cond: positivity}, \ref{cond: bounded Q}, \ref{cond: donsker}, and \ref{cond: mean Q or consistent Q},
    \begin{align}
        &\Bigg| \sigma_{\dagger}^2 - \sigma_{\#}^2 \\
        &\quad- \Bigg\{ \expect \left[ P_* \left\{ f \hl{(\tilde{Q}_1-Q_{\#,1})} + \frac{1}{f} \hl{(\tilde{Q}_0-Q_{\#,0})} \right\}^2 \right] + 2 \sum_{a=0}^1 P_* \left\{ \frac{1}{\pi_*(a \mid \cdot)} (Q_{\#,a}-Q_{*,a}) \expect_{\tilde{Q}_a}[\hl{\tilde{Q}_a - Q_{\#,a}}] \right\} \\
        &\qquad\qquad- \hl{\Var(\tilde{\psi})} \Bigg\} \Bigg| \leq \mathscr{B}. \label{eq: ATE Var bias}
    \end{align}
\end{theorem}

\section{Series regression: an example of outcome model estimators with approximate sub-Weibull tails} \label{sec: light tail example}

In this section, I provide an example of outcome model estimators satisfying approximate sub-Weibull conditions with constants uniform over $n$.
The bounds used may be overly conservative, so approximate sub-Weibull conditions may hold more generally.
The notation introduced in this section is used only within this section.
Consider observing an i.i.d. sample $(X_i,Y_i)_{i=1}^n$ and estimating $Q_*: x \mapsto \expect[Y \mid X=x]$ with series regression.
Let $\{\phi_k\}_{k=1}^\infty$ be a basis of functions chosen by the user and $\Phi_K: x \mapsto (\phi_1(x),\ldots,\phi_K(x))^\top$ be the vector of the first $K$ basis functions.
Usually, $K=K(n)$ grows to infinity at an appropriate rate slower than $n$.
In the special case of ordinary linear regression, $K$ is a constant.
The corresponding series estimator $\hat{Q}$ of $Q_*$ is $x \mapsto \Phi_K(x)^\top \hat{\beta}$, where $\hat{\beta} := \{P_n \Phi_K(X) \Phi_K(X)^\top\}^+ \{P_n \Phi_K(X) Y\}$ and the superscript ``$+$'' denotes the Moore-Penrose pseudo-inverse.
Let $Q_\#: x \mapsto \Phi_K(x)^\top \beta_\#$ where $\beta_\# := \{P_* \Phi_K(X) \Phi_K(X)^\top\}^+ \{P_* \Phi_K(X) Y\}$, so $Q_\#$ is the true best approximation to $Q_*$ with the first $K$ basis functions.
This scenario corresponds to estimating $Q_{*,a}$ with series regression among individuals with $A=a$.

Let $\lambda_{\max}$ and $\lambda_{\min}$ denote the maximum and minimum eigenvalues of a matrix, respectively.
Let $\| \cdot \|_{\ltwo}$ denote the $\ell_2$ norm of a vector.

\begin{seriescond} \label{seriescond: l2}
    There exists a sequence of constants $\{\xi_K\}_{K=1}^\infty$ and constant invertible $K \times K$ matrices $\{T_K\}_{K=1}^\infty$ such that $\sup_{x \in \mathcal{X}} \| T_K \Phi_K(x) \|_\ltwo \leq \xi_K$.
\end{seriescond}

\begin{seriescond} \label{seriescond: eigen}
    There exist constants $0 < \underline{\lambda} \leq \bar{\lambda} < \infty$ such that $\underline{\lambda} \leq \lambda_{\min}(P_* T_K \Phi_K(X) \Phi_K(X)^\top T_K^\top) \leq \lambda_{\max}(P_* T_K \Phi_K(X) \Phi_K(X)^\top T_K^\top) \leq \bar{\lambda}$ for all $K$.
\end{seriescond}

\begin{seriescond} \label{seriescond: growth rate}
    The following rates hold: $K/n = \smallo(1), \xi_K^2/K = \bigO(1)$.
\end{seriescond}

\begin{seriescond} \label{seriescond: bounded}
    $Y$ is bounded and $\Var(Y) \mid X$ is bounded away from zero.
\end{seriescond}

These conditions are common in the series regression literature \citepsupp[e.g.,][]{Newey1997}.
The first part of the growth rate condition \ref{seriescond: growth rate} is very weak because only the variance $\hat{Q}-Q_\#$ needs to be controlled but the bias $Q_\#-Q_*$ need not.
The second part can be restrictive, but it is satisfied by some common series, including trigonometric polynomials, splines, orthogonal wavelets \protect\citepsupp{Chen2007}.
The boundedness of $Y$ in SeriesCond~\ref{seriescond: bounded} ensures sub-Gaussianity. It is generally mild and might be able to be relaxed to light-tail distributions.

\begin{proposition} \label{prop: series tail}
    There exist constants $a,b,c>0$ independent of $(K,n)$ such that, for all $t>0$,
    $$\Pr \left( \frac{\| \hat{Q} - Q_\# \|_{2}}{K/n} > t \right) \leq a \exp \left( -c t \right) + b \sqrt{\log n/n}.$$
\end{proposition}

Here, $r(n)$ is the convergence rate of $\| \hat{Q} - Q_\# \|_{2}$.
Since the sample size for estimating the outcome model has a tiny probability (decaying exponentially with $n$) to be insufficient, Conditions~\ref{cond: CV tail} and Condition~\ref{cond: tail} hold.
By Lemma~\ref{lemma: transform sqaure} and the boundedness of $\hat{Q}$, the tail of $P_* \transform_a(\hat{Q})^2$ is similar to $\| \hat{Q} - Q_{\#,a} \|_2$.
Thus, Conditions~\ref{cond: more CV tail} and \ref{cond: more tail} also hold.

\section{Impact of the choice of $K$ on the rates in Theorem~\ref{thm: asymptotic CV BE bound}} \label{sec: K choice}

For the cross-fit procedure, the choice of the number of folds $K$ in cross-fitting under exchangeable sample-splitting (Condition~\ref{cond: CV exchangeable}) is an interesting question.
If the impact of $K$ on $\expect \| \hat{Q}_{k,a} - Q_{\#,a}\|_2^2$ is ignored, Theorem~\ref{thm: asymptotic CV BE bound} suggests choosing a small $K$ because the bound increases with $K$.
However, a small $K$ would generally lead to a larger $\expect \| \hat{Q}_{k,a} - Q_{\#,a}\|_2^2$ because of fewer training data.
Without additional knowledge or assumptions on how $\expect \| \hat{Q}_{k,a} - Q_{\#,a}\|_2^2$ varies with training sample size $n(K-1)/K$, it appears impossible to obtain concrete theoretical guidance on the choice of $K$.
Suppose that $\expect \| \hat{Q}_{k,a} - Q_{\#,a}\|_2^2 = \bigO(\{n (K-1)/K\}^{-r})$ for some constant $r > 0$.
For example, if $\hat{Q}_{k,a}$ is estimated in a parametric model, $r$ can be anticipated to be one;
if the usual $\smallo_p(n^{-1/4})$-rate requirement on $\hat{Q}_{k,a}$ is achieved, $r$ can be anticipated to be larger than 1/2.
The resulting rate in Theorem~\ref{thm: asymptotic CV BE bound} is $n^{-r/3} K^{(r+2)/3} (K-1)^{-r/3}$, or $n^{-r/2} K^{(r+1)/2} (K-1)^{-r/2}$ (up to log factors) under approximate sub-Weibull conditions \ref{cond: CV tail} and \ref{cond: more CV tail}.
Figure~\ref{fig: K rate} illustrates the multiplicative factor involving $K$ in the rate for a range of $r$ and $K$, and suggests that a small $K$ is generally preferable.

This result differs from the practical recommendation in \protect\citetsupp{Phillips2023}, where a larger $K$ like 20 is recommended in smaller samples.
Because I use a crude union bound to analyze cross-fitting, the rate in Theorem~\ref{thm: asymptotic CV BE bound} might be suboptimal regarding $K$. In particular, I might exaggerate the loss in the rate due to a large $K$, so the above analysis might not accurately characterize the impact of $K$ in cross-fitting.
A sound theoretical justification on the choice of $K$ is still an open question.

\begin{figure}[tbh]
    \centering
    \includegraphics[width=0.9\linewidth]{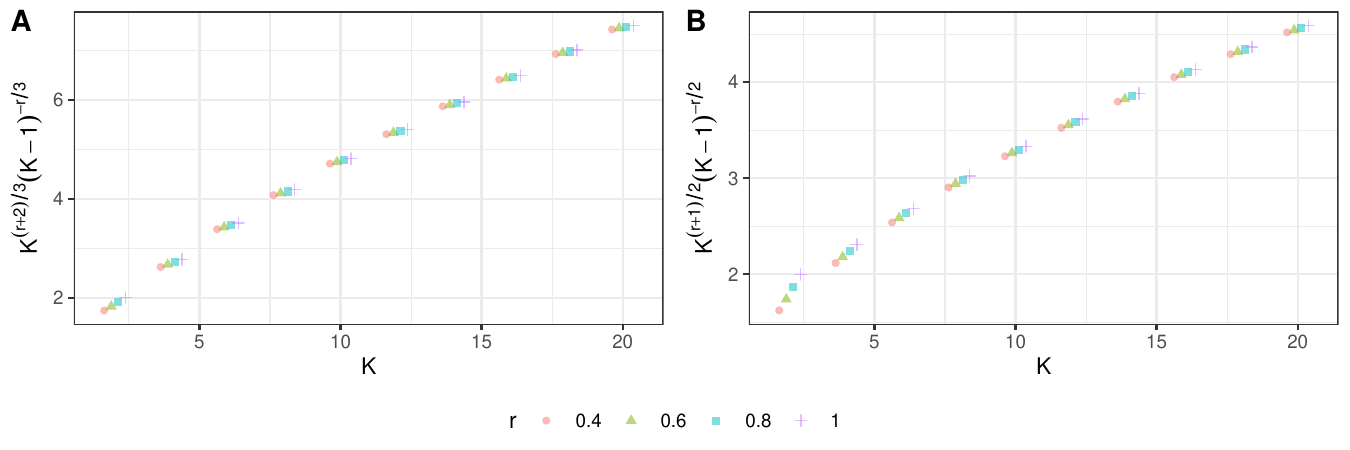}
    \caption{The factor concerning $K$ in the cross-fit Wald-CI coverage's convergence rate to the nominal coverage. The number of folds $K$ ranges over even numbers between 2 and 20. The rate $r$ ranges from 0.4 (slow) to 1 (parametric). Subfigures~(A) and (B) are for the cases with and without approximate sub-Weibull conditions \ref{cond: CV tail} and \ref{cond: more CV tail}, respectively. The points are dodged horizontally.}
    \label{fig: K rate}
\end{figure}

\section{Additional simulation details} \label{sec: sim2}

For \texttt{misSL} and \texttt{CVmisSL}, the Super Learner library for estimating $(x,a) \mapsto Q_{*,a}(x)$ contains the following learners:
\begin{itemize}
    \item zero function $(x,a) \mapsto 0$
    \item marginal mean outcome $(x,a) \mapsto \bar{Y}$
    \item linear regression with predictors $(A,X)$
    \item linear regression with predictors $(A,X)$ and all first-order interactions between predictors
    \item lasso linear regression with predictors $(A,X)$, with the pevalty tuning parameter selected via cross-validation.
\end{itemize}
For \texttt{SL} and \texttt{CVSL}, the library further contains highly adaptive lasso as implemented in R package \texttt{hal9001} \protect\citepsupp{Hejazi2020}, with \texttt{max\textunderscore degree}$\in \{2,3\}$, \texttt{smoothness\textunderscore orders}$\in \{0,1\}$, and knots automatically selected by the function \texttt{num\textunderscore knots\textunderscore generator}.

To calculate $\sigma_{\#,a}^2$, I first generate an independent test data of size $10^4$.
In each simulation run, I compute $\hat{Q}_{k,a}(x)$ or $\tilde{Q}_a(x)$ for $x$ in the test data, i.e., predict the outcome with the estimated outcome model.
For each $x$ in the test data, I calculate the average of these predictions across simulation runs, which approximates $Q_{\#,a}(x)$.
I finally approximate oracle scaled variance by the variance of the influence function based on $Q_{\#,a}$ over the test data.
The test data size $10^4$ is chosen to accommodate the limited computer memory. Given more memory, this size could be larger (e.g., $10^6$) to achieve more accuracy in Monte Carlo estimation.

\section{Proof} \label{sec: proof}

I focus on the proof of the results for estimating the mean counterfactual outcome $\psi_{*,a}$.
The results for the ATE $\psi_*=\psi_{*,1}-\psi_{*,0}$ can be derived similarly.
The proof of the variance estimators' bias is based on direct calculation, decomposing the bias into terms concerning different estimation errors, and applying non-asymptotic bounds to each term.
The proof of the Berry-Esseen-type bounds follows this road map: (1) use Lemma~\ref{lemma: delta} to upper- and lower-bound the coverage probability with terms involving finite-sample variance approximations $\sigma_{\#,a}$, deviation of the estimator from its linear expansion, and deviation of the variance estimator from its finite-sample approximation, with the help of a slacking variable $\eta$, (2) bound each deviation separately using non-asymptotic bounds, including concentration inequalities and Berry-Esseen Theorem, (3) use Berry-Esseen Theorem to bound the deviation of the linear expansion from the limiting normal distribution, (4) combine all above bounds to obtain a bound involving the slacking variable, and (5) choose an appropriate slacking variable to minimize or nearly minimize the bound.

\subsection{General preliminary lemmas} \label{sec: general lemmas}

The following lemma is a more general version of the delta method for CI coverage \protect\citepsupp[e.g., Section~2.7 in][]{Hall1992}.

\begin{lemma} \label{lemma: delta}
    For any random variables $A,B,t,\eta$ ($\eta \geq 0$),
    \begin{align*}
        \Pr(B \leq t - \eta) - \Pr(|A-B| > \eta) \leq \Pr(A \leq t) \leq \Pr(B \leq t+ \eta) + \Pr(|A-B| > \eta).
    \end{align*}
\end{lemma}
\begin{proof}
\begin{align*}
    \Pr(A \leq t) &= \Pr(A \leq t, |A-B| \leq \eta) + \Pr(A \leq t, |A-B| > \eta) \\
    &\leq \Pr(B \leq t + \eta) + \Pr(|A-B| > \eta).
\end{align*}
Thus, the second inequality holds.

Switch the role of $A$ and $B$:
$$\Pr(B \leq t) \leq \Pr(A \leq t+\eta) + \Pr(|A-B| > \eta),$$
that is
$$\Pr(B \leq t) - \Pr(|A-B| > \eta) \leq \Pr(A \leq t+\eta).$$
Since $t$ and $\eta$ are arbitrary, replacing $t$ with $t-\eta$ yields
$$\Pr(B \leq t - \eta) - \Pr(|A-B| > \eta) \leq \Pr(A \leq t).$$
Thus, the first inequality holds.
\end{proof}

The next lemma is an application of Lemma~\ref{lemma: delta} to the AIPW estimator.
\begin{lemma} \label{lemma: AIPW delta}
    Let $\alpha \in (0,1)$ and $\eta \geq 0$ be any constants. For the non-cross-fit AIPW estimator $\tilde{\psi}_a$,
    \begin{align*}
        &\Pr(\sqrt{n} P_n D_a(Q_{\#,a},\psi_{*,a}) \leq z_\alpha \sigma_{\dagger,a} - \eta) - \Pr(|\sqrt{n} \{ \tilde{\psi}_a - \psi_{*,a} - P_n D_a(Q_{\#,a},\psi_{*,a}) \} - z_\alpha (\tilde{\sigma}_a - \sigma_{\dagger,a})| > \eta) \\
        &\leq \Pr(\sqrt{n} (\tilde{\psi}_a - \psi_{*,a}) \leq z_\alpha \tilde{\sigma}_a) \\
        &\leq \Pr(\sqrt{n} P_n D_a(Q_{\#,a},\psi_{*,a}) \leq z_\alpha \sigma_{\dagger,a} + \eta) + \Pr(|\sqrt{n} \{ \tilde{\psi}_a - \psi_{*,a} - P_n D_a(Q_{\#,a},\psi_{*,a}) \} - z_\alpha (\tilde{\sigma}_a - \sigma_{\dagger,a})| > \eta)
    \end{align*}
    The same result holds for the cross-fit estimator $\hat{\psi}_a$ with $(\tilde{\psi}_a,\tilde{\sigma}_a)$ replaced by $(\hat{\psi}_a,\hat{\sigma}_a)$.
\end{lemma}
\begin{proof}
    For the non-cross-fit version, taking $A=\sqrt{n}(\tilde{\psi}_a - \psi_{*,a}) - z_\alpha \tilde{\sigma}_a$, $B=\sqrt{n} P_n D_a(Q_{\#,a},\psi_{*,a}) - z_\alpha \sigma_{\dagger,a}$, and $t=0$ in Lemma~\ref{lemma: delta} yields the desired result. The cross-fit version can be shown similarly.
\end{proof}

The next lemma provides further bounds on the bounds in Lemma~\ref{lemma: AIPW delta}.
\begin{lemma} \label{lemma: normal quantile}
    Under Conditions~\ref{cond: 2nd moment} and \ref{cond: 3rd moment},
    \begin{align*}
        \Pr(\sqrt{n} P_n D_a(Q_{\#,a},\psi_{*,a}) \leq z_\alpha \sigma_{\dagger,a} + \eta) &\leq 1-\alpha + \const \frac{\rho_{\#,a}}{\sigma_{\#,a}^3 \sqrt{n}} + \phi(z_\alpha) z_\alpha \frac{\sigma_{\dagger,a}-\sigma_{\#,a}}{\sigma_{\#,a}} + \const \frac{(\sigma_{\dagger,a}-\sigma_{\#,a})^2}{\sigma_{\#,a}^2} + \const \frac{\eta}{\sigma_{\#,a}}, \\
        \Pr(\sqrt{n} P_n D_a(Q_{\#,a},\psi_{*,a}) \leq z_\alpha \sigma_{\dagger,a} - \eta) &\geq 1-\alpha - \const \frac{\rho_{\#,a}}{\sigma_{\#,a}^3 \sqrt{n}} + \phi(z_\alpha) z_\alpha \frac{\sigma_{\dagger,a}-\sigma_{\#,a}}{\sigma_{\#,a}} - \const \frac{(\sigma_{\dagger,a}-\sigma_{\#,a})^2}{\sigma_{\#,a}^2} - \const \frac{\eta}{\sigma_{\#,a}}.
    \end{align*}
\end{lemma}
\begin{proof}
    Let $\Phi$ denote the cumulative distribution function of standard Gaussian. Note that $P_n D_a(Q_{\#,a},\psi_{*,a})$ is a sample mean of i.i.d. random variables, $P_* D_a(Q_{\#,a},\psi_{*,a}) =0$, $\rho_{\#,a} = P_* |D_a(Q_{\#,a},\psi_{*,a})|^3 < \infty$, $\sigma_{\#,a}^2 < \infty$ because $\rho_{\#,a}<\infty$, and $\sigma_{\#,a}^2 \geq \sigma_{*,a}^2 > 0$ because $D_a(Q_{*,a},\psi_{*,a})$ is the efficient influence function. By the Berry-Esseen Theorem,
    \begin{align*}
        &\Pr(\sqrt{n} P_n D_a(Q_{\#,a},\psi_{*,a}) \leq z_\alpha \sigma_{\dagger,a} + \eta) \\
        &= \Pr \left( \sqrt{n} \frac{P_n D_a(Q_{\#,a},\psi_{*,a})}{\sigma_{\#,a}} \leq z_\alpha + \frac{z_\alpha (\sigma_{\dagger,a} - \sigma_{\#,a}) + \eta}{\sigma_{\#,a}} \right) \\
        &\leq \Phi \left( z_\alpha + \frac{z_\alpha (\sigma_{\dagger,a} - \sigma_{\#,a}) + \eta}{\sigma_{\#,a}} \right) + \const \frac{\rho_{\#,a}}{\sigma_{\#,a}^3 \sqrt{n}} \\
        &= 1-\alpha + \const \frac{\rho_{\#,a}}{\sigma_{\#,a}^3 \sqrt{n}} + \Phi \left( z_\alpha + \frac{z_\alpha (\sigma_{\dagger,a} - \sigma_{\#,a})}{\sigma_{\#,a}} + \frac{\eta}{\sigma_{\#,a}} \right) - \Phi(z_\alpha) \\
        &\leq 1-\alpha + \const \frac{\rho_{\#,a}}{\sigma_{\#,a}^3 \sqrt{n}} + \phi(z_\alpha) z_\alpha \frac{\sigma_{\dagger,a} - \sigma_{\#,a}}{\sigma_{\#,a}} + \const \frac{(\sigma_{\dagger,a}-\sigma_{\#,a})^2}{\sigma_{\#,a}^2} + \const \frac{\eta}{\sigma_{\#,a}},
    \end{align*}
    where the last step follows from Taylor's expansion of $\Phi$ at $z_\alpha$, the positivity of $\phi$, and the boundedness of the first two derivatives of $\Phi$.
    The second inequality is similar.
\end{proof}

In all the following proof, I define $c/0 := \infty$ for any $c>0$, and $\infty \cdot 0 :=0$;
for any $a,b \in \real$, I use $a \wedge b$ to denote $\min \{a,b\}$.

\begin{lemma} \label{lemma: square diff inequality}
    Let $a,b \in \real$, $c > 0$. If $|a^2-b^2|>c$, then $|a-b| > \sqrt{b^2 + c} - |b| \geq (\sqrt{2}-1) \{\sqrt{c} \wedge c/|b|\}$.
\end{lemma}
\begin{proof}
    By triangle inequality, $c < |a^2-b^2| = |a+b| |a-b| \leq (|a-b|+2|b|) |a-b| = |a-b|^2 + 2|b| |a-b|$. Solving this quadratic inequality for $|a-b|$ yields $|a-b| > \sqrt{b^2+c} - |b|$, noting that $|a-b| \geq 0$. Moreover,
    $$\sqrt{b^2+c} - |b| = \frac{c}{\sqrt{b^2+c} + |b|} \geq
    \begin{cases}
        \frac{c}{\sqrt{2 c}+\sqrt{c}} & (\text{if } |b| \leq \sqrt{c}) \\
        \frac{c}{\sqrt{2 b^2} + |b|} & (\text{if } |b| > \sqrt{c})
    \end{cases}
    \geq \frac{\sqrt{c} \wedge c/|b|}{\sqrt{2}+1} = (\sqrt{2}-1) \{\sqrt{c} \wedge c/|b|\}.$$
\end{proof}

For any function $Q$ and $a' \in \{0,1\}$, define function $H_{a'}(Q): v=(x,a,y) \mapsto \transform_{a'}(Q)(v) - \transform_{a'}(Q_{\#,a'})(v) = \{1- \ind(a=a')/\pi_*(a' \mid x)\} \{Q(x) - Q_{\#,a'}(x)\}$.
The next lemma provides the first and second moments of $H_a(Q)$.

\begin{lemma} \label{lemma: H mean var}
    For any function $Q$ and $a \in \{0,1\}$, $P_* H_a(Q) = 0$ and $P_* H_a(Q)^2 = P_* \frac{1-\pi_*(a \mid \cdot)}{\pi_*(a \mid \cdot)} (Q - Q_{\#,a})^2$, which lies in the interval $[\frac{\tau_\pi}{1-\tau_\pi} \| Q - Q_{\#,a} \|_{2}^2, \frac{1-\tau_\pi}{\tau_\pi} \| Q - Q_{\#,a} \|_{2}^2]$ under Condition~\ref{cond: positivity}.
\end{lemma}

\begin{proof}
    \begin{align*}
        P_* H_a(Q) &= \expect \left[ \left( 1 - \frac{\ind(A=a)}{\pi_*(a \mid X)} \right) \left\{ Q(X) - Q_{\#,a}(X) \right\} \right] \\
        &= \expect \left[ \expect \left[ \left( 1 - \frac{\ind(A=a)}{\pi_*(a \mid X)} \right) \left\{ Q(X) - Q_{\#,a}(X) \right\} \mid X \right] \right] = 0, \\
        P_* H_a(Q)^2 &= \expect \left[ \left( 1 - \frac{\ind(A=a)}{\pi_*(a \mid X)} \right)^2 \left\{ Q(X) - Q_{\#,a}(X) \right\}^2 \right] \\
        &= \expect \left[ \expect \left[ \left( 1 - \frac{\ind(A=a)}{\pi_*(a \mid X)} \right)^2 \left\{ Q(X) - Q_{\#,a}(X) \right\}^2 \mid X \right] \right] \\
        &= \expect \left[ \expect \left[ \left( 1 - 2 \frac{\ind(A=a)}{\pi_*(a \mid X)} + \frac{\ind(A=a)}{\pi_*(a \mid X)^2} \right) \left\{ Q(X) - Q_{\#,a}(X) \right\}^2 \mid X \right] \right] \\
        &= \expect \left[ \left( 1-2 + \frac{1}{\pi_*(a \mid X)} \right) \left\{ Q(X) - Q_{\#,a}(X) \right\}^2 \right] \\
        &= \expect \left[ \frac{1-\pi_*(a \mid X)}{\pi_*(a \mid X)} \left\{ Q(X) - Q_{\#,a}(X) \right\}^2 \right].
    \end{align*}
    The bounds of $P_* H_a(Q)^2$ follow immediately from Condition~\ref{cond: positivity}.
\end{proof}

The following lemma shows some results about $\transform_a(Q)^2$ for any bounded function $Q$ based on straightforward but tedious calculations.
\begin{lemma} \label{lemma: transform sqaure}
    For any function $Q$,
    \begin{align*}
        \transform_a(Q)(V)^2 &= \frac{\ind(A=a)}{\pi_*(a \mid X)^2} \{Y - Q(X)\}^2 + 2 \frac{\ind(A=a)}{\pi_*(a \mid X)} \{Y - Q(X)\} Q(X) + Q(X)^2 \\
        &= \frac{\ind(A=a)}{\pi_*(a \mid X)^2} Y^2 - 2 \frac{\ind(A=a) (1-\pi_*(a \mid X))}{\pi_*(a \mid X)^2} Y Q(X) \\
        &\quad+ \left\{ \frac{\ind(A=a)}{\pi_*(a \mid X)^2} - 2 \frac{\ind(A=a)}{\pi_*(a \mid X)} + 1 \right\} Q(X)^2, \\
        P_* \transform_a(Q)^2 &= P_* \frac{\expect[Y^2 \mid A=a,X=\cdot]}{\pi_*(a \mid \cdot)} - 2 P_* \frac{1-\pi_*(a \mid \cdot)}{\pi_*(a \mid \cdot)} Q_{*,a} Q + P_* \frac{1-\pi_*(a \mid \cdot)}{\pi_*(a \mid \cdot)} Q^2.
    \end{align*}
    Thus, for any random function $\hat{Q}$,
    \begin{align*}
        \expect_{\hat{Q}}[\transform_a(\hat{Q})(V)^2] &= \frac{\ind(A=a)}{\pi_*(a \mid X)^2} Y^2 - 2 \frac{\ind(A=a) (1-\pi_*(a \mid X))}{\pi_*(a \mid X)^2} Y \expect_{\hat{Q}}[\hat{Q}(X)] \\
        &\quad+ \left\{ \frac{\ind(A=a)}{\pi_*(a \mid X)^2} - 2 \frac{\ind(A=a)}{\pi_*(a \mid X)} + 1 \right\} \expect_{\hat{Q}}[\hat{Q}(X)^2], \\
        \expect \left[ P_* \{ \transform_a(\hat{Q})^2 - \transform_a(Q_{\#,a})^2 \} \right] &= \expect \left[ P_* \frac{1-\pi_*(a \mid \cdot)}{\pi_*(a \mid \cdot)} (\hat{Q} - Q_{\#,a})^2 \right] \\
        &\quad+ 2 \expect \left[ P_* \frac{1-\pi_*(a \mid \cdot)}{\pi_*(a \mid \cdot)} (\hat{Q} - Q_{\#,a}) (Q_{\#,a} - Q_{*,a}) \right], \\
        P_* \{\transform_a(\hat{Q})^2 - \expect_{\hat{Q}}[\transform_a(\hat{Q})^2]\} &= P_* \frac{1-\pi_*(a \mid \cdot)}{\pi_*(a \mid \cdot)} \left\{ -2 Q_{*,a} (\hat{Q} - \expect_{\hat{Q}}[\hat{Q}]) + \hat{Q}^2 - \expect_{\hat{Q}}[\hat{Q}^2] \right\}.
    \end{align*}
    If $\hat{Q}$ is bounded by constant $M$,
    \begin{align*}
        &\expect[\{ P_* (\transform_a(\hat{Q})^2 - \expect_{\hat{Q}}[\transform_a(\hat{Q})^2]) \}^2] \leq 16M^2 P_* \left( \left\{ \frac{1-\pi_*(a \mid \cdot)}{\pi_*(a \mid \cdot)} \right\}^2 \expect_{\hat{Q}} \left[ \left( \hat{Q} - Q_{\#,a} \right)^2 \right] \right).
    \end{align*}
    Additionally under Condition~\ref{cond: positivity}, this variance is further upper bounded by
    $$16M^2 \left( \frac{1-\tau_\pi}{\tau_\pi} \right)^2 \expect \| \hat{Q}-Q_{\#,a} \|_{2}^2.$$
    Similarly, for any functions $Q_1$ and $Q_0$,
    \begin{align*}
        &P_* \{ \transform_1(Q_1) -\transform_0(Q_0) \}^2 = P_* \{ \transform_1(Q_1)^2 + \transform_0(Q_0)^2 - 2 \transform_1(Q_1) \transform_0(Q_0) \} \\
        &= P_* \frac{\expect[Y^2 \mid A=1,X=\cdot]}{\pi_*(1 \mid \cdot)} - 2 P_* \frac{\pi_*(0 \mid \cdot)}{\pi_*(1 \mid \cdot)} Q_{*,1} Q_1 + P_* \frac{\pi_*(0 \mid \cdot)}{\pi_*(1 \mid \cdot)} Q_1^2 \\
        &\quad+ P_* \frac{\expect[Y^2 \mid A=0,X=\cdot]}{\pi_*(0 \mid \cdot)} - 2 P_* \frac{\pi_*(1 \mid \cdot)}{\pi_*(0 \mid \cdot)} Q_{*,0} Q_0 + P_* \frac{\pi_*(1 \mid \cdot)}{\pi_*(0 \mid \cdot)} Q_0^2 \\
        &\quad-2 P_* (Q_{*,1} Q_0 + Q_{*,0} Q_1 - Q_1 Q_0).
    \end{align*}
\end{lemma}
\begin{proof}
    The first five equations follow from direct calculation.
    Moreover, $\expect[P_* (\transform_a(\hat{Q})^2 - \expect_{\hat{Q}}[\transform_a(\hat{Q})^2])] = 0$ and
    \begin{align*}
        &\expect[\{ P_* (\transform_a(\hat{Q})^2 - \expect_{\hat{Q}}[\transform_a(\hat{Q})^2]) \}^2] \\
        &\leq \expect \left[ P_* \left\{ \frac{1-\pi_*(a \mid \cdot)}{\pi_*(a \mid \cdot)} \right\}^2 \left\{ -2 Q_{*,a} (\hat{Q} - \expect_{\hat{Q}}[\hat{Q}]) + \hat{Q}^2 - \expect_{\hat{Q}}[\hat{Q}^2] \right\}^2 \right].
    \end{align*}
    If $\hat{Q}$ is bounded by $M$, then
    \begin{align*}
        &\expect \left[ P_* \left\{ \frac{1-\pi_*(a \mid \cdot)}{\pi_*(a \mid \cdot)} \right\}^2 \left\{ -2 Q_{*,a} (\hat{Q} - \expect_{\hat{Q}}[\hat{Q}]) \right\}^2 \right] \\
        &= 4 P_* \left\{ \frac{1-\pi_*(a \mid \cdot)}{\pi_*(a \mid \cdot)} \right\}^2 Q_{*,a}^2 \expect_{\hat{Q}} \left[ (\hat{Q} - \expect_{\hat{Q}}[\hat{Q}])^2 \right] \\
        &\leq 4 M^2 P_* \left\{ \frac{1-\pi_*(a \mid \cdot)}{\pi_*(a \mid \cdot)} \right\}^2 \expect_{\hat{Q}} \left[ (\hat{Q} - Q_{\#,a} )^2 \right], \\
        &\expect \left[ P_* \left\{ \frac{1-\pi_*(a \mid \cdot)}{\pi_*(a \mid \cdot)} \right\}^2 \left\{ \hat{Q}^2 - \expect_{\hat{Q}}[\hat{Q}^2] \right\}^2 \right] \\
        &= P_* \left\{ \frac{1-\pi_*(a \mid \cdot)}{\pi_*(a \mid \cdot)} \right\}^2 \expect_{\hat{Q}} \left[ \left( \hat{Q}^2 - \expect_{\hat{Q}}[\hat{Q}^2] \right)^2 \right] \\
        &\leq P_* \left\{ \frac{1-\pi_*(a \mid \cdot)}{\pi_*(a \mid \cdot)} \right\}^2 \expect_{\hat{Q}} \left[ \left( \hat{Q}^2 - Q_{\#,a}^2 \right)^2 \right] \\
        &\leq 4 M^2 P_* \left\{ \frac{1-\pi_*(a \mid \cdot)}{\pi_*(a \mid \cdot)} \right\}^2 \expect_{\hat{Q}} \left[ \left( \hat{Q} - Q_{\#,a} \right)^2 \right].
    \end{align*}
    By triangle inequality for the $L_2(P_*)$ norm,
    \begin{align*}
        &\expect \left[ P_* \left\{ \frac{1-\pi_*(a \mid \cdot)}{\pi_*(a \mid \cdot)} \right\}^2 \left\{ -2 Q_{*,a} (\hat{Q} - \expect_{\hat{Q}}[\hat{Q}]) + \hat{Q}^2 - \expect_{\hat{Q}}[\hat{Q}^2] \right\}^2 \right] \\
        &\leq \Bigg\{ \sqrt{\expect \left[ P_* \left\{ \frac{1-\pi_*(a \mid \cdot)}{\pi_*(a \mid \cdot)} \right\}^2 \left\{ -2 Q_{*,a} (\hat{Q} - \expect_{\hat{Q}}[\hat{Q}]) \right\}^2 \right]} \\
        &\qquad+ \sqrt{\expect \left[ P_* \left\{ \frac{1-\pi_*(a \mid \cdot)}{\pi_*(a \mid \cdot)} \right\}^2 \left\{ \hat{Q}^2 - \expect_{\hat{Q}}[\hat{Q}^2] \right\}^2 \right]} \Bigg\}^{2} \\
        &\leq 16M^2 P_* \left\{ \frac{1-\pi_*(a \mid \cdot)}{\pi_*(a \mid \cdot)} \right\}^2 \expect_{\hat{Q}} \left[ \left( \hat{Q} - Q_{\#,a} \right)^2 \right].
    \end{align*}
    The bound on $\expect[\{ P_* (\transform_a(\hat{Q})^2 - \expect_{\hat{Q}}[\transform_a(\hat{Q})^2]) \}^2]$ under Condition~\ref{cond: positivity} follows immediately by upper-bounding $\{1-\pi_*(a \mid x)\}/\pi_*(a \mid x)$.
\end{proof}

\subsection{Results for cross-fit AIPW procedure} \label{sec: CV proof}

The next lemma is a tail bound for $P_{n,k} H_a(\hat{Q}_{k,a})$.
The tail bound without the approximate sub-Weibull condition \ref{cond: CV tail} can be improved under Condition~\ref{cond: bounded Q}, but it should not affect the rate in Theorem~\ref{thm: asymptotic CV BE bound}.

\begin{lemma} \label{lemma: CV empirical process}
    Under Condition~\ref{cond: positivity}, for any constant $\eta > 0$ and any $k \in \{1,\ldots,K\}$,
    $$\Pr(|I_k| |P_{n,k} H_a(\hat{Q}_{k,a})| > n^{1/2} \eta) \leq \frac{|I_k|}{n} \frac{1-\tau_\pi}{\tau_\pi \eta^2} \expect \| \hat{Q}_{k,a} - Q_{\#,a} \|_{2}^2.$$
    Moreover, under Conditions~\ref{cond: bounded Q} and \ref{cond: CV tail},
    for any constant $q>0$,
    \begin{align*}
        &\Pr(|I_k| |P_{n,k} H_a(\hat{Q}_{k,a})| > n^{1/2} \eta) \leq 2 \exp (-\eta') + 2 \hat{a}_1 n^{-q} + 2 \hat{b}_1 \sqrt{\log n / n},
    \end{align*}
    where
    $$\eta' := \eta'(\eta) := \frac{\tau_\pi}{1-\tau_\pi} \left\{ \frac{n [(q+\hat{d}_1) \hat{c}_1^{-1} \log n]^{-2/\hat{q}_1} \eta^2}{4 |I_k| \hat{r}_1(n)^2} \wedge \frac{3 \sqrt{n} \eta}{8 M} \right\}.$$
\end{lemma}

\begin{proof}
    Note that $P_{n,k} H_a(\hat{Q}_{k,a})$ is a sample average of $H_a(\hat{Q}_{k,a})(V_i)$ ($i \in I_k$), which are i.i.d.\  random variables conditional on $\hat{Q}_{k,a}$.
    By Lemma~\ref{lemma: H mean var}, $\expect[P_{n,k} H_a(\hat{Q}_{k,a}) \mid \hat{Q}_{k,a}] = 0$ and $\Var(|I_k| P_{n,k} H_a(\hat{Q}_{k,a}) \mid \hat{Q}_{k,a}) = |I_k| P_* \frac{1-\pi_*(a \mid \cdot)}{\pi_*(a \mid \cdot)} (\hat{Q}_{k,a} - Q_{\#,a})^2$.
    
    First show the bound without Conditions~\ref{cond: bounded Q} and \ref{cond: CV tail}.
    By Chebyshev's inequality,
    \begin{align*}
        \Pr(|I_k| |P_{n,k} H_a(\hat{Q}_{k,a})| > n^{1/2} \eta) &= \expect \left[ \Pr( |P_{n,k} H_a(\hat{Q}_{k,a})| > n^{1/2} \eta/|I_k| \mid \hat{Q}_{k,a} ) \right] \\
        &\leq \expect \left[ \frac{|I_k|}{n \eta^2} P_* \frac{1-\pi_*(a \mid \cdot)}{\pi_*(a \mid \cdot)} (\hat{Q}_{k,a} - Q_{\#,a})^2 \right] \\
        &\leq \frac{|I_k|}{n} \frac{1-\tau_\pi}{\tau_\pi \eta^2} \expect \| \hat{Q}_{k,a} - Q_{\#,a} \|_{2}^2,
    \end{align*}
    where the last step follows from Condition~\ref{cond: positivity}.

    Next show the bound under Conditions~\ref{cond: bounded Q} and \ref{cond: CV tail}.
    For any $x \geq 0$, let
    $$\hat{\eta}(x) := \frac{\tau_\pi}{1-\tau_\pi} \left( \frac{n \eta^2}{4 |I_k| x^2} \wedge \frac{3 \sqrt{n} \eta}{8 M} \right),$$
    where $1/0 := \infty$.
    Under Conditions~\ref{cond: positivity} and \ref{cond: bounded Q}, $|H_a(\hat{Q}_{k,a})(V)| \leq 2 \frac{1-\tau_\pi}{\tau_\pi} M$.
    By Bernstein's inequality for bounded distributions \citepsupp[e.g., Theorem~2.8.4 in][]{Vershynin2018} and the fact that $1/(a+b) \geq \min\{1/a,1/b\}/2$ for any $a,b >0$,
    \begin{align}
        &\Pr(|I_k| |P_{n,k} H_a(\hat{Q}_{k,a})| > n^{1/2} \eta) \nonumber \\
        &= \expect \left[ \Pr( |I_k| |P_{n,k} H_a(\hat{Q}_{k,a})| > n^{1/2} \eta \mid \hat{Q}_{k,a} ) \right] \nonumber \\
        &\leq 2 \expect \left[ \exp \left\{ - \frac{n \eta^2}{2 |I_k| P_* \frac{1-\pi_*(a \mid \cdot)}{\pi_*(a \mid \cdot)} (\hat{Q}_{k,a} - Q_{\#,a})^2 + \frac{4}{3} \frac{1-\tau_\pi}{\tau_\pi} M n^{1/2} \eta} \right\} \right] \nonumber \\
        &\leq 2 \expect \left[ \exp \left\{ - \frac{\tau_\pi}{1-\tau_\pi} \frac{n \eta^2}{2 |I_k| \| \hat{Q}_{k,a} - Q_{\#,a} \|_2^2 + 4 M n^{1/2} \eta/3} \right\} \right] \nonumber \\
        &\leq 2 \expect \left[ \exp \left\{ - \frac{\tau_\pi}{1-\tau_\pi} \left( \frac{n \eta^2}{4 |I_k| \| \hat{Q}_{k,a} - Q_{\#,a} \|_2^2} \wedge \frac{3 \sqrt{n} \eta}{8 M} \right) \right\} \right] \nonumber \\
        &= 2 \expect \left[ \exp \{- \hat{\eta}(\| \hat{Q}_{k,a} - Q_{\#,a} \|_2)\} \right]. \label{eq: CV tail1}
    \end{align}
    Let $t>0$ be arbitrary. The right-hand side of \eqref{eq: CV tail1} equals
    \begin{align*}
        &2 \expect \left[ \exp \{- \hat{\eta}(\| \hat{Q}_{k,a} - Q_{\#,a} \|_2)\} \ind(\| \hat{Q}_{k,a} - Q_{\#,a} \|_{2} \leq t) \right] \\
        &+2 \expect \left[ \exp \{- \hat{\eta}(\| \hat{Q}_{k,a} - Q_{\#,a} \|_2)\} \ind(\| \hat{Q}_{k,a} - Q_{\#,a} \|_{2} > t) \right] \\
        &\leq 2 \exp \{- \hat{\eta}(t)\} + 2 \Pr(\| \hat{Q}_{k,a} - Q_{\#,a} \|_{2} > t) \\
        &\leq 2 \exp \{- \hat{\eta}(t)\} + 2 \hat{a}_1 n^{\hat{d}_1} \exp \left\{ -\hat{c}_1 \left( \frac{t}{\hat{r}_1(n)} \right)^{\hat{q}_1} \right\} + 2 \hat{b}_1 \sqrt{\log n / n}.
    \end{align*}
    Taking $t = \hat{r}_1(n) [(q+\hat{d}_1) \hat{c}_1^{-1} \log n]^{1/\hat{q}_1}$ so that $\eta' = \hat{\eta}(t)$ yields the desired inequality.

\end{proof}

The tail bound under Condition~\ref{cond: CV tail} could be slightly improved by choosing $t$ that balances the two exponents, but this leads to a much more complicated bound and does not ultimately improve the rate in Theorem~\ref{thm: asymptotic CV BE bound}.

The next lemma concerns finite-sample mean and variance of the cross-fit AIPW estimator $\hat{\psi}_a$.
\begin{lemma} \label{lemma: bound CV Var psi}
    Let $k \in \{1,\ldots,K\}$ be arbitrary. It holds that $\expect[\hat{\psi}_{k,a}]=\expect[\hat{\psi}_a]=\psi_{*,a}$,
    \begin{align*}
        \Var(\hat{\psi}_{k,a}) &\leq \frac{1}{|I_k|} \left\{ \sigma_{\#,a} + \sqrt{\expect \left[ P_* \frac{1-\pi_*(a \mid \cdot)}{\pi_*(a \mid \cdot)} (\hat{Q}_{k,a} - Q_{\#,a})^2 \right]} \right\}^2, \\
        \Var(\hat{\psi}_a) &\leq \left\{ \frac{\sigma_{\#,a}}{\sqrt{n}} + \sum_{k=1}^K \frac{|I_k|}{n} \sqrt{\expect \left[ P_* \frac{1-\pi_*(a \mid \cdot)}{\pi_*(a \mid \cdot)} \frac{(\hat{Q}_{k,a} - Q_{\#,a})^2}{|I_k|} \right]} \right\}^2.
    \end{align*}
    Additionally under Condition~\ref{cond: positivity},
    \begin{align*}
        \Var(\hat{\psi}_{k,a}) &\leq \frac{1}{|I_k|} \left\{ \sigma_{\#,a} + \sqrt{\frac{1-\tau_\pi}{\tau_\pi} \expect \|\hat{Q}_{k,a} - Q_{\#,a}\|_2^2} \right\}^2, \\
        \Var(\hat{\psi}_a) &\leq \left\{ \frac{\sigma_{\#,a}}{\sqrt{n}} + \sum_{k=1}^K \frac{\sqrt{|I_k|}}{n} \sqrt{\frac{1-\tau_\pi}{\tau_\pi} \expect \|\hat{Q}_{k,a} - Q_{\#,a}\|_2^2} \right\}^2.
    \end{align*}
\end{lemma}
\begin{proof}
    Note that $P_* \transform_a(Q) = \psi_{*,a}$ for any function $Q$. Thus,
    \begin{align*}
        \hat{\psi}_{k,a} &= P_{n,k} \transform_a(\hat{Q}_{k,a}) \\
        &= P_* \transform_a(Q_{\#,a}) + P_* \{ \transform_a(\hat{Q}_{k,a}) - \transform_a(Q_{\#,a}) \} + (P_{n,k}-P_*) \transform_a(Q_{\#,a}) + (P_{n,k}-P_*) \{ \transform_a(\hat{Q}_{k,a}) - \transform_a(Q_{\#,a}) \} \\
        &= \psi_{*,a} + (P_{n,k}-P_*) \transform_a(Q_{\#,a}) + (P_{n,k}-P_*) H_a(\hat{Q}_{k,a}), \\
        \hat{\psi}_a &= \frac{1}{n} \sum_{k=1}^K |I_k| \hat{\psi}_{k,a} = \psi_{*,a} + (P_n - P_*) \transform_a(Q_{\#,a}) + \sum_{k=1}^K \frac{|I_k|}{n} (P_{n,k}-P_*) H_a(\hat{Q}_{k,a}).
    \end{align*}
    Since $\hat{Q}_{k,a}$ and $P_{n,k}$ are based on independent data, by Lemma~\ref{lemma: H mean var}, $\expect [(P_{n,k}-P_*) H_a(\hat{Q}_{k,a})]=0$, and so $\expect[\hat{\psi}_{k,a}] = \expect[\hat{\psi}_a] = \psi_{*,a}$.

    To bound $\Var(\hat{\psi}_{k,a})$ and $\Var(\hat{\psi}_a)$, note that $\Var((P_{n,k}-P_*) \transform_a(Q_{\#,a})) = \sigma_{\#,a}^2/|I_k|$, $\Var((P_n-P_*) \transform_a(Q_{\#,a})) = \sigma_{\#,a}^2/n$, and, by Lemma~\ref{lemma: H mean var}, 
    $$\Var((P_{n,k}-P_*) H_a(\hat{Q}_{k,a})) = \expect[P_* H_a(\hat{Q}_{k,a})^2]/|I_k| = \frac{1}{|I_k|} \expect \left[ P_* \frac{1-\pi_*(a \mid \cdot)}{\pi_*(a \mid \cdot)} (\hat{Q}_{k,a} - Q_{\#,a})^2 \right].$$
    By triangle inequality for $L_2(P_*)$-norm,
    \begin{align*}
        \Var(\hat{\psi}_{k,a}) = \expect[(\hat{\psi}_{k,a}-\psi_{*,a})^2] &\leq \left\{ \sqrt{\Var[(P_{n,k}-P_*) \transform_a(Q_{\#,a})]} + \sqrt{\Var [ (P_{n,k}-P_*) H_a(\hat{Q}_{k,a}) ]} \right\}^2 \\
        &= \frac{1}{|I_k|} \left\{ \sigma_{\#,a} + \sqrt{\expect \left[ P_* \frac{1-\pi_*(a \mid \cdot)}{\pi_*(a \mid \cdot)} (\hat{Q}_{k,a} - Q_{\#,a})^2 \right]} \right\}^2, \\
        \Var(\hat{\psi}_a) = \expect[(\hat{\psi}_a-\psi_{*,a})^2] &\leq \left\{ \sqrt{\Var((P_n-P_*) \transform_a(Q_{\#,a}))} + \sum_{k=1}^K \sqrt{\Var \left( \frac{|I_k|}{n} (P_{n,k}-P_*) H_a(\hat{Q}_{k,a}) \right)} \right\}^2 \\
        &= \left\{ \frac{\sigma_{\#,a}}{\sqrt{n}} + \sum_{k=1}^K \frac{|I_k|}{n} \sqrt{\expect \left[ P_* \frac{1-\pi_*(a \mid \cdot)}{\pi_*(a \mid \cdot)} \frac{(\hat{Q}_{k,a} - Q_{\#,a})^2}{|I_k|} \right]} \right\}^2.
    \end{align*}    
    The bounds under Condition~\ref{cond: positivity} follow from the boundedness of $\{1-\pi_*(a \mid x)\}/\pi_*(a \mid x)$.
\end{proof}

The following lemma bounds the tail probability of $\hat{\sigma}_{k,a}^2$.
Recall the constants $\bar{\sigma}_a^2$ and $\bar{\rho}$ introduced below Condition~\ref{cond: higher moments}.

\begin{lemma} \label{lemma: bound CV sigma2 tail}
    Under Conditions~\ref{cond: 2nd moment} and \ref{cond: positivity}--\ref{cond: higher moments}, for any $k \in \{1,\ldots,K\}$, $a \in \{0,1\}$, $\lambda_1,\lambda_2 > 0$, and $\lambda_3 > |I_k| \Var(\hat{\psi}_{k,a})/n$,
    it holds that
    \begin{align*}
        &\Pr \left( |\hat{\sigma}_{k,a}^2 - \expect[\hat{\sigma}_{k,a}^2]| > \frac{n}{|I_k|} \sum_{j=1}^3 \lambda_j \right) \\
        &\leq 2 \exp \left\{ - \frac{n^2}{2 |I_k| \bar{m}} \lambda_1^2 \right\} + \const \frac{\bar{m}}{\underline{m}^{3/2} \sqrt{|I_k|}} + \bar{\hat{\mathcal{R}}}(\lambda_2) + \const \frac{\bar{\rho}_a}{\sigma_{*,a}^{3/2} \sqrt{|I_k|}} \\
        &\quad+ 2 \exp \left\{ -\frac{(\sqrt{2}-1)^2}{2} \frac{|I_k|}{\bar{\sigma}_a^2} \left\{ \left( \frac{n}{|I_k|} \lambda_3 - \Var(\hat{\psi}_{k,a}) \right) \wedge \frac{ \left( \frac{n}{|I_k|} \lambda_3 - \Var(\hat{\psi}_{k,a}) \right)^2}{|\psi_{*,a}|^2} \right\} \right\}
    \end{align*}
    where
    $$\bar{\hat{\mathcal{R}}}(\lambda_2) := 16 M^2 \frac{|I_k|^2}{n^2 \lambda_2^2} \left( \frac{1-\tau_\pi}{\tau_\pi} \right)^2 \expect \| \hat{Q}_{k,a}-Q_{\#,a} \|_2^2.$$
    Under Condition~\ref{cond: more CV tail}, $\bar{\hat{\mathcal{R}}}(\lambda_2)$ may be replaced by
    $$\bar{\hat{\mathcal{R}}}'(\lambda_2) := \hat{a}_2 n^{\hat{d}_2} \exp \left\{ - \hat{c}_2 \left( \frac{n \lambda_2}{|I_k| \hat{r}_2(n)} \right)^{\hat{q}_2} \right\} + \hat{b}_2 \sqrt{\log n / n}.$$
\end{lemma}
\begin{proof}
    Since $\hat{\sigma}_{k,a}^2 = P_{n,k} \transform_a(\hat{Q}_{k,a})^2 - \hat{\psi}_{k,a}^2$ and $\expect[\hat{\sigma}_{k,a}^2] = P_* \expect_{\hat{Q}_{k,a}}[\transform_a(\hat{Q}_{k,a})^2] - \expect[\hat{\psi}_{k,a}^2] = P_* \expect_{\hat{Q}_{k,a}}[\transform_a(\hat{Q}_{k,a})^2] - \psi_{*,a}^2 - \Var(\hat{\psi}_{k,a})$, by union bound,
    \begin{align*}
        &\Pr \left( |\hat{\sigma}_{k,a}^2 - \expect[\hat{\sigma}_{k,a}^2]| > \frac{n}{|I_k|} \sum_{j=1}^3 \lambda_j \right) \\
        &\leq \Pr \left( |(P_{n,k}-P_*) \transform_a(\hat{Q}_{k,a})^2| > \frac{n}{|I_k|} \lambda_1 \right) + \Pr \left( |P_* \{\transform_a(\hat{Q}_{k,a})^2 - \expect_{\hat{Q}_{k,a}}[\transform_a(\hat{Q}_{k,a})^2]\}| > \frac{n}{|I_k|} \lambda_2 \right) \\
        &\quad + \Pr \left( |\hat{\psi}_{k,a}^2 - \psi_{*,a}^2 - \Var(\hat{\psi}_{k,a})| > \frac{n}{|I_k|} \lambda_3 \right).
    \end{align*}
    I next bound the three terms on the right-hand side separately.

    \noindent\textbf{Term~1:} Let $\zeta := \sqrt{P_* \{\transform_a(\hat{Q}_{k,a})^2 - P_* \transform_a(\hat{Q}_{k,a})^2\}^2}$, $\xi := P_* | \transform_a(\hat{Q}_{k,a})^2 - P_* \transform_a(\hat{Q}_{k,a})^2 |^3$, and $\Phi$ denote the cumulative distribution function of standard Gaussian. By Condition~\ref{cond: higher moments}, $\Pr( \zeta \in [\underline{m}, \bar{m}] ) = \Pr(\xi \leq \bar{m}) = 1$. By Berry-Esseen Theorem and the exponential tail bound for Gaussian distribution,
    \begin{align*}
        &\Pr \left( |(P_{n,k}-P_*) \transform_a(\hat{Q}_{k,a})^2| > \frac{n}{|I_k|} \lambda_1 \right) \\
        &= \expect \left[ \Pr \left( \sqrt{|I_k|} \left| \frac{(P_{n,k}-P_*) \transform_a(\hat{Q}_{k,a})^2}{\zeta} \right| > \sqrt{|I_k|} \frac{n}{|I_k| \zeta} \lambda_1 \mid \hat{Q}_{k,a} \right) \right] \\
        &\leq \expect \left[ 2 \Phi \left( -\sqrt{|I_k|} \frac{n}{|I_k| \zeta} \lambda_1 \right) + \const \frac{\xi}{\zeta^3 \sqrt{|I_k|}} \right] \\
        &\leq \expect \left[ 2 \exp \left\{ -\frac{n^2}{2|I_k| \zeta^2} \lambda_1^2 \right\} + \const \frac{\xi}{\zeta^3 \sqrt{|I_k|}} \right] \\
        &\leq 2 \exp \left\{ - \frac{n^2}{2 |I_k| \bar{m}} \lambda_1^2 \right\} + \const \frac{\bar{m}}{\underline{m}^{3/2} \sqrt{|I_k|}}.
    \end{align*}

    \noindent\textbf{Term~2:} By Lemma~\ref{lemma: transform sqaure},
    \begin{align*}
        \expect[\{ P_* (\transform_a(\hat{Q}_{k,a})^2 - \expect_{\hat{Q}_{k,a}}[\transform_a(\hat{Q}_{k,a})^2]) \}^2] \leq 16M^2 \left( \frac{1-\tau_\pi}{\tau_\pi} \right)^2 \expect \| \hat{Q}_{k,a}-Q_{\#,a} \|_{2}^2.
    \end{align*}
    By Chebyshev's inequality, Term~2 is upper bounded by $\bar{\hat{\mathcal{R}}}(\lambda_2)$.
    If Condition~\ref{cond: more CV tail} holds, then Term~2 is upper bounded by $\bar{\hat{\mathcal{R}}}'(\lambda_2)$.

    \noindent\textbf{Term~3:} Term~3 is upper bounded by $\Pr \left( |\hat{\psi}_{k,a}^2 - \psi_{*,a}^2| > c \right)$ where $c := n \lambda_3/|I_k| - \Var(\hat{\psi}_{k,a}) > 0$.
    By Lemma~\ref{lemma: square diff inequality}, this probability is further upper bounded by
    $$\Pr \left( |\hat{\psi}_{k,a} - \psi_{*,a}| > (\sqrt{2}-1) \left\{ \sqrt{c} \wedge \frac{c}{|\psi_{*,a}|} \right\} \right),$$
    which, by Berry-Esseen Theorem and the exponential tail bound for Gaussian distribution, is further upper bounded by
    \begin{align*}
        &= \expect \left[ \Pr \left( \sqrt{|I_k|} \frac{\left| P_{n,k} D_a(\hat{Q}_{k,a},\psi_{*,a}) \right|}{\sqrt{P_* D_a(\hat{Q}_{k,a},\psi_{*,a})^2}} > (\sqrt{2}-1) \frac{\sqrt{|I_k|}}{\sqrt{P_* D_a(\hat{Q}_{k,a},\psi_{*,a})^2}} \left\{ \sqrt{c} \wedge \frac{c}{|\psi_{*,a}|} \right\} \mid \hat{Q}_{k,a} \right) \right] \\
        &\leq \expect \left[ 2\Phi \left( -(\sqrt{2}-1) \frac{\sqrt{|I_k|}}{\sqrt{P_* D_a(\hat{Q}_{k,a},\psi_{*,a})^2}} \left\{ \sqrt{c} \wedge \frac{c}{|\psi_{*,a}|} \right\} \right) + \const \frac{P_*|D_a(\hat{Q}_{k,a},\psi_{*,a})|^3}{\left\{ P_* D_a(\hat{Q}_{k,a},\psi_{*,a})^2 \right\}^{3/2} \sqrt{|I_k|}} \right] \\
        &\leq \expect \left[ 2 \exp \left\{ -\frac{(\sqrt{2}-1)^2}{2} \frac{|I_k|}{P_* D_a(\hat{Q}_{k,a},\psi_{*,a})^2} \left( c \wedge \frac{c^2}{|\psi_{*,a}|^2} \right) \right\} + \const \frac{P_*|D_a(\hat{Q}_{k,a},\psi_{*,a})|^3}{\left\{ P_* D_a(\hat{Q}_{k,a},\psi_{*,a})^2 \right\}^{3/2} \sqrt{|I_k|}} \right] \\
        &\leq 2 \exp \left\{ -\frac{(\sqrt{2}-1)^2}{2} \frac{|I_k|}{\bar{\sigma}_a^2} \left( c \wedge \frac{c^2}{|\psi_{*,a}|^2} \right) \right\} + \const \frac{\bar{\rho}_a}{\sigma_{*,a}^{3/2} \sqrt{|I_k|}}.
    \end{align*}
    In the last step, the following fact was used: $P_* D_a(\hat{Q}_{k,a},\psi_{*,a})^2 \geq \sigma_{*,a}^2 > 0$ by Condition~\ref{cond: 2nd moment}, because $\sigma_{*,a}^2$ is the efficiency bound for estimating $\psi_{*,a}$ and $D_a(\hat{Q}_{k,a},\psi_{*,a})$ is an influence function that might not be efficient.
\end{proof}

\begin{proof}[Proof of Theorem~\ref{thm: CV expect Var}]
    By Lemma~\ref{lemma: transform sqaure},
    \begin{align*}
        &\sigma_{\dagger,a}^2 - \sigma_{\#,a}^2 \\
        &= \sum_{k=1}^K \frac{|I_k|}{n} \left\{ \expect[\hat{\sigma}_{k,a}^2] - \sigma_{\#,a}^2 \right\} \\
        &= \sum_{k=1}^K \frac{|I_k|}{n} \left\{ \expect[P_{n,k} \transform_a(\hat{Q}_{k,a})^2 - \hat{\psi}_{k,a}^2] - P_* \transform_a(Q_{\#,a})^2 + \psi_{*,a}^2 \right\} \\
        &= \sum_{k=1}^K \frac{|I_k|}{n} \Bigg\{ \expect \left[ P_* \frac{1-\pi_*(a \mid \cdot)}{\pi_*(a \mid \cdot)} (\hat{Q}_{k,a} - Q_{\#,a})^2 \right] + 2 \expect \left[ P_* \frac{1-\pi_*(a \mid \cdot)}{\pi_*(a \mid \cdot)} (\hat{Q}_{k,a} - Q_{\#,a}) (Q_{\#,a} - Q_{*,a}) \right] \\
        &\qquad\qquad- \Var(\hat{\psi}_{k,a}) \Bigg\}.
    \end{align*}
    Thus, \eqref{eq: CV Var bias} holds. Eq.~\ref{eq: ATE CV Var bias} can be shown similarly.
    
    Under Condition~\ref{cond: positivity}, by the triangle inequality, Cauchy-Schwarz inequality, and Lemma~\ref{lemma: bound CV Var psi}, $|\sigma_{\dagger,a}^2 - \sigma_{\#,a}^2| \leq \sum_{k=1}^K \hat{S}_k$.
    Under Condition~\ref{cond: CV common mean Q or consistent Q}, $|\sigma_{\dagger,a}^2 - \sigma_{\#,a}^2| \leq \sum_{k=1}^K \hat{S}_k'$ under Condition~\ref{cond: positivity} by Lemma~\ref{lemma: bound CV Var psi}.
    The bounds for $|\sigma_{\dagger,a} - \sigma_{\#,a}|$ follow from the fact that $|\sigma_{\dagger,a} - \sigma_{\#,a}| = |\sigma_{\dagger,a}^2 - \sigma_{\#,a}^2|/(\sigma_{\dagger,a} + \sigma_{\#,a}) \leq |\sigma_{\dagger,a}^2 - \sigma_{\#,a}^2|/\sigma_{\#,a}$, where $\sigma_{\#,a} \geq \sigma_{*,a} > 0$ by Condition~\ref{cond: 2nd moment} and the fact that $\sigma_{*,a}^2$ is the efficiency bound.
\end{proof}

The following lemma bounds the difference between $\sqrt{n} (\hat{\psi}_a-\psi_{*,a}) - z_\alpha \hat{\sigma}_a$ and $\sqrt{n} P_n D_a(Q_{\#,a},\psi_{*,a}) - z_\alpha \sigma_{\dagger,a}$, which involves the linear approximation to $\hat{\psi}_a-\psi_{*,a}$ and the approximation $\sigma_{\dagger,a}$ to $\hat{\sigma}_a$.
\begin{lemma} \label{lemma: bound CV linear approx}
    Recall $\bar{\hat{\mathcal{R}}}$ and $\bar{\hat{\mathcal{R}}}'$ defined in Lemma~\ref{lemma: bound CV sigma2 tail}. 
    Suppose that $\sigma_{\dagger,a}>0$.
    Under Conditions~\ref{cond: 2nd moment}--\ref{cond: higher moments}, for any constants $\eta_k, \lambda_{k,j} > 0$ and $\lambda_{k,3} > \frac{z_\alpha |I_k| \Var(\hat{\psi}_{k,a})}{2 \sigma_{*,a} n}$ ($k=1,\ldots,K$, $j=1,2$), with $\bar{\eta} := \sum_{k=1}^K \eta_k + \sum_{k=1}^K \sum_{j=1}^3 \lambda_{k,j}$,
    \begin{align*}
        &\Pr \left( |\sqrt{n} \{ \hat{\psi}_a - \psi_{*,a} - P_n D_a(Q_{\#,a},\psi_{*,a}) \} - z_\alpha (\hat{\sigma}_a - \sigma_{\dagger,a})| > \bar{\eta} \right) \\
        &\leq \sum_{k=1}^K \Bigg[ \hat{\hat{\mathcal{R}}}(\eta_k,\lambda_{k,2}) + 2 \ind(z_\alpha \neq 0) \exp \left\{ - \frac{n^2 \sigma_{\dagger,a}^2}{2 |I_k| \bar{m} z_\alpha^2} \lambda_{k,1}^2 \right\} + \const \frac{\bar{m}}{\underline{m}^{3/2} \sqrt{|I_k|}} + \const \ind(z_\alpha \neq 0) \frac{\bar{\rho}_a}{\sigma_{*,a}^{3/2} \sqrt{|I_k|}} \nonumber \\
        &\qquad+ 2 \ind(z_\alpha \neq 0) \exp \left\{ -\frac{(\sqrt{2}-1)^2}{2} \frac{|I_k|}{\bar{\sigma}_a^2} \left[ \left( \frac{n \sigma_{\dagger,a}}{|I_k| |z_\alpha|} \lambda_{k,3} - \Var(\hat{\psi}_{k,a}) \right) \wedge \frac{ \left( \frac{n \sigma_{\dagger,a}}{|I_k| |z_\alpha|} \lambda_{k,3} - \Var(\hat{\psi}_{k,a}) \right)^2}{|\psi_{*,a}|^2} \right] \right\} \Bigg]
    \end{align*}
    where $\hat{\hat{\mathcal{R}}}$ is defined pointwise as
    \begin{align*}
        \hat{\hat{\mathcal{R}}}(\eta,\lambda_2) &:= \frac{|I_k|}{n} \frac{1-\tau_\pi}{\tau_\pi \eta^2} \expect \| \hat{Q}_{k,a} - Q_{\#,a} \|_{2}^2 + \ind(z_\alpha \neq 0) \bar{\hat{\mathcal{R}}} \left( \frac{\sigma_{\dagger,a}}{|z_\alpha|} \lambda_2 \right).
    \end{align*}
    Additionally under Conditions~\ref{cond: CV tail} and \ref{cond: more CV tail}, for any constant $q>0$ and all $k=1,\ldots,K$, $\hat{\hat{\mathcal{R}}}(\eta_k,\lambda_{k,2})$ can be replaced by $\hat{\hat{\mathcal{R}}}'(\eta_k,\lambda_{k,2},q)$, which is defined pointwise as
    \begin{align*}
        \hat{\hat{\mathcal{R}}}'(\eta,\lambda_2,q) &:= 2 \exp(-\eta')+ 2 \hat{a}_1 n^{-q} + 2 \hat{b}_1 \sqrt{\log n / n} + \ind(z_\alpha \neq 0) \bar{\hat{\mathcal{R}}}' \left( \frac{\sigma_{\dagger,a}}{|z_\alpha|} \lambda_2 \right).
    \end{align*}
    where $\eta'=\eta'(\eta)$ is defined in Lemma~\ref{lemma: CV empirical process}.
\end{lemma}
\begin{proof}
    Note that $\sqrt{n} \{ \hat{\psi}_a - \psi_{*,a} - P_n D_a(Q_{\#,a},\psi_{*,a}) \} = n^{-1/2} \sum_{k=1}^K |I_k| P_{n,k} H_a(\hat{Q}_{k,a})$. By union bound,
    \begin{align*}
        &\Pr \left( |\sqrt{n} \{ \hat{\psi}_a - \psi_{*,a} - P_n D_a(Q_{\#,a},\psi_{*,a}) \} - z_\alpha (\hat{\sigma}_a - \sigma_{\dagger,a})| > \sum_{k=1}^K \eta_k + \sum_{k=1}^K \sum_{j=1}^3 \lambda_{k,j} \right) \\
        &\leq \sum_{k=1}^K \Pr \left( |I_k| |P_{n,k} H_a(\hat{Q}_{k,a})| > n^{1/2} \eta_k \right) + \ind(z_\alpha \neq 0) \Pr \left( |\hat{\sigma}_a - \sigma_{\dagger,a}| > \sum_{k=1}^K \sum_{j=1}^3 \lambda_{k,j} / |z_\alpha| \right).
    \end{align*}
    Here, $\Pr \left( |I_k| |P_{n,k} H_a(\hat{Q}_{k,a})| > n^{1/2} \eta_k \right)$ can be bounded as in Lemma~\ref{lemma: CV empirical process}.
    Since $\sigma_{\dagger,a} + \hat{\sigma}_a \geq \sigma_{\dagger,a} > 0$,
    if $z_\alpha \neq 0$,
    \begin{align*}
        & \Pr \left( |\hat{\sigma}_a - \sigma_{\dagger,a}| > \sum_{k=1}^K \sum_{j=1}^3 \lambda_{k,j} / |z_\alpha| \right) \\
        &\leq \Pr \left( |\hat{\sigma}_a^2 - \sigma_{\dagger,a}^2| > \sum_{k=1}^K \sum_{j=1}^3 \lambda_{k,j} \sigma_{\dagger,a} / |z_\alpha| \right) \\
        &= \Pr \left( \left| \frac{1}{n} \sum_{k=1}^K |I_k| (\hat{\sigma}_{k,a}^2 - \expect[\hat{\sigma}_{k,a}^2]) \right| > \sum_{k=1}^K \sum_{j=1}^3 \lambda_{k,j} \sigma_{\dagger,a} / |z_\alpha| \right) \\
        &\leq \sum_{k=1}^K \Pr \left( |\hat{\sigma}_{k,a}^2 - \expect[\hat{\sigma}_{k,a}^2]| > \frac{n}{|I_k|} \sum_{j=1}^3 \frac{\lambda_{k,j} \sigma_{\dagger,a}}{|z_\alpha|} \right),
    \end{align*}
    which can be bounded as in Lemma~\ref{lemma: bound CV sigma2 tail}.
\end{proof}

The above results yield Theorem~\ref{thm: CV BE bound}.
\begin{proof}[Proof of Theorem~\ref{thm: CV BE bound}]
    Taking $\eta$ in Lemmas~\ref{lemma: AIPW delta} and \ref{lemma: normal quantile} to be $\sum_{k=1}^K \eta_k + \sum_{k=1}^K \sum_{j=1}^3 \lambda_{k,j}$, by Lemma~\ref{lemma: bound CV linear approx}, it holds that
    \begin{align}
        &\left| \Pr(\sqrt{n} (\hat{\psi}_a - \psi_{*,a}) \leq z_\alpha \hat{\sigma}_a) - (1-\alpha) - \phi(z_\alpha) z_\alpha \frac{\sigma_{\dagger,a}-\sigma_{\#,a}}{\sigma_{\#,a}} \right| \nonumber \\
        &\leq \const \frac{\rho_{\#,a}}{\sigma_{\#,a}^3 \sqrt{n}} + \const \frac{(\sigma_{\dagger,a}-\sigma_{\#,a})^2}{\sigma_{\#,a}^2} + \const \frac{\sum_{k=1}^K \eta_k + \sum_{k=1}^K \sum_{j=1}^3 \lambda_{k,j}}{\sigma_{\#,a}} \nonumber \\
        &\quad+ \sum_{k=1}^K \Bigg[ \hat{\hat{\mathcal{R}}}(\eta_k,\lambda_{k,2}) + 2 \ind(z_\alpha \neq 0) \exp \left\{ - \frac{n^2 \sigma_{\dagger,a}^2}{2 |I_k| \bar{m} z_\alpha^2} \lambda_{k,1}^2 \right\} + \const \frac{\bar{m}}{\underline{m}^{3/2} \sqrt{|I_k|}} + \const \ind(z_\alpha \neq 0) \frac{\bar{\rho}_a}{\sigma_{*,a}^{3/2} \sqrt{|I_k|}} \nonumber \\
        &\quad+ 2 \ind(z_\alpha \neq 0) \exp \left\{ -\frac{(\sqrt{2}-1)^2}{2} \frac{|I_k|}{\bar{\sigma}_a^2} \left[ \left( \frac{n \sigma_{\dagger,a}}{|I_k| |z_\alpha|} \lambda_{k,3} - \Var(\hat{\psi}_{k,a}) \right) \wedge \frac{ \left( \frac{n \sigma_{\dagger,a}}{|I_k| |z_\alpha|} \lambda_{k,3} - \Var(\hat{\psi}_{k,a}) \right)^2}{|\psi_{*,a}|^2} \right] \right\} \Bigg]. \label{eq: CV BE rough bound}
    \end{align}
    
    Take
    \begin{align*}
        & \eta_k = \left\{ \frac{|I_k|}{n} \frac{1-\tau_\pi}{\tau_\pi} \sigma_{\#,a} \expect \|\hat{Q}_{k,a} - Q_{\#,a} \|_2^2 \right\}^{1/3}, \quad \lambda_{k,1} = \begin{cases}
            \frac{|z_\alpha| \sqrt{2 q |I_k| \bar{m} \log n}}{n \sigma_{\dagger,a}} & (\text{if } z_\alpha \neq 0) \\
            n^{-q} & (\text{if } z_\alpha = 0)
        \end{cases},\\
        & \lambda_{k,2} = \left\{ 16 M^2 \frac{|I_k|^2 z_\alpha^2 \sigma_{\#,a}}{n^2 \sigma_{\dagger,a}^2} \left( \frac{1-\tau_\pi}{\tau_\pi} \right)^2 \expect \|\hat{Q}_{k,a} - Q_{\#,a} \|_2^2 \right\}^{1/3}, \\
        & \lambda_{k,3} = \begin{cases}
            \frac{|I_k| |z_\alpha|}{n \sigma_{\dagger,a}} \left\{ \sqrt{2} (\sqrt{2}+1) |\psi_{*,a}| \sqrt{q \frac{\bar{\sigma}_a^2}{|I_k|} \log n} + \Var(\hat{\psi}_{k,a}) \right\} & (\text{if } z_\alpha \neq 0, \psi_{*,a} \neq 0) \\
            \frac{|I_k| |z_\alpha|}{n \sigma_{\dagger,a}} \left\{ 2 (\sqrt{2}+1)^2 q \frac{\bar{\sigma}_a^2}{|I_k|} \log n + \Var(\hat{\psi}_{k,a}) \right\} & (\text{if } z_\alpha \neq 0, \psi_{*,a}=0) \\
            n^{-q} & (\text{if } z_\alpha=0)
        \end{cases}
    \end{align*}
    in \eqref{eq: CV BE rough bound} so that
    $$\frac{n \sigma_{\dagger,a}}{|I_k| |z_\alpha|} \lambda_{k,3} - \Var(\hat{\psi}_{k,a}) \geq \frac{ \left( \frac{n \sigma_{\dagger,a}}{|I_k| |z_\alpha|} \lambda_{k,3} - \Var(\hat{\psi}_{k,a}) \right)^2}{|\psi_{*,a}|^2} \text{ when } \psi_{*,a} \neq 0, z_\alpha \neq 0.$$
    By Lemma~\ref{lemma: bound CV Var psi} and Theorem~\ref{thm: CV expect Var}, \eqref{eq: CV BE bound} holds.

    Additionally under Conditions~\ref{cond: CV tail} and \ref{cond: more CV tail}, by Lemma~\ref{lemma: bound CV linear approx}, \eqref{eq: CV BE rough bound} holds with $\hat{\hat{\mathcal{R}}}(\eta_k,\lambda_{k,2})$ replaced by $\hat{\hat{\mathcal{R}}}'(\eta_k,\lambda_{k,2},q)$.
    Taking
    \begin{align*}
        & \eta_k = \left\{ 2 [(q+\hat{d}_1) \log n]^{(\hat{q}_1+2)/(2 \hat{q}_1)} \hat{c}_1^{-1/\hat{q}_1} \hat{r}_1(n) \sqrt{\frac{\tau_\pi}{1-\tau_\pi} \frac{|I_k|}{n}} \right\} \vee \left\{ \frac{8}{3} M \frac{1-\tau_\pi}{\tau_\pi} \frac{(q+\hat{d}_1) \log n}{\sqrt{n}} \right\}, \\
        & \lambda_{k,2} = \left[ \frac{(q+\hat{d}_2) \log n}{\hat{c}_2}\right]^{1/\hat{q}_2} \frac{|I_k| |z_\alpha|}{n \sigma_{\dagger,a}} \hat{r}_2(n),
    \end{align*}
    and $(\lambda_{k,1},\lambda_{k,3})$ the same as above,
    so that $\eta'(\eta_k) \geq (q+\hat{d}_1) \log n$,
    by Lemma~\ref{lemma: bound CV Var psi} and Theorem~\ref{thm: CV expect Var}, 
    \eqref{eq: CV BE bound} holds with $\mathcal{R}_k$ replaced by $\hat{\mathcal{R}}_k'$.
\end{proof}

\subsection{Results for non-cross-fit AIPW procedure} \label{sec: non CV proof}

For any function class $\gfunclass$ and any $\delta>0$, define $\gfunclass_\delta := \{ f \in \gfunclass: \|f\|_2 \leq \delta\}$.
Let $\empro_n$ denote the empirical process, so $\empro_n f = \sqrt{n} (P_n - P_*) f$ for any function $f \in L_2(P_*)$, and $\| \empro_n \|_\gfunclass := \sup_{f \in \gfunclass} |\empro_n f|$ for any function class $\gfunclass \subseteq L_2(P_*)$.
Let $\| \cdot \|_{2,\mu}$ denote the $L_2(\mu)$ norm for any measure $\mu$.
Define $\funclass^{H_a} := \{ H_a(Q): Q \in \funclass \}$, $\funclass^{\transform_a^2} := \{ \transform_a(Q)^2 - \transform_a(Q_{\#,a})^2: Q \in \funclass \}$, where $\funclass$ is introduced in Condition~\ref{cond: donsker}.
Recall $F^{H_a}$ and $F^{\transform_a^2}$ defined above Theorem~\ref{thm: BE bound}.

\begin{lemma}[Bounds on covering numbers] \label{lemma: covering number}
    Suppose that Conditions~\ref{cond: positivity}, \ref{cond: bounded Q}, and \ref{cond: donsker} hold.
    For any $a \in \{0,1\}$, $\funclass^{H_{a}}$ and $\funclass^{\transform_{a}^2}$ have constant envelopes $F^{H_a}$ and $F^{\transform_a^2}$, respectively.
    Moreover, for any probability measure $\mu$,
    \begin{align*}
        \| H_{a'}(Q_1) - H_{a'}(Q_2) \|_{2,\mu} &\leq \frac{1 - \tau_\pi}{\tau_\pi} \| Q_1-Q_2 \|_{2,\mu}, \\
        \| \transform_{a'}(Q_1)^2 - \transform_{a'}(Q_2)^2 \|_{2,\mu} &\leq 4 M \frac{1-\tau_\pi}{\tau_\pi^2} \| Q_1-Q_2 \|_{2,\mu},
    \end{align*}
    so both $N (\epsilon F^{H_a}, \funclass^{H_a}, L_2(\mu))$ and $N(\epsilon F^{\transform_a^2}, \funclass^{\transform_a^2}, L_2(\mu))$ are upper bounded by $N(\epsilon \cdot 2M, \funclass,L_2(\mu))$.
    Thus, both $J(\delta,\funclass^{H_{a}},F^{H_a})$ and $J(\delta, \funclass^{\transform_{a}^2},\funclass^{\transform_a^2})$ are upper bounded by $J(2 \delta, \funclass,M)/2$ for any $\delta>0$.
\end{lemma}
\begin{proof}
    First, show the validity of envelopes. Let $a' \in \{0,1\}$.
    \begin{align*}
        |H_{a'}(Q)(v)| &= \left| \left\{ 1 - \frac{\ind(a=a')}{\pi_*(a' \mid x)} \right\}  \{ Q(x) - Q_{\#,a}(x) \} \right| \\
        &= \left\{ \ind(a=a') \frac{1-\pi_*(a' \mid x)}{\pi_*(a' \mid x)} + \ind(a \neq a') \right\} |Q(x) - Q_{\#,a}(x)| \\
        &\leq 2 M \frac{1-\tau_\pi}{\tau_\pi} = F^{H_{a'}}.
    \end{align*}
    By Lemma~\ref{lemma: transform sqaure},
    \begin{align*}
        &\transform_{a'}(Q)(v)^2 - \transform_{a'}(Q_{\#,a'})(v)^2 \\
        &= \left\{ -2 \frac{\ind(a=a') (1-\pi_*(a' \mid x))}{\pi_*(a' \mid x)^2} y + \left( \frac{\ind(a=a')}{\pi_*(a' \mid x)} - 1 \right)^2 (Q(x) + Q_{\#,a'}(x)) \right\} (Q(x) - Q_{\#,a'}(x)) \\
        &= \left[ \ind(a=a') \left\{ -2 \frac{1-\pi_*(a' \mid x)}{\pi_*(a' \mid x)} y + \left( \frac{1-\pi_*(a' \mid x)}{\pi_*(a' \mid x)} \right)^2 (Q(x) + Q_{\#,a'}(x)) \right\} + \ind(a \neq a') (Q(x) + Q_{\#,a'}(x)) \right] \\
        &\qquad\times (Q(x) - Q_{\#,a'}(x)),
    \end{align*}
    so
    \begin{align*}
        &\left| \transform_{a'}(Q)(v)^2 - \transform_{a'}(Q_{\#,a'})(v)^2 \right| \leq \left\{ 2 M \frac{1-\tau_\pi}{\tau_\pi} + 2 M \left( \frac{1-\tau_\pi}{\tau_\pi} \right)^2 \right\} \cdot 2M = F^{\transform_{a'}^2}.
    \end{align*}
    
    For $\epsilon>0$ and any functions $Q_1,Q_2 \in \funclass$,
    \begin{align*}
        \| H_{a'}(Q_1) - H_{a'}(Q_2) \|_{2,\mu} = \left[ \int \left\{ 1 - \frac{\ind(a=a')}{\pi_*(a' \mid x)} \right\}^2 \{ Q_1(x)-Q_2(x)\}^2 \intd \mu(x,a,y) \right]^{1/2} \leq \frac{1-\tau_\pi}{\tau_\pi} \| Q_1-Q_2 \|_{2,\mu},
    \end{align*}
    where the bound on $1 - \ind(a=a')/\pi_*(a' \mid x)$ is derived using a similar argument as in the proof of the envelope $F^{H_{a'}}$ of $\funclass^{H_{a'}}$.
    Thus, $N (\epsilon F^{H_{a'}}, \funclass^{H_{a'}}, L_2(\mu)) \leq N(\epsilon \cdot 2 M, \funclass, L_2(\mu))$.
    Similarly, for any functions $Q_1,Q_2 \in \funclass$, by Lemma~\ref{lemma: transform sqaure},
    \begin{align*}
        \| \transform_{a'}(Q_1)^2 - \transform_{a'}(Q_2)^2 \|_{2,\mu} \leq 4 M \frac{1-\tau_\pi}{\tau_\pi^2} \| Q_1-Q_2 \|_{2,\mu},
    \end{align*}
    so $N(\epsilon F^{\transform_{a'}^2}, \funclass^{\transform_{a'}^2}, L_2(\mu)) \leq N(\epsilon \cdot 2M, \funclass,L_2(\mu))$.
    The bounds for uniform entropy integrals follow from the above bounds on covering numbers and a change of variable.
\end{proof}

\begin{lemma} \label{lemma: rough empro tail bound}
    For any $\delta > 0$, $\alpha>0$, $q \geq 2$ and $t \geq 1$, under Conditions~\ref{cond: positivity}, \ref{cond: bounded Q} and \ref{cond: donsker},
    \begin{align*}
        & \max \left\{ \expect \| \empro_n \|_{\funclass^{H_a}_{\delta F^{H_a}}}, \expect \| \empro_n \|_{\funclass^{\transform_a^2}_{\delta F^{\transform_a^2}}} \right\} \leq \const F^{H_a} J(2 \delta, \funclass, M) + \const \frac{F^{H_a} J^2(2 \delta,\funclass,M)}{\delta^2 \sqrt{n}}, \\
        & \max \left\{ \expect \| \empro_n \|_{\funclass^{H_a}}, \expect \| \empro_n \|_{\funclass^{\transform_a^2}} \right\} \leq \const F^{H_a} J(2, \funclass, M),
    \end{align*}
    and, for sufficiently large absolute constant $\const$, the following three probabilities are all upper bounded by $t^{-q/2}$:
    \begin{align*}
        &\Pr \left( |\empro_n H_a(\tilde{Q}_a)| > (1+\alpha) \expect \| \empro_n \|_{\funclass^{H_a}_{\delta F^{H_a}}} + \const q \left\{ \left( \delta F^{H_a} + \frac{F^{H_a}}{\sqrt{n}} \right) \sqrt{t} + \frac{F^{H_a} t}{\alpha \sqrt{n}} \right\}, \| H_a(\tilde{Q}_a) \|_2 \leq \delta F^{H_a} \right), \\
        &\Pr \Big( |\empro_n \{\transform_a(\tilde{Q}_a)^2- \transform_a(Q_{\#,a})^2\}| > (1+\alpha) \expect \| \empro_n \|_{\funclass^{\transform_a^2}_{\delta F^{\transform_a^2}}} + \const q \left\{ \left( \delta F^{\transform_a^2} + \frac{F^{\transform_a^2}}{\sqrt{n}} \right) \sqrt{t} + \frac{F^{\transform_a^2} t}{\alpha \sqrt{n}} \right\}, \\
        &\qquad\qquad \| \transform_a(\tilde{Q}_a)^2- \transform_a(Q_{\#,a})^2 \|_2 \leq \delta F^{\transform_a^2} \Big), \\
        &\Pr \left( |\empro_n H_a(\tilde{Q}_a)| > (1+\alpha) \expect \| \empro_n \|_{\funclass^{H_a}} + \const q \left\{ F^{H_a} \sqrt{t} + \frac{F^{H_a} t}{\alpha \sqrt{n}} \right\} \right).
    \end{align*}
\end{lemma}
\begin{proof}
    This lemma essentially follows from Theorems~5.1 and 5.2 in \protect\citetsupp{Chernozhukov2014}. Based on the proof of this theorem, because $q \geq 2$ and $n \geq 1$, the constant $K(q)$ in this theorem can be taken as $\const q$ for an absolute constant $\const$.
    
    Consider the result for $H_a(\tilde{Q}_a)$ first.
    By Lemma~\ref{lemma: covering number}, $\sup_{f \in \funclass_{\delta F^{H_a}}^{H_a}} P_* f^2 \leq \delta^2 (F^{H_a})^2$. Taking the envelope in Lemma~\ref{lemma: covering number} and $\sigma^2$ in Section~5 in \protect\citepsupp{Chernozhukov2014} as $\delta^2 (F^{H_a})^2$, Theorem~5.1 in \protect\citepsupp{Chernozhukov2014} implies that, for any $q \geq 2$, $t \geq 1$, and $\alpha>0$, with probability at least $1-t^{-q/2}$,
    $$\|\empro_n\|_{\funclass_{\delta F^{H_a}}^{H_a}} \leq (1+\alpha) \expect \|\empro_n \|_{\funclass_{\delta F^{H_a}}^{H_a}} + \const q \left\{ \left( \delta F^{H_a} + \frac{F^{H_a}}{\sqrt{n}} \right) \sqrt{t} + \frac{F^{H_a} t}{\alpha \sqrt{n}} \right\}.$$
    Moreover, by Theorem~5.2 in \protect\citetsupp{Chernozhukov2014},
    $$\expect \|\empro_n \|_{\funclass_{\delta F^{H_a}}^{H_a}} \leq \const F^{H_a} J(\delta, \funclass^{H_a}, F^{H_a}) + \const \frac{F^{H_a} J^2(\delta,\funclass^{H_a},F^{H_a})}{\delta^2 \sqrt{n}}.$$
    The desired result follows from Condition~\ref{cond: donsker} and Lemma~\ref{lemma: covering number}.

    The result for $H_a(\tilde{Q}_a)$ without $\| H_a(\tilde{Q}_a) \|_2 \leq \delta F^{H_a}$ is also similar, except that $\funclass^{H_a}$ is considered so that $\sigma$ in Theorem~5.2 in \protect\citetsupp{Chernozhukov2014} is taken to be 1, and, by the proof of Theorem~2.5.2 in \protect\citetsupp{VanderVaart2023} (particularly the fourth displayed equation on page~193), $\expect \| \empro_n \|_{\funclass^{H_a}} \leq \const F^{H_a} J(1,\funclass^{H_a},F^{H_a})$.
    The bounds for $\funclass_{\transform_a^2}$ and $\transform_a(\tilde{Q}_a)^2$ can be shown similarly. 
\end{proof}

The next lemma shows rough bounds on the expectation, the variance, and the tail of $\tilde{\psi}_a$.
When $P_*$ and the estimation procedure are kept constant, and only $n$ varies, the rate should be tight (if the adopted empirical process bounds are tight), although the constants might be improved.
\begin{lemma} \label{lemma: bound Var tail psi}
    Under Conditions~\ref{cond: 2nd moment}, \ref{cond: positivity}, \ref{cond: bounded Q} and \ref{cond: donsker}, $|\expect[\tilde{\psi}_a] - \psi_{*,a}| \leq \const n^{-1/2} J(1,\funclass^{H_a},F^{H_a}) F^{J_a}$,
    $$\Var(\tilde{\psi}_a) \leq \const n^{-1} \left\{ \sigma_{\#,a}^2 + (F^{H_a})^2 J^2(2,\funclass,M) \right\},$$
    and, for any positive constants $\eta_1,\eta_3,\eta_5, \alpha_2,t_2 \geq 1,q_2 \geq 2,\alpha_4,t_4 \geq 1,q_4 \geq 2$, with
    \begin{align*}
        \eta_2 &:= \frac{\const}{n} \left[ (1+\alpha_2) J(2,\funclass,M) F^{H_a} + q_2 \left\{ F^{H_a} \sqrt{t_2} + \frac{F^{H_a} t_2}{\alpha_2 \sqrt{n}} \right\} \right]^2, \\
        \eta_4 &:= 2 |\psi_{*,a}| n^{-1/2} \left[ \const (1+\alpha_4) J(2,\funclass,M) F^{H_a} + \const q_4 \left\{ F^{H_a} \sqrt{t_4} + \frac{F^{H_a} t_4}{\alpha_4 \sqrt{n}} \right\} \right] \\
        &\qquad+ \const |\psi_{*,a}| n^{-1/2} F^{H_a} J(2, \funclass, M) + \const n^{-1} \left\{ \sigma_{\#,a}^2 + (F^{H_a})^2 J^2(2,\funclass,M) \right\}
    \end{align*}
    for some sufficiently large absolute constant $\const$, it holds that
    \begin{align*}
        &\Pr \left( \left| \tilde{\psi}_a^2 - \expect[\tilde{\psi}_a^2] \right| > \sum_{j=1}^5 \eta_j \right) \\
        &\leq \const \frac{\rho_{\#,a}}{\sigma_{\#,a}^3 \sqrt{n}} + 2 \exp \left\{ - \frac{\eta_1 n}{2 \sigma_{\#,a}^2} \right\} + 2 \ind(\psi_{*,a} \neq 0) \exp \left\{- \frac{\eta_3^2 n}{8 \psi_{*,a}^2 \sigma_{\#,a}^2} \right\} + 2 \exp \left\{ - \frac{\eta_5^2 n^2}{16 (F^{H_a})^2 \sigma_{\#,a}^2} \right\} \\
        &\quad+ t_2^{-q_2/2} + \ind(\psi_{*,a} \neq 0) t_4^{-q_4/2}.
    \end{align*}
\end{lemma}
\begin{proof}
    Since $\tilde{\psi}_a = P_n \transform_a(\tilde{Q}_a)$ and $P_* \transform_a(Q_{\#,a}) = P_* \transform(\tilde{Q}_a) = \psi_{*,a}$, it holds that $\tilde{\psi}_a = \psi_{*,a} + (P_n-P_*) \transform_a(Q_{\#,a}) + n^{-1/2} \empro_n H_a(\tilde{Q}_a)$. Thus, $\expect[\tilde{\psi}_a] = \psi_{*,a} + n^{-1/2} \expect[\empro_n H_a(\tilde{Q}_a)]$, and so $|\expect[\tilde{\psi}_a] - \psi_{*,a}|$ can be bounded using Lemma~\ref{lemma: rough empro tail bound}.
    
    By Theorem~2.14.1 in \protect\citetsupp{VanderVaart2023}, $\expect [\{ \empro_n H_a(\tilde{Q}_a) \}^2] \leq \const J(1,\funclass^{H_a},F^{H_a})^2 (F^{H_a})^2$. Thus, 
    \begin{align*}
        \Var(\tilde{\psi}_a) \leq 2 \Var((P_n-P_*) \transform_a(Q_{\#,a})) + \frac{2}{n} \Var(\empro_n H_a(\tilde{Q}_a)) \leq \frac{2}{n} \left\{ \sigma_{\#,a}^2 + \const J(1,\funclass^{H_a},F^{H_a})^2 (F^{H_a})^2 \right\}.
    \end{align*}
    
    Note that
    \begin{align*}
        \tilde{\psi}_a^2 &= \psi_{*,a}^2 + \{(P_n-P_*) \transform_a(Q_{\#,a})\}^2 + n^{-1} \{\empro_n H_a(\tilde{Q}_a)\}^2 + 2 \psi_{*,a} (P_n-P_*) \transform_a(Q_{\#,a}) + 2 \psi_{*,a} n^{-1/2} \empro_n H_a(\tilde{Q}_a) \\
        &\quad+ 2 n^{-1/2} (P_n-P_*) \transform_a(Q_{\#,a}) \cdot \empro_n H_a(\tilde{Q}_a), \\
        \expect[\tilde{\psi}_a^2] &= (\expect[\tilde{\psi}_a])^2 + \Var(\tilde{\psi}_a) = \psi_{*,a}^2 + n^{-1} \{\expect[\empro_n H_a(\tilde{Q}_a)]\}^2 + 2 \psi_{*,a} n^{-1/2} \expect[\empro_n H_a(\tilde{Q}_a)] + \Var(\tilde{\psi}_a).
    \end{align*}
    By the union bound, 
    \begin{align*}
        &\Pr \left( \left| \tilde{\psi}_a^2 - \expect[\tilde{\psi}_a^2] \right| > \sum_{j=1}^5 \eta_j \right) \\
        &= \Pr \Bigg( \Bigg| \{(P_n-P_*) \transform_a(Q_{\#,a})\}^2 + n^{-1} \{\empro_n H_a(\tilde{Q}_a)\}^2 \\
        &\qquad+ 2 \psi_{*,a} (P_n-P_*) \transform_a(Q_{\#,a}) + 2 \psi_{*,a} n^{-1/2} \empro_n H_a(\tilde{Q}_a) + 2 n^{-1/2} (P_n-P_*) \transform_a(Q_{\#,a}) \cdot \empro_n H_a(\tilde{Q}_a) \\
        &\qquad- \left\{ n^{-1} \{\expect[\empro_n H_a(\tilde{Q}_a)]\}^2 + 2 \psi_{*,a} n^{-1/2} \expect[\empro_n H_a(\tilde{Q}_a)] + \Var(\tilde{\psi}_a) \right\} \Bigg| > \sum_{j=1}^5 \eta_j \Bigg) \\
        &\leq \Pr \left( \{(P_n-P_*) \transform_a(Q_{\#,a})\}^2 > \eta_1 \right) + \Pr \left( |n^{-1} \{\empro_n H_a(\tilde{Q}_a)\}^2| > \eta_2 \right) + \Pr \left( |2 \psi_{*,a} (P_n-P_*) \transform_a(Q_{\#,a})| > \eta_3 \right) \\
        &\quad+ \Pr \left( | 2 \psi_{*,a} n^{-1/2} \empro_n H_a(\tilde{Q}_a) - 2 \psi_{*,a} n^{-1/2} \expect[\empro_n H_a(\tilde{Q}_a)] - \left\{ n^{-1} \{\expect[\empro_n H_a(\tilde{Q}_a)]\}^2 + \Var(\tilde{\psi}_a) \right\} | > \eta_4 \right) \\
        &\quad+ \Pr \left( | 2 n^{-1/2} (P_n-P_*) \transform_a(Q_{\#,a}) \cdot \empro_n H_a(\tilde{Q}_a) | > \eta_5 \right).
    \end{align*}
    I next bound the five terms on the right-hand side separately.

    \noindent\textbf{Term~1}: Term~1 equals $\Pr \left( | \empro_n \transform_a(Q_{\#,a})|/\sigma_{\#,a} > \sqrt{\eta_1 n}/\sigma_{\#,a} \right)$.
    By Berry-Esseen bound and Gaussian tail bound, Term~1 is upper bounded by
    $$\const \frac{\rho_{\#,a}}{\sigma_{\#,a}^3 \sqrt{n}} + 2 \exp \left\{ - \frac{\eta_1 n}{2 \sigma_{\#,a}^2} \right\}.$$

    \noindent\textbf{Term~2}: Term~2 equals $\Pr(|\empro_n H_a(\tilde{Q}_a)| > \sqrt{\eta_2 n})$. With the chosen $\eta_2$, by Lemma~\ref{lemma: rough empro tail bound}, Term~2 is upper bounded by $t_2^{-q_2/2}$.

    \noindent\textbf{Term~3}: If $\psi_{*,a}=0$, then Term~3 equals 0. Otherwise, Term~3 equals $\Pr(|\empro_n \transform_a(Q_{\#,a})| > \eta_3 \sqrt{n}/(2 |\psi_{*,a}|))$. By Berry-Esseen bound and Gaussian tail bound, Term~3 is upper bounded by
    $$\const \frac{\rho_{\#,a}}{\sigma_{\#,a}^3 \sqrt{n}} + 2 \exp \left\{- \frac{\eta_3^2 n}{8 \psi_{*,a}^2 \sigma_{\#,a}^2} \right\}.$$

    \noindent\textbf{Term~4}: For sufficiently large $\const$, the chosen $\eta_4$ satisfies
    \begin{align*}
        \eta_4 &> 2 |\psi_{*,a}| n^{-1/2} \left[ \const (1+\alpha_4) J(2,\funclass,M) F^{H_a} + \const q_4 \left\{ F^{H_a} \sqrt{t_4} + \frac{F^{H_a} t_4}{\alpha_4 \sqrt{n}} \right\} \right] \\
        &\qquad+ \left| 2 \psi_{*,a} n^{-1/2} \expect[\empro_n H_a(\tilde{Q}_a)] + n^{-1} \{\expect[\empro_n H_a(\tilde{Q}_a)]\}^2 + \Var(\tilde{\psi}_a) \right| \\
        &\geq \left| 2 \psi_{*,a} n^{-1/2} \expect[\empro_n H_a(\tilde{Q}_a)] + n^{-1} \{\expect[\empro_n H_a(\tilde{Q}_a)]\}^2 + \Var(\tilde{\psi}_a) \right|.
    \end{align*}
    If $\psi_{*,a}=0$, then Term~4 equals 0.
    Otherwise, Term~4 is upper bounded by
    \begin{align*}
        &\Pr \left( |2 \psi_{*,a} n^{-1/2} \empro_n H_a(\tilde{Q}_a) | > \eta_4 - \left| 2 \psi_{*,a} n^{-1/2} \expect[\empro_n H_a(\tilde{Q}_a)] + n^{-1} \{\expect[\empro_n H_a(\tilde{Q}_a)]\}^2 + \Var(\tilde{\psi}_a) \right| \right) \\
        &= \Pr \left( |\empro_n H_a(\tilde{Q}_a)| > \sqrt{n} \frac{\eta_4 - \left| 2 \psi_{*,a} n^{-1/2} \expect[\empro_n H_a(\tilde{Q}_a)] + n^{-1} \{\expect[\empro_n H_a(\tilde{Q}_a)]\}^2 + \Var(\tilde{\psi}_a) \right|}{2 |\psi_{*,a}|} \right) \\
        &\leq \Pr \left( |\empro_n H_a(\tilde{Q}_a)| > \const (1+\alpha_4) J(2,\funclass,M) F^{H_a} + \const q_4 \left\{ F^{H_a} \sqrt{t_4} + \frac{F^{H_a} t_4}{\alpha_4 \sqrt{n}} \right\} \right)
    \end{align*}
    By Lemma~\ref{lemma: rough empro tail bound}, Term~4 is upper bounded by $t_4^{-q_4/2}$.

    \noindent\textbf{Term~5}: Since $|\empro_n H_a(\tilde{Q}_a)| \leq 2 F^{H_a}$, Term~5 is upper bounded by $\Pr(|\empro_n \transform_a(Q_{\#,a})|/\sigma_{\#,a} > \eta_5 n/(4 F^{H_a} \sigma_{\#,a}))$, which, by Berry-Esseen bound and Gaussian tail bound, is further upper bounded by
    $$\const \frac{\rho_{\#,a}}{\sigma_{\#,a}^3 \sqrt{n}} + 2 \exp \left\{ - \frac{\eta_5^2 n^2}{32 (F^{H_a})^2 \sigma_{\#,a}^2} \right\}.$$

    Combining the bounds for these five terms yields the desired tail bound for $\tilde{\psi}_a^2$.
\end{proof}

\begin{lemma} \label{lemma: bound sigma2 tail}
    Suppose that Conditions~\ref{cond: 2nd moment}--\ref{cond: higher moments} and \ref{cond: donsker} hold.
    Recall the constant $\varrho_{\#,a}$ defined above Theorem~\ref{thm: BE bound}.
    For any positive constants $\delta, \alpha_6, q_6 \geq 2, t_6 \geq 1, \eta_7, \eta_8$, with
    $$\eta_6 := \const n^{-1/2} (1+\alpha_6) \left\{ F^{\transform_a^2} J(2 \delta, \funclass, M) + \frac{F^{\transform_a^2} J^2(2 \delta,\funclass,M)}{\delta^2 \sqrt{n}} \right\} + \const n^{-1/2} q_6 \left\{ \left( \delta F^{\transform_a^2} + \frac{F^{\transform_a^2}}{\sqrt{n}} \right) \sqrt{t_6} + \frac{F^{\transform_a^2} t_6}{\alpha_6 \sqrt{n}} \right\}$$
    for sufficiently large $\const$ and the constants in Lemma~\ref{lemma: bound Var tail psi},
    \begin{align*}
        &\Pr \left( | \tilde{\sigma}_a^2 - \expect[\tilde{\sigma}_a^2]| > \sum_{j=1}^8\eta_j, \| \transform_a(\tilde{Q}_a)^2- \transform_a(Q_{\#,a})^2 \|_2 \leq \delta F^{\transform_a^2} \right) \\
        &\leq \const \frac{\rho_{\#,a}}{\sigma_{\#,a}^3 \sqrt{n}} + 2 \exp \left\{ - \frac{\eta_1 n}{2 \sigma_{\#,a}^2} \right\} + 2 \ind(\psi_{*,a} \neq 0) \exp \left\{- \frac{\eta_3^2 n}{8 \psi_{*,a}^2 \sigma_{\#,a}^2} \right\} + 2 \exp \left\{ - \frac{\eta_5^2 n^2}{32 (F^{H_a})^2 \sigma_{\#,a}^2} \right\} \\
        &\quad+ \const \frac{\varrho_{\#,a}}{\varsigma_{\#,a}^3 \sqrt{n}} + 2 \exp \left\{ -\frac{\eta_7^2 n}{2 \varsigma_{\#,a}^2} \right\} + t_2^{-q_2/2} + \ind(\psi_{*,a} \neq 0) t_4^{-q_4/2} + t_6^{-q_6/2} + \frac{16M^2 (1-\tau_\pi)^2 \expect \| \tilde{Q}_a-Q_{\#,a} \|_{2}^2}{\tau_\pi^2 \eta_8^2}.
    \end{align*}
    Additionally under Condition~\ref{cond: more tail}, the last term on the right-hand side can be replaced by
    $$\tilde{a}_2 n^{\tilde{d}_2} \exp \left\{ -\tilde{c}_2 \left( \frac{\eta_8}{\tilde{r}_2(n)} \right)^{\tilde{q}_2} \right\} + \tilde{b}_2 n^{-1/2} \log n.$$
\end{lemma}
\begin{proof}
    Note that
    \begin{align}
        \tilde{\sigma}_a^2 &= P_n \transform_a(\tilde{Q}_a)^2 - \tilde{\psi}_a^2 \nonumber \\
        &= n^{-1/2} \empro_n \{ \transform_a(\tilde{Q}_a)^2 - \transform_a(Q_{\#,a})^2 \} + (P_n-P_*) \transform_a(Q_{\#,a})^2 + P_* \transform_a(\tilde{Q}_a)^2 - \tilde{\psi}_a^2, \nonumber \\
        \expect[\tilde{\sigma}_a^2] &= n^{-1/2} \expect[ \empro_n \{ \transform_a(\tilde{Q}_a)^2 - \transform_a(Q_{\#,a})^2 \} ] + \expect[ P_* \transform_a(\tilde{Q}_a)^2] - \expect[\tilde{\psi}_a^2]. \label{eq: expect Var}
    \end{align}
    Let $E$ denote the event $\| \transform_a(\tilde{Q}_a)^2- \transform_a(Q_{\#,a})^2 \|_2 \leq \delta F^{\transform_a^2}$.
    It holds that
    \begin{align*}
        &\Pr \left( | \tilde{\sigma}_a^2 - \expect[\tilde{\sigma}_a^2]| > \sum_{j=1}^8\eta_j, E \right) \\
        &\leq \Pr \left( n^{-1/2} \left| \empro \{ \transform_a(\tilde{Q}_a)^2 - \transform_a(Q_{\#,a})^2 \} - \expect[\empro \{ \transform_a(\tilde{Q}_a)^2 - \transform_a(Q_{\#,a})^2 \}] \right| > \eta_6, E \right) \\
        &\quad+ \Pr( |(P_n-P_*) \transform_a(Q_{\#,a})^2| > \eta_7 ) + \Pr(|P_* \transform_a(\tilde{Q}_a)^2 - \expect[ P_* \transform_a(\tilde{Q}_a)^2]| > \eta_8) \\
        &\quad+ \Pr \left( \left| \tilde{\psi}_a^2 - \expect[\tilde{\psi}_a^2] \right| > \sum_{j=1}^5 \eta_j \right).
    \end{align*}

    For sufficiently large $\const$, when $E$ occurs, by Lemma~\ref{lemma: rough empro tail bound},
    \begin{align*}
        \sqrt{n} \eta_6 &\geq |\expect[\empro \{ \transform_a(\tilde{Q}_a)^2 - \transform_a(Q_{\#,a})^2 \}]| + \const (1+\alpha) \expect \| \empro_n \|_{\funclass^{\transform_a^2}_{\delta F^{\transform_a^2}}} + \const q \left\{ \left( \delta F^{\transform_a^2} + \frac{F^{\transform_a^2}}{\sqrt{n}} \right) \sqrt{t} + \frac{F^{\transform_a^2} t}{\alpha \sqrt{n}} \right\} \\
        &> |\expect[\empro \{ \transform_a(\tilde{Q}_a)^2 - \transform_a(Q_{\#,a})^2 \}]|,
    \end{align*}
    and thus, by Lemma~\ref{lemma: rough empro tail bound} again,
    $$\Pr \left( n^{-1/2} \left| \empro \{ \transform_a(\tilde{Q}_a)^2 - \transform_a(Q_{\#,a})^2 \} - \expect[\empro \{ \transform_a(\tilde{Q}_a)^2 - \transform_a(Q_{\#,a})^2 \}] \right| > \eta_6, E \right) \leq t_6^{-q_6/2}.$$

    By Conditions~\ref{cond: 3rd moment} and \ref{cond: bounded Q}, $\varrho_{\#,a} < \infty$. Thus, by Berry-Esseen bound, Gaussian tail bound, and Condition~\ref{cond: higher moments},
    \begin{align*}
        \Pr( |(P_n-P_*) \transform_a(Q_{\#,a})^2| > \eta_7 ) &= \Pr \left( \sqrt{n} |(P_n-P_*) \transform_a(Q_{\#,a})^2|/\varsigma_{\#,a} > \eta_7 \sqrt{n}/\varsigma_{\#,a} \right) \\
        &\leq \const \frac{\varrho_{\#,a}}{\varsigma_{\#,a}^3 \sqrt{n}} + 2 \exp \left\{ -\frac{\eta_7^2 n}{2 \varsigma_{\#,a}^2} \right\}.
    \end{align*}

    By Lemma~\ref{lemma: transform sqaure} and Chebyshev inequality,
    $$\Pr(|P_* \transform_a(\tilde{Q}_a)^2 - \expect[ P_* \transform_a(\tilde{Q}_a)^2]| > \eta_8) \leq \frac{16M^2 (1-\tau_\pi)^2 \expect \| \tilde{Q}_a-Q_{\#,a} \|_{2}^2}{\tau_\pi^2 \eta_8^2}.$$
    Additionally under Condition~\ref{cond: more tail},
    \begin{align*}
        &\Pr(|P_* \transform_a(\tilde{Q}_a)^2 - \expect[ P_* \transform_a(\tilde{Q}_a)^2]| > \eta_8) \leq \tilde{a}_2 n^{\tilde{d}_2} \exp \left\{ -\tilde{c}_2 \left( \frac{\eta_8}{\tilde{r}_2(n)} \right)^{\tilde{q}_2} \right\} + \tilde{b}_2 n^{-1/2} \log n.
    \end{align*}
    
    Finally, $\Pr \left( \left| \tilde{\psi}_a^2 - \expect[\tilde{\psi}_a^2] \right| > \sum_{j=1}^5 \eta_j \right)$ can be bounded as in Lemma~\ref{lemma: bound Var tail psi}.
\end{proof}

\begin{proof}[Proof of Theorem~\ref{thm: expect Var}]
    Note that $\sigma_{\#,a}^2 = P_* \transform_a(Q_{\#,a})^2 - \psi_{*,a}^2$. By \eqref{eq: expect Var},
    $$\sigma_{\dagger,a}^2 - \sigma_{\#,a}^2 = n^{-1/2} \expect[ \empro_n \{ \transform_a(\tilde{Q}_a)^2 - \transform_a(Q_{\#,a})^2 \} ] + P_* \left\{ \expect_{\tilde{Q}_a}[\transform_a(\tilde{Q}_a)^2] - \transform_a(Q_{\#,a})^2 \right\} - \Var(\tilde{\psi}_a).$$
    By Lemma~\ref{lemma: transform sqaure},
    \begin{align*}
        &P_* \left\{ \expect_{\tilde{Q}_a}[\transform_a(\tilde{Q}_a)^2] - \transform_a(Q_{\#,a})^2 \right\} \\
        &= P_* \left\{ \frac{1-\pi_*(a \mid \cdot)}{\pi_*(a \mid \cdot)} \expect_{\tilde{Q}_a} [(\tilde{Q}_a-Q_{\#,a})^2] \right\} + 2 P_* \left\{ \frac{1-\pi_*(a \mid \cdot)}{\pi_*(a \mid \cdot)} (Q_{\#,a}-Q_{*,a}) \expect_{\tilde{Q}_a}[\tilde{Q}_a - Q_{\#,a}] \right\}.
    \end{align*}
       
    By Lemmas~\ref{lemma: covering number} and \ref{lemma: rough empro tail bound},
    \begin{align}
        &|\expect[ \empro_n \{ \transform_a(\tilde{Q}_a)^2 - \transform_a(Q_{\#,a})^2 \} ]| \nonumber \\
        &= \expect[ \empro_n \{ \transform_a(\tilde{Q}_a)^2 - \transform_a(Q_{\#,a})^2 \} \ind(\| \tilde{Q}_a - Q_{\#,a} \|_2 \leq 2 \delta M) ] \nonumber \\
        &\qquad+ \expect[ \empro_n \{ \transform_a(\tilde{Q}_a)^2 - \transform_a(Q_{\#,a})^2 \} \ind(\|\tilde{Q}_a - Q_{\#,a} \|_2 > 2 \delta M) ] \nonumber \\
        &\leq \expect \| \empro_n \|_{\funclass^{\transform_a^2}_{\delta F^{\transform_a^2}}} + \expect \| \empro_n \|_{\funclass^{\transform_a^2}} \Pr(\|\tilde{Q}_a - Q_{\#,a} \|_2 > 2 \delta M) \nonumber \\
        &\leq \const F^{\transform_a^2} J(2 \delta, \funclass, M) + \const \frac{F^{\transform_a^2} J^2(2 \delta,\funclass,M)}{\delta^2 \sqrt{n}} + \const F^{\transform_a^2} J(2, \funclass, M) \Pr(\|\tilde{Q}_a - Q_{\#,a} \|_2 > 2 \delta M), \label{eq: Var bias empirical process}
    \end{align}
    where $\Pr(\|\tilde{Q}_a - Q_{\#,a} \|_2 > 2 \delta M)$ can be bounded by Markov's inequality, or with a sub-Weibull tail bound under Condition~\ref{cond: tail}.
    Thus, \eqref{eq: Var bias} holds.
    Eq.~\ref{eq: ATE Var bias} can be shown similarly by noting that
    \begin{align*}
        &P_* \left\{ \expect_{\tilde{Q}_1,\tilde{Q}_0}[ \{ \transform_1(\tilde{Q}_1) - \transform_0(\tilde{Q}_0) \}^2] - \{ \transform_1(Q_{\#,1}) - \transform_0(Q_{\#,0}) \}^2 \right\} \\
        &= \sum_{a=0}^1 P_* \left\{ \frac{1-\pi_*(a \mid \cdot)}{\pi_*(a \mid \cdot)} \expect_{\tilde{Q}_a} [(\tilde{Q}_a-Q_{\#,a})^2] \right\} + 2 \sum_{a=0}^1 P_* \left\{ \frac{1-\pi_*(a \mid \cdot)}{\pi_*(a \mid \cdot)} (Q_{\#,a}-Q_{*,a}) \expect_{\tilde{Q}_a}[\tilde{Q}_a - Q_{\#,a}] \right\} \\
        &\quad-2 P_* \left\{ \expect_{\tilde{Q}_1,\tilde{Q}_0}[\transform_1(\tilde{Q}_1) \transform_0(\tilde{Q}_0)] - \transform_1(Q_{\#,1}) \transform_0(Q_{\#,0}) \right\} \\
        &= \expect \left[ P_* \left\{ f (\tilde{Q}_1-Q_{\#,1}) + \frac{1}{f} (\tilde{Q}_0-Q_{\#,0}) \right\}^2 \right] + 2 \sum_{a=0}^1 P_* \left\{ \frac{1}{\pi_*(a \mid \cdot)} (Q_{\#,a}-Q_{*,a}) \expect_{\tilde{Q}_a}[\tilde{Q}_a - Q_{\#,a}] \right\}.
    \end{align*}
    By the triangle inequality and Lemma~\ref{lemma: bound Var tail psi},
    \begin{align*}
        &|\sigma_{\dagger,a}^2 - \sigma_{\#,a}^2| \\
        &\leq \frac{\const}{\sqrt{n}} F^{\transform_a^2} J(2 \delta, \funclass, M) + \const \frac{F^{\transform_a^2} J^2(2 \delta,\funclass,M)}{\delta^2 n} + \frac{\const}{\sqrt{n}} F^{\transform_a^2} J(2, \funclass, M) \Pr(\|\tilde{Q}_a - Q_{\#,a} \|_2 > 2 \delta M) \\
        &\quad+ P_* \left\{ \frac{1-\pi_*(a \mid \cdot)}{\pi_*(a \mid \cdot)} \expect_{\tilde{Q}_a} [(\tilde{Q}_a-Q_{\#,a})^2] \right\} + 2 \left| P_* \left\{ \frac{1-\pi_*(a \mid \cdot)}{\pi_*(a \mid \cdot)} (Q_{\#,a}-Q_{*,a}) \expect_{\tilde{Q}_a}[\tilde{Q}_a - Q_{\#,a}] \right\} \right| \\
        &\quad+ \frac{\const}{n} \left\{ \sigma_{\#,a}^2 + (F^{H_a})^2 J^2(2,\funclass,M) \right\}.
    \end{align*}
    Apply Cauchy-Schwarz inequality, Markov's inequality, and the inequality $|\sigma_{\dagger,a} - \sigma_{\#,a}| \leq |\sigma_{\dagger,a}^2 - \sigma_{\#,a}^2|/\sigma_{\#,a}$ to obtain the desired bound.
\end{proof}

\begin{proof}[Proof of Theorem~\ref{thm: BE bound}]
    For any positive constants $\delta, \tilde{\alpha},\tilde{q} \geq 2, \tilde{t} \geq 1$, let
    \begin{align*}
        \tilde{\eta} &:= \const (1+\tilde{\alpha}) \left\{ J(2 \delta, \funclass, M) F^{H_a} + \frac{F^{H_a} J^2(2\delta,\funclass,M)}{\delta^2 \sqrt{n}} \right\} + \const \tilde{q} \left\{ \left( \delta F^{H_a} + \frac{F^{H_a}}{\sqrt{n}} \right) \sqrt{\tilde{t}} + \frac{F^{H_a} \tilde{t}}{\tilde{\alpha} \sqrt{n}} \right\},
    \end{align*}
    for some sufficiently large constant $\const$.
    For any positive constants $\tilde{\eta}_1,\ldots,\tilde{\eta}_8$, taking $\eta = \tilde{\eta}+\sum_{j=1}^8 \tilde{\eta}_j$ in Lemmas~\ref{lemma: AIPW delta} and \ref{lemma: normal quantile}, and considering the two cases of whether $\| \tilde{Q}_a-Q_{\#,a} \|_2 \leq 2 \delta M$ separately,
    \begin{align*}
        &\left| \Pr(\sqrt{n} (\tilde{\psi}_a - \psi_{*,a}) \leq z_\alpha \tilde{\sigma}_a) - (1-\alpha) - \phi(z_\alpha) z_\alpha \frac{\sigma_{\dagger,a}-\sigma_{\#,a}}{\sigma_{\#,a}} \right| \\
        &\leq \const \frac{\rho_{\#,a}}{\sigma_{\#,a}^3 \sqrt{n}} + \const \frac{(\sigma_{\dagger,a}-\sigma_{\#,a})^2}{\sigma_{\#,a}^2} + \const \frac{\tilde{\eta}+\sum_{j=1}^8 \tilde{\eta}_j}{\sigma_{\#,a}} \\
        &\quad+ \Pr(|\sqrt{n} \{ \tilde{\psi}_a - \psi_{*,a} - P_n D_a(Q_{\#,a},\psi_{*,a}) \}| > \tilde{\eta}, \| \tilde{Q}_a-Q_{\#,a} \|_2 \leq 2 \delta M) \\
        &\quad+ \Pr \left( | z_\alpha (\tilde{\sigma}_a - \sigma_{\dagger,a}) | > \sum_{j=1}^8 \tilde{\eta}_j, \| \tilde{Q}_a-Q_{\#,a} \|_2 \leq 2 \delta M \right) + \Pr(\| \tilde{Q}_a-Q_{\#,a} \|_2 > 2 \delta M).
    \end{align*}
    Note that $\| \tilde{Q}_a-Q_{\#,a} \|_2 \leq 2 \delta M$ implies $\| H_a(\tilde{Q}_a) \|_2 = \| H_a(\tilde{Q}_a) - H_a(Q_{\#,a}) \|_2 \leq \delta F^{H_a}$ and $\| \transform_a(\tilde{Q}_a)^2- \transform_a(Q_{\#,a})^2 \|_2 \leq \delta F^{\transform_a^2}$ by Lemma~\ref{lemma: covering number}.
    Thus, by Lemma~\ref{lemma: rough empro tail bound},
    \begin{align*}
        &\Pr(|\sqrt{n} \{ \tilde{\psi}_a - \psi_{*,a} - P_n D_a(Q_{\#,a},\psi_{*,a}) \}| > \tilde{\eta}, \| \tilde{Q}_a-Q_{\#,a} \|_2 \leq 2 \delta M) \\
        &\leq \Pr(|\empro_n H_a(\tilde{Q}_a)| > \tilde{\eta}, \| H_a(\tilde{Q}_a) \|_2 \leq \delta F^{H_a}) \leq \tilde{t}^{-\tilde{q}/2}.
    \end{align*}

    Next bound $\Pr \left( | z_\alpha (\tilde{\sigma}_a - \sigma_{\dagger,a}) | > \sum_{j=1}^8 \tilde{\eta}_j, \| \tilde{Q}_a-Q_{\#,a} \|_2 \leq 2 \delta M \right)$. This term is zero if $z_\alpha=0$ and $\sum_{j=1}^8 \tilde{\eta}_j > 0$.
    Otherwise, since $\tilde{\sigma}_a \geq 0$ and $\sigma_{\dagger,a} > 0$, it holds that
    \begin{align*}
        &\Pr \left( | z_\alpha (\tilde{\sigma}_a - \sigma_{\dagger,a}) | > \sum_{j=1}^8 \tilde{\eta}_j, \| \tilde{Q}_a-Q_{\#,a} \|_2 \leq 2 \delta M \right) \\
        &\leq \Pr \left( | \tilde{\sigma}_a^2 - \sigma_{\dagger,a}^2 | > \sigma_{\dagger,a} \sum_{j=1}^8 \tilde{\eta}_j/|z_\alpha|, \| \transform_a(\tilde{Q}_a)^2- \transform_a(Q_{\#,a})^2 \|_2 \leq \delta F^{\transform_a^2} \right).
    \end{align*}
    For each $j=1,\ldots,8$, take $\tilde{\eta}_j$ to be $\eta_j |z_\alpha| / \sigma_{\dagger,a}$ with $\eta_j$ defined in Lemma~\ref{lemma: bound Var tail psi} or \ref{lemma: bound sigma2 tail}. This term can be bounded as in Lemma~\ref{lemma: bound sigma2 tail}.
    In conclusion, I will take $\tilde{\eta}_j$ to be $\eta_j |z_\alpha| / \sigma_{\dagger,a} + \epsilon/7$ for a constant $\epsilon > 0$ to handle both cases.

    By Markov's inequality,
    \begin{equation}
        \Pr(\| \tilde{Q}_a-Q_{\#,a} \|_2 > 2 \delta M) \leq \frac{\expect \| \tilde{Q}_a-Q_{\#,a} \|_2^2}{4 \delta^2 M^2}. \label{eq: L2 tail}
    \end{equation}
    Additionally under Condition~\ref{cond: tail},
    \begin{equation}
        \Pr(\| \tilde{Q}_a-Q_{\#,a} \|_2 > 2 \delta M) \leq \tilde{a}_1 n^{\tilde{d}_1} \exp \left\{ - \tilde{c}_1 \left( \frac{2 \delta M}{\tilde{r}_1(n)} \right)^{\tilde{q}_1} \right\} + \tilde{b}_1 n^{-1/2} \log n. \label{eq: L2 tail tail}
    \end{equation}

    Combining all the above bounds yields that, for any constant $\epsilon>0$,
    \begin{align}
        &\left| \Pr(\sqrt{n} (\tilde{\psi}_a - \psi_{*,a}) \leq z_\alpha \tilde{\sigma}_a) - (1-\alpha) - \phi(z_\alpha) z_\alpha \frac{\sigma_{\dagger,a}-\sigma_{\#,a}}{\sigma_{\#,a}} \right| \nonumber \\
        &\leq \const \frac{\rho_{\#,a}}{\sigma_{\#,a}^3 \sqrt{n}} + \const \frac{(\sigma_{\dagger,a}-\sigma_{\#,a})^2}{\sigma_{\#,a}^2} + \const \frac{\tilde{\eta}+ |z_\alpha| \sum_{j=1}^8 \eta_j/\sigma_{\dagger,a} + \ind(z_\alpha = 0) \epsilon}{\sigma_{\#,a}} + \tilde{t}^{-\tilde{q}/2} \nonumber \\
        &\quad + \ind(z_\alpha \neq 0) \Bigg[ 2 \exp \left\{ - \frac{\eta_1 n}{2 \sigma_{\#,a}^2} \right\} + 2 \ind(\psi_{*,a} \neq 0) \exp \left\{- \frac{\eta_3^2 n}{8 \psi_{*,a}^2 \sigma_{\#,a}^2} \right\} + 2 \exp \left\{ - \frac{\eta_5^2 n^2}{32 (F^{H_a})^2 \sigma_{\#,a}^2} \right\} \nonumber \\
        &\quad+ \const \frac{\varrho_{\#,a}}{\varsigma_{\#,a}^3 \sqrt{n}} + 2 \exp \left\{ -\frac{\eta_7^2 n}{2 \varsigma_{\#,a}^2} \right\} + t_2^{-q_2/2} + \ind(\psi_{*,a} \neq 0) t_4^{-q_4/2} + t_6^{-q_6/2} + \frac{16M^2 (1-\tau_\pi)^2 \expect \| \tilde{Q}_a-Q_{\#,a} \|_{2}^2}{\tau_\pi^2 \eta_8^2} \Bigg] \nonumber \\
        &\quad+ \frac{\expect \| \tilde{Q}_a-Q_{\#,a} \|_2^2}{4 \delta^2 M^2}, \label{eq: rough BE bound}
    \end{align}
    where
    \begin{align*}
        &\frac{16M^2 (1-\tau_\pi)^2 \expect \| \tilde{Q}_a-Q_{\#,a} \|_{2}^2}{\tau_\pi^2 \eta_8^2} \text{ can be replaced by } \tilde{a}_2 n^{\tilde{d}_2} \exp \left\{ -\tilde{c}_2 \left( \frac{\eta_8}{\tilde{r}_2(n)} \right)^{\tilde{q}_2} \right\} + \tilde{b}_2 n^{-1/2} \log n
    \end{align*}
    under Condition~\ref{cond: more tail}, and
    $$\frac{\expect \| \tilde{Q}_a-Q_{\#,a} \|_2^2}{4 \delta^2 M^2} \text{ can be replaced by } \tilde{a}_1 n^{\tilde{d}_1} \exp \left\{ - \tilde{c}_1 \left( \frac{2 \delta M}{\tilde{r}_1(n)} \right)^{\tilde{q}_1} \right\} + \tilde{b}_1 n^{-1/2} \log n$$
    under Condition~\ref{cond: tail}.

    Let $q \geq 1$ be arbitrary. Taking
    \begin{align*}
        &\eta_1 = \frac{2 \sigma_{\#,a}^2 q \log n}{n}, \quad \eta_3 = \begin{cases}
            n^{-q} \text{ or any positive number} & (\text{if } z_\alpha =0) \\
            2 \sqrt{2} |\psi_{*,a}| \sigma_{\#,a} \sqrt{\frac{q \log n}{n}} & (\text{if } \psi_{*,a} \neq 0, z_\alpha \neq 0) \\
            \frac{\sigma_{\#,a} \sigma_{\dagger,a}}{|z_\alpha|} n^{-q} & (\text{if } \psi_{*,a} = 0, z_\alpha \neq 0)
        \end{cases}, \\
        &\eta_5 = \frac{4 \sqrt{2} F^{H_a} \sigma_{\#,a} \sqrt{q \log n}}{n}, \quad \eta_7 = \varsigma_{\#,a} \sqrt{\frac{2 q \log n}{n}}, \\
        &\eta_8 = \begin{cases}
            n^{-q} \text{ or any positive number} & (\text{if } z_\alpha = 0) \\
            \left\{ \frac{16 M^2 (1-\tau_\pi)^2 \sigma_{\#,a} \sigma_{\dagger,a}}{\tau_\pi^2 |z_\alpha|} \expect \| \tilde{Q}_a-Q_{\#,a} \|_{2}^2 \right\}^{1/3} & (\text{without Condition~\ref{cond: more tail} and if } z_\alpha \neq 0) \\
            \tilde{r}_2(n) \left( \frac{q+\tilde{d}_2}{\tilde{c}_2} \log n \right)^{1/\tilde{q}_2} & (\text{with Condition~\ref{cond: more tail} and if } z_\alpha \neq 0)
        \end{cases},
    \end{align*}
    $\epsilon = \sigma_{\#,a} n^{-q}$, $\tilde{\alpha} = \alpha_2 = \alpha_4 = \alpha_6 = 1$, $\tilde{q}=q_2=q_4=q_6= 2 q \log n \geq 2$ and $\tilde{t} = t_2 = t_4 = t_6 = \exp(1) > 1$, and using Theorem~\ref{thm: expect Var} to bound $(\sigma_{\dagger,a} - \sigma_{\#,a})^2$ (with the dummy $\delta$ in Theorem~\ref{thm: expect Var} replaced by $\delta'$), the desired result follows.
\end{proof}

\begin{proof}[Proof of Corollary~\ref{coro: VC & hull rates}]
    The key is to upper bound uniform entropy integrals $J(2 \delta, \funclass,M)$ based on bounds on the covering number. In this proof, to simplify notations, $C_1(\nu)$ denotes a constant depending on $\nu$ only that may vary line by line, and similarly for $C_2(\nu)$.
    First consider the case when $\funclass$ is VC with VC index $\nu$. For all $\delta > 0$ smaller than some absolute constant,
    \begin{align*}
        J(2 \delta,\funclass,M) &\leq \int_0^{2 \delta} \sqrt{C_1(\nu) + 2 \nu \log (1/\epsilon) } \intd \epsilon \\
        &\leq C_1(\nu) \int_0^{2 \delta} \sqrt{\log(1/\epsilon)} \intd \epsilon \\
        &= C_1(\nu) \int_\infty^{\sqrt{-\log(2\delta)}} x \intd(\exp(-x^2)) \qquad\qquad \left( x := \sqrt{\log(1/\epsilon)} \right) \\
        &= C_1(\nu) I
    \end{align*}
    where $I := -\int_{\sqrt{-\log(2\delta)}}^\infty x \intd(\exp(-x^2))$.
    By integration by parts and the assumption that $\delta$ is sufficiently small,
    \begin{align*}
        I &= -\left. \left[ x \exp(-x^2) \right] \right|_{x=\sqrt{-\log(2\delta)}}^\infty + \int_{\sqrt{-\log(2\delta)}}^\infty \exp(-x^2) \intd x \\
        &= 2\delta \sqrt{\log \left( \frac{1}{2\delta} \right)} - \int_{\sqrt{-\log(2\delta)}}^\infty \frac{x}{2x^2} \intd(\exp(-x^2)) \\
        &\leq 2\delta \sqrt{\log \left( \frac{1}{2\delta} \right)} - \frac{1}{2} \int_{\sqrt{-\log(2\delta)}}^\infty x \intd(\exp(-x^2)) \\
        &= 2\delta \sqrt{\log \left( \frac{1}{2\delta} \right)} + \frac{1}{2} I,
    \end{align*}
    so $I \leq 4 \delta \sqrt{- \log ( 2\delta)}$, and thus
    \begin{equation}
        J(2 \delta,\funclass,M) \leq C_1(\nu) \delta \sqrt{\log(1/\delta)}. \label{eq: VC UEI bound}
    \end{equation}
    By Theorem~\ref{thm: asymptotic BE bound}, for all sufficiently large $n$ and sufficiently small $\delta$, the left-hand side of \eqref{eq: asymptotic BE bound} is bounded by
    $$\bigO \left( \tilde{\Delta}(\delta) + \delta \sqrt{\log(1/\delta)} + \frac{\log(1/\delta)}{\sqrt{n}} + (\delta + n^{-1/2}) \log n \right).$$
    Taking $\delta = (\expect \| \tilde{Q}_a - Q_{\#,a} \|_{2}^2)^{1/3}$ yields the desired rate.
    Additionally under approximate sub-Weibull conditions \ref{cond: tail} and \ref{cond: more tail}, $\tilde{\Delta}(\delta)$ can be replaced by $\tilde{\Delta}'(\delta)$. Taking $\delta = \tilde{r}_1(n) [(1+\tilde{d}_1) \tilde{c}_1^{-1} \log n]^{1/\tilde{q}_1}/(2M)$, which tends to zero as $n \to \infty$, yields the other desired rate.

    Next, consider the case when $\funclass$ is a VC-hull such that the VC index of the associated VC class is $\nu$. For all $\delta>0$ smaller than some absolute constant,
    \begin{equation}
        J(2\delta,\funclass,M) \leq \int_0^{2 \delta} \sqrt{ C_2(\nu) \epsilon^{-2\nu/(2\nu+1)}} \intd \epsilon = C_2(\nu) \int_0^{2 \delta} \epsilon^{-\nu/(2\nu+1)} \intd \epsilon = C_2(\nu) \delta^{(\nu+1)/(2\nu+1)}. \label{eq: VC hull UEI bound}
    \end{equation}
    By Theorem~\ref{thm: asymptotic BE bound}, for all sufficiently large $n$ and sufficiently small $\delta$, the left-hand side of \eqref{eq: asymptotic BE bound} is bounded by
    $$\bigO \left( \tilde{\Delta}(\delta) + \delta^{(\nu+1)/(2\nu+1)} + \frac{1}{\delta^{2\nu/(2\nu+1)} \sqrt{n}} + (\delta + n^{-1/2}) \log n \right).$$
    Taking $\delta = (\expect \| \tilde{Q}_a - Q_{\#,a} \|_{2}^2)^{(2\nu+1)/(5\nu+3)} + n^{-(2\nu+1)/(6 \nu +2)}$ yields the desired rate.
    Additionally under Conditions~\ref{cond: tail} and \ref{cond: more tail}, taking $\delta = \tilde{r}_1(n) [(1+\tilde{d}_1) \tilde{c}_1^{-1} \log n]^{1/\tilde{q}_1}/(2M)$, which tends to zero as $n \to \infty$, yields the other desired rate.
\end{proof}

\begin{proof}[Proof of Proposition~\ref{prop: sign Var bias}]
    By Lemma~\ref{lemma: transform sqaure},
    \begin{align*}
        &\sigma_{\dagger,a}^2 - \sigma_{\#,a}^2 \\
        &= \expect[P_n \transform_a(\tilde{Q}_a)^2 - \tilde{\psi}_a^2] - P_* \transform_a(Q_{\#,a})^2 + \psi_{*,a}^2 \\
        &= \expect \left[ P_n \left\{ \transform_a(\tilde{Q}_a)^2 - \transform_a(Q_{\#,a})^2 \right\} \right] - \Var(\tilde{\psi}_a) \\
        &= \expect \left[ P_n \left\{ \frac{\ind(A=a)}{\pi_*(a \mid X)^2} \right\} \left\{ [Y-\tilde{Q}_a(X)]^2 - [Y-Q_{\#,a}(X)]^2 \right\} \right] \\
        &\quad+ 2 \expect \left[ P_n \frac{\ind(A=a)}{\pi_*(a \mid X)} \{Y - \tilde{Q}_a(X)\} \tilde{Q}_a(X) \right] - 2 P_* \frac{\ind(A=a)}{\pi_*(a \mid X)} \{Y-Q_{\#,a}(X)\} Q_{\#,a}(X) \\
        &\quad+ \expect[P_n (\tilde{Q}_a^2 - Q_{\#,a}^2)] - \Var(\tilde{\psi}_a).
    \end{align*}
    Proposition~\ref{prop: sign Var bias} follows immediately from Lemma~\ref{lemma: bound Var tail psi}. Both $\Theta_+(n^{-1})$ terms in \eqref{eq: asymptotic Var bias} and \eqref{eq: asymptotic Var bias 2} correspond to $\Var(\tilde{\psi}_a)$.

    Similarly,
    \begin{align*}
        &\sigma_{\dagger}^2 - \sigma_{\#}^2 \\
        &= \expect\left[ P_n \left\{ \transform_1(\tilde{Q}_1)^2 - \transform_1(Q_{\#,1})^2 \right\} + P_n \left\{ \transform_0(\tilde{Q}_0)^2 - \transform_0(Q_{\#,0})^2 \right\} - 2 P_n \left\{ \transform_1(\tilde{Q}_1) \transform_0(\tilde{Q}_0) - \transform_1(Q_{\#,1}) \transform_1(Q_{\#,0}) \right\} \right] \\
        &\quad- \Var(\tilde{\psi}) \\
        &= \expect \left[ P_n \left\{ \frac{\ind(A=1)}{\pi_*(1 \mid X)^2} \right\} \left\{ [Y-\tilde{Q}_1(X)]^2 - [Y-Q_{\#,1}(X)]^2 \right\} \right] \\
        &\quad+ \expect \left[ P_n \left\{ \frac{\ind(A=0)}{\pi_*(0 \mid X)^2} \right\} \left\{ [Y-\tilde{Q}_0(X)]^2 - [Y-Q_{\#,0}(X)]^2 \right\} \right] \\
        &\quad+ \expect[P_n (\tilde{Q}_1^2 - Q_{\#,1}^2)] + \expect[P_n (\tilde{Q}_0^2 - Q_{\#,0}^2)] - \Var(\tilde{\psi}).
    \end{align*}
\end{proof}

\subsection{Series regression}

I will use the notation in Section~\ref{sec: light tail example}.

\begin{proof}[Proof of Proposition~\ref{prop: series tail}]
    Since regressions with $T_K \Phi_K(X)$ and $\Phi_K(X)$ lead to the same estimators, without loss of generality, assume that $T_K$ is identity.
    By SeriesCond~\ref{seriescond: eigen}, $P_* \Phi_K(X) \Phi_K(X)^\top$ is invertible. 
    Let $Z := \Phi_K(X) \Phi_K(X)^\top - P_* \Phi_K(X) \Phi_K(X)^\top$ and $\mathcal{E}$ denote the event that $\lambda_{\max}((P_n-P_*) \Phi_K(X) \Phi_K(X)^\top) \leq \underline{\lambda}/2$.
    By SeriesCond~\ref{seriescond: l2} and \ref{seriescond: eigen}, $\lambda_{\max}(Z) \leq \lambda_{\max}(\Phi_K(X) \Phi_K(X)^\top) + \lambda_{\max}(P_* \Phi_K(X) \allowbreak \Phi_K(X)^\top) \leq \xi_K^2 + \bar{\lambda}$.
    Since $\expect[Z^2] = P_* \Phi_K(X) \Phi_K(X)^\top \Phi_K(X) \Phi_K(X)^\top - (P_* \Phi_K(X) \Phi_K(X)^\top)^2$ and $\lambda_{\max}(P_* \Phi_K(X) \Phi_K(X)^\top \Phi_K(X) \Phi_K(X)^\top) \leq \xi_K^2 \lambda_{\max} (P_* \Phi_K(X) \Phi_K(X)^\top) \leq \xi_K^2 \bar{\lambda}$, it holds that $\lambda_{\max}(Z^2) \leq \xi_K^2 \bar{\lambda}$.
    Applying matrix Bernstein inequality \citepsupp[e.g., Theorem~6.1.1 in][]{Tropp2015} to the sample mean of $Z$, namely $P_n Z = (P_n-P_*) \Phi_K(X) \Phi_K(X)^\top$, yields that
    $$\Pr(\mathcal{E}) \geq 1-2K \exp \left\{ - \frac{n \underline{\lambda}^2}{8 \xi_K^2 \bar{\lambda} + 4 (\xi_K^2 + \bar{\lambda}) \underline{\lambda}/3} \right\}.$$
    By Weyl's inequality, when $\mathcal{E}$ occurs, $\lambda_{\min}(P_n \Phi_K(X) \Phi_K(X)^\top) \geq \underline{\lambda} - \lambda_{\max}((P_n-P_*) \Phi_K(X) \Phi_K(X)^\top) \geq \underline{\lambda}/2 > 0$ and so $P_n \Phi_K(X) \Phi_K(X)^\top$ is invertible.
    Moreover, $\lambda_{\max}(P_n \Phi_K(X) \Phi_K(X)^\top) \leq \lambda_{\max}(P_* \Phi_K(X) \allowbreak \Phi_K(X)^\top) + \lambda_{\max}((P_n-P_*) \Phi_K(X) \Phi_K(X)^\top) \leq \bar{\lambda} + \underline{\lambda}/2$.

    Let $\epsilon := Y - Q_\#(X) = Y - \Phi_K(X)^\top \beta_\#$ and $W := \Phi_K(X) \epsilon$. 
    Since $Y$ is bounded and $Q_\#$ is bounded by some constant independent of $K$ when $\mathcal{E}$ occurs. Thus, when $\mathcal{E}$ occurs, $\epsilon$ is uniformly bounded in $K$.
    Suppose that $\mathcal{E}$ occurs, then
    $$\hat{\beta} - \beta_\# = \{P_n \Phi_K(X) \Phi_K(X)^\top\}^{-1} \{ P_n \Phi_K(X) Y - P_n \Phi_K(X) \Phi_K(X)^\top \beta_\#\} = \{P_n \Phi_K(X) \Phi_K(X)^\top\}^{-1} P_n W.$$
    and
    \begin{align*}
        \| \hat{Q} - Q_\# \|_2^2 &= P_* \{ \Phi_K(X)^\top (\hat{\beta}-\beta_\#) \}^2 \\
        &= (\hat{\beta} - \beta_\#)^\top \{ P_* \Phi_K(X) \Phi_K(X)^\top \} (\hat{\beta} - \beta_\#) \\
        &= \{P_n W^\top\} \{P_n \Phi_K(X) \Phi_K(X)^\top\}^{-1} \{ P_* \Phi_K(X) \Phi_K(X)^\top \} \{P_n \Phi_K(X) \Phi_K(X)^\top\}^{-1} \{P_n W\}.
    \end{align*}

    By construction, $P_* W=0$. Applying matrix Bernstein inequality to $P_n W$ yields that, there exists a constant $C_2$ independent of $(K,n)$ such that, for any $t>0$,
    $$\Pr( \| P_n W \|_\ltwo > t) \leq (K+1) \exp \left\{ - C_2 \frac{n t^2}{2 \xi_K^2 + 2 \xi_K t/3} \right\}.$$
    Thus, for any $t>0$,
    \begin{align*}
        &\Pr \left( \frac{\| \hat{Q} - Q_\# \|_2}{\sqrt{K/n}} > t \right) = \Pr \left( \frac{\| \hat{Q} - Q_\# \|_2^2}{K/n} > t^2 \right) \\
        &\leq \Pr \left( \frac{\| \hat{Q} - Q_\# \|_2^2}{K/n} > t^2,  \mathcal{E} \right) + 1 - \Pr(\mathcal{E}) \\
        &\leq \Pr \left( \frac{n}{K} \frac{\bar{\lambda}}{(\underline{\lambda}/2)^2} \| P_n W \|_{\ltwo}^2 > t^2 \right) + 2K \exp \left\{ - \frac{n \underline{\lambda}^2}{8 \xi_K^2 \bar{\lambda} + 4 (\xi_K^2 + \bar{\lambda}) \underline{\lambda}/3} \right\} \\
        &\leq (K+1) \exp \left\{ - A_1 \frac{\sqrt{n} K t^2}{A_2 \sqrt{n} \xi_K^2 + A_3 \xi_K \sqrt{K} t} \right\} + 2K \exp \left\{ - \frac{n \underline{\lambda}^2}{8 \xi_K^2 \bar{\lambda} + 4 (\xi_K^2 + \bar{\lambda}) \underline{\lambda}/3} \right\}
    \end{align*}
    where $A_1,A_2,A_3$ are constants independent of $(K,n)$. The desired result follows from SeriesCond~\ref{seriescond: growth rate}.
\end{proof}

It is possible to use concentration inequalities for vector $\ell_2$ norms \citepsupp[e.g.][]{Pinelis1994,Maurer2021}, particularly those centered at the expected norm, to obtain sharper tail bounds.

\bibliographystylesupp{apalike}
\bibliographysupp{ref}

\end{document}